\documentclass[a4paper]{amsart}
\usepackage{amsmath,amssymb,enumerate}
\usepackage{graphicx}
\usepackage{hyperref}

\newcommand{\TeXmacs}{T\kern-.1667em\lower.5ex\hbox{E}\kern-.125emX\kern-.1em\lower.5ex\hbox{\textsc{m\kern-.05ema\kern-.125emc\kern-.05ems}}}
\newcommand{\withTeXmacstext}{This document has been produced using \TeXmacs (see \url{http://www.texmacs.org})}
\newcommand{\tmem}[1]{{\em #1\/}}
\newcommand{\tmmathbf}[1]{\ensuremath{\boldsymbol{#1}}}
\newcommand{\tmop}[1]{\ensuremath{\operatorname{#1}}}
\newcommand{\tmtextit}[1]{{\itshape{#1}}}
\newcommand{\tmtextup}[1]{{\upshape{#1}}}
\newenvironment{enumerateroman}{\begin{enumerate}[i.] }{\end{enumerate}}
\renewenvironment{proof}{\noindent\textbf{Proof\ }}{\hspace*{\fill}$\Box$\medskip}

\theoremstyle{plain}
\newtheorem{theorem}{Theorem}[section]
\newtheorem{proposition}[theorem]{Proposition}
\newtheorem{corollary}[theorem]{Corollary}
\newtheorem{lemma}[theorem]{Lemma}

\theoremstyle{remark}

\newtheorem{definition}[theorem]{Definition}

\def\mathbbm{\mathbb}

\numberwithin{equation}{section}

\DeclareGraphicsExtensions{.eps}

\hypersetup{
pdfauthor = {Marc Wouts},
pdftitle={Surface tension in the dilute Ising model. The Wulff construction},
pdfsubject = {Surface tension in the dilute Ising model},
pdfkeywords = {Ising model, random media, surface tension, phase coexistence, large deviations, maximal flows},
pdfcreator = {TeXmacs and pdfLaTeX},
}

\title{Surface tension in the dilute Ising model. The Wulff
construction.}
\author{Marc Wouts}
\address{Modal'X, Universit\'e Paris Ouest - Nanterre La D\'efense. B\^at. G. \\
200 avenue de la R\'epublique, 92001 Nanterre Cedex.}
\email{marc.wouts@u-paris10.fr}
\date{April 14 \textsuperscript{th}, 2008}
\thanks{\withTeXmacstext}

\begin{document}
\hyphenpenalty=5000
\tolerance=1000

\maketitle

\begin{abstract}
  We study the surface tension and the phenomenon of phase coexistence for the
  Ising model on $\mathbbm{Z}^d$ ($d \geqslant 2$) with ferromagnetic but
  random couplings. We prove the convergence in probability (with respect to
  random couplings) of surface tension and analyze its large deviations :
  upper deviations occur at volume order while lower deviations occur at
  surface order. We study the asymptotics of surface tension at low
  temperatures and relate the quenched value $\tau^q$ of surface tension to
  maximal flows (first passage times if $d = 2$). For a broad class of
  distributions of the couplings we show that the inequality $\tau^a \leqslant
  \tau^q$ -- where $\tau^a$ is the surface tension under the averaged Gibbs
  measure -- is strict at low temperatures. We also describe the phenomenon of
  phase coexistence in the dilute Ising model and discuss some of the
  consequences of the media randomness. All of our results hold as well for
  the dilute Potts and random cluster models.\\

\end{abstract}

{\tableofcontents}

\label{sec-intro}

A considerable amount of work permitted to understand on a rigorous
basis the phenomenon of phase coexistence in models of statistical mechanics
like the Ising model. Phase coexistence in the Ising model was first described
in the pioneer work {\cite{N303}}, in the two dimensional case and at low
temperatures. The construction was then simplified {\cite{N306}} and extended
up to the critical temperature {\cite{N308,N309,N305}}, still in the two
dimensional case. The generalization to higher dimensions was achieved later
thanks to the $L^1$-approach~{\cite{N06,N11,N12}}. The interested reader will
find pedagogical presentations of the problem and the methods in the
course~{\cite{N70}} and the review~{\cite{N07}}.

The present work is concerned with the phenomenon of phase coexistence for the
dilute model with random (ferromagnetic) couplings. The random couplings model
either {\tmem{rare defects}} in the media, either {\tmem{intrinsic
randomness}}. As an example, quenched alloys made of magnetic materials have
intrinsic randomness since the strength of the interaction between two spins
depends on the nature of the two corresponding atoms.

In order to describe rigorously the phenomenon of phase coexistence in
presence of phase coexistence, we followed the same plan as in the above
mentioned works. In a first step we established a {\tmem{coarse graining}} for
the model {\cite{M01}}. In a second step -- the present one -- we study
{\tmem{surface tension}}. The combination of these tools allow us describe the
phenomenon of phase coexistence in the presence of random media.

Before we turn to the presentation of the model and of our results, we would
like to stress two consequences of the media randomness on the phenomenon of
phase coexistence: first, it is the case that the shape of crystals are
smoother than in presence of uniform couplings. Second, we give an insight to
the expected {\tmem{localization}} phenomenon of the crystal which is
determined by the realization of the media under averaged Gibbs measure.

The organization of the paper is as follows. In Section \ref{sec-results}
below we introduce the model and give a complete summary of our results on
surface tension, its low temperature asymptotics (maximal flows) and phase
coexistence. Proofs and intermediate results are given in the three
corresponding Sections \ref{sec-tau}, \ref{sec-lowt} and \ref{sec-phco}.

\section{The model and our main results}

\label{sec-results}

\subsection{The dilute Ising model}

The canonical vectors of $\mathbbm{R}^d$ are denoted $(\tmmathbf{e}_i)_{i = 1
\ldots d}$ and for any $x = \sum_{i = 1}^n x_i \tmmathbf{e}_i = (x_1, \ldots,
x_d) \in \mathbbm{R}^d$ we consider the following norms on $\mathbbm{R}^d$:
\begin{equation}
  \|x\|_1 = \sum_{i = 1}^d |x_i | \text{, \ \ \ } \|x\|_2 = \left( \sum_{i =
  1}^d x_i^2 \right)^{1 / 2} \text{ \ \ and \ \ } \|x\|_{\infty} = \max_{i =
  1}^d |x_i |. \label{eq-norms}
\end{equation}
Given $x, y \in \mathbbm{Z}^d$ we say that $x, y$ are nearest neighbors (which
we denote $x \sim y$) if they are at Euclidean distance $1$, i.e. if $\|x -
y\|_2 = 1$. To any domain $\Lambda \subset \mathbbm{Z}^d$ we associate the
edge sets
\begin{eqnarray}
  E (\Lambda) & = & \left\{ \left\{ x, y \right\} : x, y \in \Lambda \text{ \
  and \ } x \sim y \right\} \\
  \text{and \ \ } E^w (\Lambda) & = & \left\{ \left\{ x, y \right\} : x \in
  \Lambda, y \in \mathbbm{Z}^d \text{ \ and \ } x \sim y \right\} . 
\end{eqnarray}

We consider in this paper the dilute Ising model on $\mathbbm{Z}^d$ for $d
\geqslant 2$. It is defined in two steps : first, the couplings between
adjacent spins are represented by a random sequence $J = (J_e)_{e \in E
(\mathbbm{Z}^d)}$ of law $\mathbbm{P}$, such that the $(J_e)_{e \in E
(\mathbbm{Z}^d)}$ are independent, identically distributed in $[0, 1]$ under
$\mathbbm{P}$. Then, given $\Lambda \subset \mathbbm{Z}^d$ a finite domain and
a spin configuration $\sigma \in \Sigma^+_{\Lambda}$, where
\[ \Sigma^+_{\Lambda} = \left\{ \sigma : \mathbbm{Z}^d \rightarrow \{\pm 1\}:
   \sigma_z = 1, \forall z \notin \Lambda \right\}, \]
we let
\begin{equation}
  H_{\Lambda}^{J, +} (\sigma) = - \sum_{e =\{x, y\} \in E^w (\Lambda)} J_e
  \sigma_x \sigma_y
\end{equation}
the Hamiltonian with plus boundary condition on $\Lambda$. The dilute Ising
model on $\Lambda$ with plus boundary condition, given a realization $J$ of
the couplings, is the probability measure $\mu^{J, +}_{\Lambda}$ on
$\Sigma^+_{\Lambda}$ that satisfies
\begin{equation}
  \mu^{J, +}_{\Lambda} (\{\sigma\}) = \frac{1}{Z^{J, +}_{\Lambda, \beta}} \exp
  \left( - \frac{\beta}{2} H_{\Lambda}^{J, +} (\sigma) \right) \text{, \ \ }
  \forall \sigma \in \Sigma^+_{\Lambda}
\end{equation}
where $\beta \geqslant 0$ is the inverse temperature and $Z^{J, +}_{\Lambda,
\beta}$ is the partition function
\begin{equation}
  Z_{\Lambda, \beta}^{J, +} = \sum_{\sigma \in \Sigma^+_{\Lambda}} \exp \left(
  - \frac{\beta}{2} H_{\Lambda}^{J, +} (\sigma) \right) .
\end{equation}

Consider
\begin{equation}
  m_{\beta} = \lim_{N \rightarrow \infty} \mathbbm{E} \mu^{J,
  +}_{\hat{\Lambda}_N, \beta}  \left( \sigma_0 \right)
\end{equation}
the magnetization in the thermodynamic limit, where $\hat{\Lambda}_N$ is the
symmetric box $\hat{\Lambda}_N =\{- N, \ldots, N\}^d$ and $\mathbbm{E}$ the
expectation associated with $\mathbbm{P}$. When $m_{\beta} > 0$ the boundary
condition has an influence on the spins at an arbitrary distance. In the
region $m_{\beta} > 0$ we say that the Ising model has two phases because the
structure of the spins under $\mathbbm{E} \mu^{J, +}_{\Lambda}$ (the plus
phase) is not the same as the structure of the spins under $\mathbbm{E}
\mu^{J, -}_{\Lambda}$ (the minus phase), where $\mu^{J, -}_{\Lambda}$
corresponds to the minus boundary condition.

It is shown in {\cite{N21}} that the dilute Ising model undergoes a phase
transition at low temperature when the random interactions percolate. In our
settings, this means that the critical inverse temperature
\begin{equation}
  \beta_c = \inf \left\{ \beta \geqslant 0 : m_{\beta} > 0 \right\},
\end{equation}
which is never smaller than $\beta_c^{\tmop{pure}}$ -- the critical inverse
temperature for the pure Ising model ($J \equiv 1$) -- is finite if and only
if $\mathbbm{P}(J_e > 0) > p_c (d)$ where $p_c (d)$ is the threshold for bond
percolation on $\mathbbm{Z}^d$.

The aim of the paper is to understand the mechanism of phase coexistence in
the dilute Ising model, hence we will consider in the following a distribution
$\mathbbm{P}$ of the couplings such that $\mathbbm{P}(J_e > 0) > p_c (d)$ and
an inverse temperature $\beta > \beta_c$. However, some of our results hold on
a possibly stronger assumption $\beta > \hat{\beta}_c \geqslant \beta_c$ where
$\hat{\beta}_c$ is the critical inverse temperature for slab percolation --
see (\ref{eq-betasp}) below -- as this assumption allows us to use the
renormalization framework of {\cite{M01}}.

\subsection{The Fortuin-Kasteleyn representation}

The study of surface tension for the dilute Ising model will be led under the
random-cluster model that corresponds to the measure $\mu^{J, +}_{\Lambda,
\beta}$. We call
\[ \Omega = \left\{ \omega : E (\mathbbm{Z}^d) \rightarrow \{0, 1\} \right\}
\]
the set of cluster configurations on $E (\mathbbm{Z}^d)$, and for any $\omega
\in \Omega$ and $E \subset E (\mathbbm{Z}^d)$ we call $\omega_{|E}$ the
restriction of $\omega$ to $E$, defined by
\[ (\omega_{|E})_e = \left\{ \begin{array}{ll}
     \omega_e & \text{if } e \in E\\
     0 & \text{else.}
   \end{array} \right. \]
The set of cluster configurations on $E$ is $\Omega_E =\{\omega_{|E}, \omega
\in \Omega\}$. Given a parameter $q \geqslant 1$ and an inverse temperature
$\beta \geqslant 0$, a realization of the random couplings $J : E
(\mathbbm{Z}^d) \rightarrow [0, 1]$, a finite edge set $E \subset E
(\mathbbm{Z}^d)$ and a boundary condition $\pi \in \Omega_{E^c}$ we consider
the random cluster model $\Phi_{E, \beta}^{J, \pi, q}$ on $\Omega_E$ defined
by
\begin{equation}
  \Phi_{E, \beta}^{J, \pi, q} \left( \{\omega\} \right) = \frac{1}{Z_{E,
  \beta}^{J, \pi, q}} \prod_{e \in E} p_e^{\omega_e} (1 - p_e)^{1 - \omega_e}
  \times q^{C^{\pi}_E (\omega)} \text{, \ \ \ } \forall \omega \in \Omega_E
  \label{eq-def-FK}
\end{equation}
where $p_e = 1 - \exp (- \beta J_e)$, $C^{\pi}_E (\omega)$ is the number of
clusters of the set of vertices in $\mathbbm{Z}^d$ attained by $E$ under the
wiring $\omega \vee \pi$ such that $(\omega \vee \pi)_e = \max (\omega_e,
\pi_e)$, and $Z_{E, \beta}^{J, \pi, q}$ is the renormalization constant making
$\Phi_{E, \beta}^{J, \pi, q}$ a probability measure.

For convenience we use the same notation for the probability measure
$\Phi_{E, \beta}^{J, \pi, q}$ and for its expectation. Most often we will take
either $\pi = f$, where $f$ is the free boundary condition : $f_e = 0, \forall
e \in E^c$, or $\pi = w$ where $w$ is the wired boundary condition : $w_e = 1,
\forall e \in E^c$. When the parameters $q$ and $\beta$ are clear from the
context we omit them. Given $\mathcal{R}$ a compact subset of $\mathbbm{R}^d$
(usually a rectangular parallelepiped) we denote by $\Phi_{\mathcal{R}}^{J,
\pi}$ the measure $\Phi_{E ( \dot{\mathcal{R}} \cap \mathbbm{Z}^d)}^{J, \pi}$
on the cluster configurations on $E ( \dot{\mathcal{R}} \cap \mathbbm{Z}^d)$,
where $\dot{\mathcal{R}}$ stands for the interior of $\mathcal{R}$. In
particular, for any $g, h : \Omega \rightarrow \mathbbm{R}$ the quantities
$\Phi_{\mathcal{R}_1}^{J, \pi} (g)$ and $\Phi_{\mathcal{R}_2}^{J, \pi} (h)$
are independent under $\mathbbm{P}$.

The connection between the dilute Ising model $\mu^{J, +}_{\Lambda, \beta}$
and the random-cluster model was made explicit in {\cite{N214}}. Consider the
joint probability measure
\[ \Psi_{\Lambda, \beta}^{J, +} \left( \left\{ (\sigma, \omega) \right\}
   \right) = \frac{\tmmathbf{1}_{\left\{ \sigma \prec \omega
   \right\}}}{\tilde{Z}^{J, +}_{\Lambda, \beta}} \prod_{e \in E^w (\Lambda)}
   \left( p_e \right)^{\omega_e}  \left( 1 - p_e \right)^{1 - \omega_e},
   \text{ \ } \forall (\sigma, \omega) \in \Sigma^+_{\Lambda} \times \Omega_{E
   (\Lambda)} \]
where $p_e = 1 - \exp (- \beta J_e)$, $\sigma \prec \omega$ is the event that
$\sigma$ and $\omega$ are compatible, namely that $\omega_e = 1 \Rightarrow
\sigma_x = \sigma_y, \forall e =\{x, y\} \in E^w (\Lambda)$, and
$\tilde{Z}^{J, +}_{\Lambda, \beta}$ is the corresponding normalizing factor.
Then,
\begin{enumerateroman}
  \item The marginal of $\Psi_{\Lambda, \beta}^{J, +}$ on the variable
  $\sigma$ is the Ising model $\mu^{J, +}_{\Lambda}$,
  
  \item Its marginal on the variable $\omega$ is the random-cluster model
  $\Phi_{E (\Lambda), \beta}^{J, w, 2}$ with wired boundary condition $w$ and
  parameter $q = 2$.
  
  \item Conditionally on $\omega$, the spin $\sigma$ of each connected
  component of $\Lambda$ for $\omega$ (now {\tmem{cluster}}) is constant, and
  equal to $+ 1$ if the cluster is connected to $\Lambda^c$. The spin of all
  clusters not touching $\Lambda^c$ are independent and equal to $+ 1$ with a
  probability $1 / 2$.
  
  \item Conditionally on $\sigma$, the edges are open (i.e. $\omega_e = 1$ for
  $e =\{x, y\}$) independently, with respective probabilities $p_e
  \delta_{\sigma_x, \sigma_y}$.
\end{enumerateroman}

According to point {\tmem{ii}} and {\tmem{iii}} we can study surface tension
for the Ising model under the Fortuin-Kasteleyn representation. This
representation allows to study at the same time the surface tension and the
phenomenon of phase coexistence for the dilute Ising model ($q = 2$), but also
for dilute percolation ($q = 1$) and for the dilute Potts model ($q \in \{3,
4, \ldots\}$).

An important benefit of the representation is that it makes possible the use
of the comparison inequalities for the random cluster model. We say that a
function $f : \Omega_E \rightarrow \mathbbm{R}^+$ is increasing if, for all
$\omega, \omega' \in \Omega_E$ one has $\omega \leqslant_{\Omega} \omega'
\Rightarrow f (\omega) \leqslant f (\omega')$ where $\leqslant_{\Omega}$
stands for the product order on $\Omega_E$. It was shown in {\cite{N210}}
that:
\begin{enumerateroman}
  \item For any $h : \Omega_E \rightarrow \mathbbm{R}^+$ increasing, $\Phi_{E,
  \beta}^{J, \pi, q} (h)$ is a non-decreasing function of $J$, $\beta$ and
  $\pi$.
  
  \item (FKG inequality) For any $g, h : \Omega_E \rightarrow \mathbbm{R}^+$
  increasing,
  \begin{equation}
    \Phi_{E, \beta}^{J, \pi, q} (g h) \geqslant \Phi_{E, \beta}^{J, \pi, q}
    (g) \Phi_{E, \beta}^{J, \pi, q} (h) .
  \end{equation}
  \item (DLR Equation) For any $E' \subsetneq E$ and $\omega' \in
  \Omega_{E'}$,
  \begin{equation}
    \Phi_{E, \beta}^{J, \pi, q} \left( . | \omega_{|E'} = \omega' \right) =
    \Phi_{E \setminus E', \beta}^{J, \pi \vee \omega', q} .
  \end{equation}
\end{enumerateroman}

Finally, let us recall the assumption of slab percolation, that is the basis
for a renormalization framework in the dilute Ising model {\cite{M01}}. When
$d \geqslant 3$, we say that slab percolation occurs under $\mathbb{E}
\Phi^{J, f, q}_{\beta}$ if, for large enough $H$,
\begin{equation}
  \inf_{L \in \mathbb{N}^{\star}} \inf_{x, y \in S_{L, H}} \mathbb{E} \Phi^{J,
  f}_{S_{L, H}} \left( x \overset{\omega}{\leftrightarrow} y \right) > 0
\end{equation}
where $S_{L, H}$ is the slab $S_{L, H} =\{1, \ldots, L\}^{d - 1} \times \{1,
\ldots H\}$. When $d = 2$, we say that slab percolation occurs when there
exists $\kappa : \mathbb{N}^{\star} \mapsto \mathbb{N}^{\star}$ with $\lim_{N
\rightarrow \infty} \kappa (N) / N = 0$ such that
\begin{equation}
  \lim_{N \rightarrow \infty} \mathbb{E} \Phi^{J, f}_{S_{N, \kappa (N)}}
  \left( \text{there is an horizontal crossing for $\omega$} \right) > 0.
\end{equation}
The critical inverse temperature for slab percolation is
\begin{equation}
  \hat{\beta}_c = \inf \left\{ \beta \geqslant 0 : \text{slab percolation
  occurs under } \text{$\mathbb{E} \Phi^{J, f, q}_{\beta}$} \right\},
  \label{eq-betasp}
\end{equation}
it satisfies $\hat{\beta}_c \geqslant \beta_c$ where $\beta_c$ is the critical
inverse temperature for phase transition in the dilute Ising (resp. Potts)
model. We believe that $\hat{\beta}_c$ and $\beta_c$ do coincide. Upper bounds
on $\hat{\beta}_c$ are derived in {\cite{M01}} from the argument of
{\cite{N21}}. The technical assumption $\beta > \hat{\beta}_c$ allows us to
use a coarse graining, which is a fundamental tool at the moment of defining
the local phase of the dilute Ising model (see Theorem 5.7 in {\cite{M01}}, or
Section \ref{sec-phco} below).

\subsection{Surface tension}

One of the main issue we address in this paper is the behavior of surface
tension and the influence of the random couplings. We consider the surface
tension in large rectangular parallelepiped oriented along some direction
$\tmmathbf{n} \in S^{d - 1}$, where $S^{d - 1}$ is the set of unit vectors of
$\mathbbm{R}^d$. The other axes of the parallelepiped are represented by
$\mathcal{S} \in \mathbbm{S}_{\tmmathbf{n}}$, where
\[ \mathbbm{S}_{\tmmathbf{n}} = \left\{ \sum_{k = 1}^{d - 1} [\pm 1 /
   2]\tmmathbf{u}_k ; (\tmmathbf{u}_1, \ldots, \tmmathbf{u}_{k - 1},
   \tmmathbf{n}) \text{ is an orthonormal basis of } \mathbbm{R}^d \right\} \]
is the set of $d - 1$ dimensional hypercubes of side-length $1$, centered at
$0$, orthogonal to $\tmmathbf{n} \in S^{d - 1}$. Finally, we call $x \in
\mathbbm{R}^d$ the center of the rectangular parallelepiped and $L, H$ its
side-lengths, and denote finally by
\begin{equation}
  \mathcal{R}_{x, L, H} (\mathcal{S}, \tmmathbf{n}) = x + L\mathcal{S}+ [- H,
  H]\tmmathbf{n} \label{eq-def-R}
\end{equation}
the rectangular parallelepiped centered at $x$, with basis $x + L\mathcal{S}$
and extension $2 H$ in the direction $\tmmathbf{n}$ (see Figure
\ref{fig-rect-LH}). The discrete version of $\mathcal{R}$ is
$\hat{\mathcal{R}} = \dot{\mathcal{R}} \cap \mathbbm{Z}^d$ and the inner
discrete boundary of $\mathcal{R}$ is
\[ \partial \hat{\mathcal{R}} = \left\{ y \in \hat{\mathcal{R}} : \exists z
   \in \mathbbm{Z}^d \setminus \hat{\mathcal{R}}, z \sim y \right\} . \]
For any $\mathcal{R}$ as in (\ref{eq-def-R}) we decompose $\partial
\hat{\mathcal{R}}$ into its \tmtextit{upper} and \tmtextit{lower} parts
$\partial^+  \hat{\mathcal{R}} =\{y \in \partial \hat{\mathcal{R}} : (y - x)
\cdot \tmmathbf{n} \geqslant 0\}$ and $\partial^-  \hat{\mathcal{R}} =\{y \in
\partial \hat{\mathcal{R}} : (y - x) \cdot \tmmathbf{n}< 0\}$.

\begin{figure}[h!]
  \begin{center}
  \includegraphics[width=8cm]{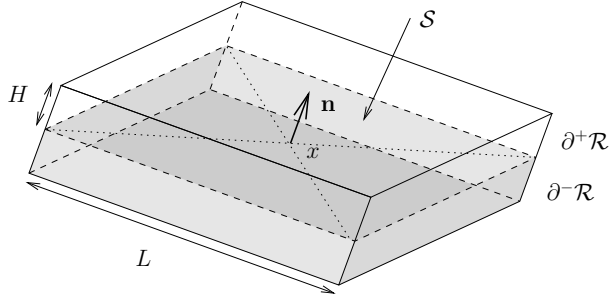}
  \end{center}
  \caption{\label{fig-rect-LH}The rectangular parallelepiped $\mathcal{R}_{x,
  L, H} (\mathcal{S}, \tmmathbf{n})$.}
\end{figure}

In the context of statistical physics, the surface tension is the excess free
energy per surface unit due to the presence of an interface. The surface
tension in $\mathcal{R}$ thus quantifies the probability of observing the plus
phase in the upper part of $\mathcal{R}$ and the minus phase in the opposite
part under the measure $\mu_{\mathcal{R}}^J$ with free boundary condition. It
is more convenient to formulate the definition under the random cluster model,
where we translate the event of {\tmem{phase coexistence}} into an event of
{\tmem{disconnection}}.

\begin{definition}
  \label{def-tauJ}Let $\mathcal{R}$ be a rectangular parallelepiped as in
  (\ref{eq-def-R}). The event of disconnection between the upper and lower
  parts of $\partial \hat{\mathcal{R}}$ is
  \begin{equation}
    \mathcal{D}_{\mathcal{R}} = \left\{ \text{$\omega \in \Omega$} :
    \partial^+  \hat{\mathcal{R}}  \overset{\omega}{\nleftrightarrow}
    \partial^-  \hat{\mathcal{R}} \right\} \label{eq-def-DR}
  \end{equation}
  and the surface tension in $\mathcal{R}$ is
  \begin{equation}
    \tau^J_{\mathcal{R}} = - \frac{1}{L^{d - 1}} \log \Phi^{J,
    w}_{\mathcal{R}} \left( \mathcal{D}_{\mathcal{R}} \right)
    \label{eq-def-tauJ} .
  \end{equation}
\end{definition}

We denote by $J^{\min}$ and $J^{\max}$ the lowest and largest values of the
couplings according to the support of $\mathbbm{P}$, that is to say :
\begin{eqnarray*}
  J^{\min} & = & \inf \{\lambda \geqslant 0 : \mathbbm{P}(J_e < \lambda) >
  0\}\\
  \text{and \ } J^{\max} & = & \sup \{\lambda \geqslant 0 : \mathbbm{P}(J_e >
  \lambda) > 0\}.
\end{eqnarray*}
We also denote by $\tau^{\min}_{\mathcal{R}}$ (resp.
$\tau^{\max}_{\mathcal{R}}$) the value of the surface tension in $\mathcal{R}$
corresponding to the constant couplings $J \equiv J^{\min}$ (resp. $J \equiv
J^{\max}$). We have:

\begin{proposition}
  \label{prop-tauJ-mon}Let $\mathcal{R}$ be a rectangular parallelepiped as in
  (\ref{eq-def-R}), with $L, H \geqslant 2 \sqrt{d}$. The surface tension
  $\tau^J_{\mathcal{R}}$ is a non-decreasing function of $J$ and $\beta$. It
  is a non-increasing function of $H$. With probability one under
  $\mathbbm{P}$,
  \[ 0 \leqslant \tau^{\min}_{\mathcal{R}} (\tmmathbf{n}) \leqslant
     \tau^J_{\mathcal{R}} (\tmmathbf{n}) \leqslant \tau^{\max}_{\mathcal{R}}
     (\tmmathbf{n}) \leqslant c_d \beta J^{\max} \]
  where $c_d < \infty$ depends on $d$ only.
\end{proposition}

As in the uniform case {\cite{N08}}, surface tension is sub-additive (see
Theorem \ref{thm-add-tauJ}), and this implies convergence in probability of
$\tau^J_{\mathcal{R}} (\tmmathbf{n})$.

\begin{theorem}
  \label{thm-conv-tauq}There exists $\tau^q_{\beta} (\tmmathbf{n}) \geqslant
  0$, the {\tmem{quenched}} surface tension, such that, for all $\beta
  \geqslant 0$ and $\tmmathbf{n} \in S^{d - 1}$,
  \[ \lim_{N \rightarrow \infty} \tau^J_{\mathcal{R}_{0, N, \delta N}
     (\mathcal{S}, \tmmathbf{n})} = \tau^q_{\beta} (\tmmathbf{n}) \text{ \ \
     in } \mathbbm{P} \text{-probability} \]
  whatever is $\mathcal{S} \in \mathbbm{S}_{\tmmathbf{n}}$ and $\delta > 0$.
\end{theorem}

Similarly, the surface tension for the constant couplings $J^{\min}$ and
$J^{\max}$ also converge and we denote by $\tau^{\min} (\tmmathbf{n})$ and
$\tau^{\max} (\tmmathbf{n})$ their respective limits.

The sub-additivity is of much help for controlling the order of deviations
from the quenched value of surface tension $\tau_{\beta}^q (\tmmathbf{n})$.
Upper large deviations happen at a volume order hence they have no influence
on the phenomena we study here:

\begin{theorem}
  \label{thm-updev-tauJ}For any $\varepsilon > 0$ and $\delta > 0$,
  \[ \limsup_N \frac{1}{N^d} \log \mathbbm{P} \left( \tau^J_{\mathcal{R}_{0,
     N, \delta N} (\mathcal{S}, \tmmathbf{n})} \geqslant \tau^q (\tmmathbf{n})
     + \varepsilon \right) < 0. \]
\end{theorem}

We will be more concerned with lower large deviations. These are possible when
$\tau^{\min} (\tmmathbf{n}) < \tau^q (\tmmathbf{n})$. The inequality is known
to be strict only in two specific cases: when $J^{\min} = 0$ and $\beta >
\hat{\beta}_c$, it is the case that $\tau^{\min} (\tmmathbf{n}) < \tau^q
(\tmmathbf{n})$ because $\tau^{\min} (\tmmathbf{n}) = 0$, while the coarse
graining {\cite{M01}} implies:

\begin{proposition}
  \label{prop-tauq-pos}Assume $\beta > \hat{\beta}_c$. For any $\tmmathbf{n}
  \in S^{d - 1}$, $\tau^q (\tmmathbf{n}) > 0$.
\end{proposition}

When $\mathbbm{P}(J_e > J^{\min}) > p_c (d)$ and $\beta$ is large enough, the
strict inequality $\tau^{\min} (\tmmathbf{n}) < \tau^q (\tmmathbf{n})$ also
holds, cf. Corollary \ref{cor-taulq-lowt} below.

When lower large deviations occur, they have at most surface order. It is
another consequence of sub-additivity that:

\begin{theorem}
  \label{thm-rate-I}For every $\tmmathbf{n} \in S^{d - 1}$ and $\beta
  \geqslant 0$, $\tau > \tau^{\min} (\tmmathbf{n})$, the limit
  \begin{equation}
    I_{\tmmathbf{n}} (\tau) = \lim_N - \frac{1}{N^{d - 1}} \log \mathbbm{P}
    \left( \tau^J_{\mathcal{R}_{0, N, \delta N} (\mathcal{S}, \tmmathbf{n})}
    \leqslant \tau \right) \label{eq-conv-Itaud}
  \end{equation}
  exists in $[0, + \infty)$ and does not depend on $\delta > 0$, nor on
  $\mathcal{S} \in \mathbbm{S}_{\tmmathbf{n}}$. $I_{\tmmathbf{n}}$ is
  continuous, convex non-increasing, and $I_{\tmmathbf{n}} (\tau) = 0$ for
  $\tau \geqslant \tau^q (\tmmathbf{n})$.
\end{theorem}

For convenience, we extend the definition of $I_{\tmmathbf{n}}$ letting
\begin{equation}
  I_{\tmmathbf{n}} (\tau) = \left\{ \begin{array}{ll}
    + \infty & \text{for } \tau < \tau^{\min} (\tmmathbf{n})\\
    \lim_{\varepsilon \rightarrow 0^+} I_{\tmmathbf{n}} (\tau + \varepsilon) &
    \text{at } \tau = \tau^{\min} (\tmmathbf{n}) .
  \end{array} \right.
\end{equation}

\begin{figure}[h!]
  \begin{center}
  \includegraphics[width=8cm]{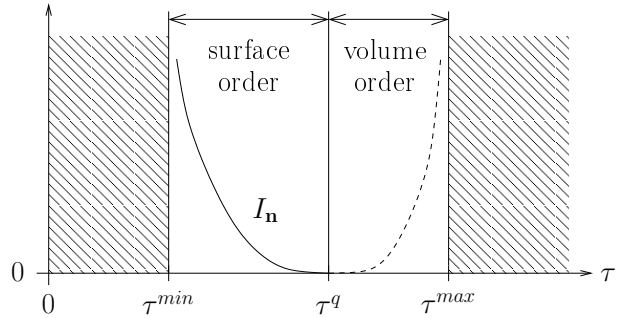}
  \end{center}
  \caption{\label{fig-rate-dev}Large deviations of surface tension.}
\end{figure}

In order to show that lower deviations are exactly of surface order, we need
to prove that $I_{\tmmathbf{n}}$ is positive on the left of $\tau^q$. We
developed an argument based on measure concentration coupled with a control of
the length of the interface. In the case of the Ising model at low
temperatures we could establish a first control:

\begin{theorem}
  \label{thm-Iquad-lowt}Assume $q = 2$ and $J^{\min} > 0$. Then, for $\beta$
  large enough there exists $c > 0$ such that, for all $r > 0$,
  \begin{equation}
    I_{\tmmathbf{n}} (\tau^q (\tmmathbf{n}) - r) \geqslant cr^2 .
  \end{equation}
\end{theorem}

In the general case a careful adaptation of the method yields:

\begin{theorem}
  \label{thm-Iquad-gal}For every $\tmmathbf{n} \in S^{d - 1}$, for
  Lebesgue-almost all $\beta \geqslant 0$,
  \begin{equation}
    \limsup_{r \rightarrow 0^+} \frac{I_{\beta, \tmmathbf{n}} (\tau^q_{\beta}
    (\tmmathbf{n}) - r)}{r^2} > 0 \label{eq-quad-lwb-I} .
  \end{equation}
\end{theorem}

These quadratic lower bounds on the rate function generalize common controls
for directed polymers models {\cite{N262,N271}}, which were introduced in
order to represent interfaces in the two-dimensional Ising model with random
couplings at low temperatures {\cite{N167,N262}}. Its is probable however that
the quadratic order in the former Theorems is not optimal in two dimensions,
as the comparison with directed polymers suggests that
\begin{equation}
  I_{\tmmathbf{n}} (\tau^q (\tmmathbf{n}) - r) \underset{r \rightarrow
  0^+}{\sim} cr^{3 / 2} \text{ \ and \ } \xi = \frac{2}{3} .
  \label{eq-Ixi-lpp}
\end{equation}
This scaling has been established rigorously for the zero-temperature limit of
directed polymers: \tmtextit{last passage percolation}, for a geometric
distribution of the passage times, see Theorem 1.1 and (2.23) in
{\cite{N258}}.

Some of our results on phase coexistence and on the dynamics of the dilute
Ising model {\cite{M00}} require that lower large deviations are actually of
surface order. An easy but important consequence of Theorem
\ref{thm-Iquad-gal} is:

\begin{corollary}
  \label{cor-NI-count}The lower large deviations are of surface order when
  $\beta \mapsto \tau^q_{\beta} (\tmmathbf{n})$ is left continuous. Hence the
  set
  \begin{equation}
    \mathcal{N}_I = \left\{ \beta \geqslant 0 : \exists \tmmathbf{n} \in S^{d
    - 1} \text{ and } r > 0 \text{ such that } I_{\beta, \tmmathbf{n}}
    (\tau^q_{\beta} (\tmmathbf{n}) - r) = 0 \right\} \label{eq-def-NI}
  \end{equation}
  is at most countable.
\end{corollary}

We end the presentation on surface tension with the definition of the surface
tension under the averaged Gibbs measure. It is the Fenchel-Legendre transform
of $I_{\tmmathbf{n}}$,
\begin{equation}
  \tau^{\lambda} (\tmmathbf{n}) = \inf_{\tau \in \mathbbm{R}} \left\{ \lambda
  \tau + I_{\tmmathbf{n}} (\tau) \right\} \label{eq-dual-Itaul}
\end{equation}
which coincides with the surface tension under an average of Gibbs measures:
\begin{equation}
  \tau^{\lambda} (\tmmathbf{n}) = \lim_{N \rightarrow \infty} - \frac{1}{N^{d
  - 1}} \log \mathbbm{E} \left( \left[ \Phi^{J, w}_{\mathcal{R}^N} \left(
  \mathcal{D}_{\mathcal{R}^N} \right) \right]^{\lambda} \right) \text{ \ \ \ \
  } \lambda > 0, \tmmathbf{n} \in S^{d - 1} \label{eq-cvg-taul}
\end{equation}
where $\mathcal{R}^N =\mathcal{R}_{0, N, \delta N} (\mathcal{S},
\tmmathbf{n})$, as shown in Proposition \ref{prop-conv-taul}. The particular
case $\lambda = 1$ corresponds to the usual notion of surface tension under
the averaged measure (or {\tmem{annealed}} surface tension) and we denote
$\tau^a (\tmmathbf{n}) = \tau^{\lambda = 1} (\tmmathbf{n})$.

The asymptotics of $\tau^{\lambda} / \lambda$ as $\lambda \rightarrow 0$ or $+
\infty$ are given in Proposition \ref{prop-taul}. An important question about
the surface tension under the averaged Gibbs measure is whether the random
media is able to turn Jensen's inequality
\begin{equation}
  \tau^{\lambda} (\tmmathbf{n}) \leqslant \lambda \tau^q (\tmmathbf{n})
\end{equation}
into a strict inequality. A partial answer to this question is given in the
next Section. Let us explicit the connection between the strict inequality
$\tau^{\lambda} (\tmmathbf{n}) \leqslant \lambda \tau^q (\tmmathbf{n})$ and
the asymptotics of $I_{\tmmathbf{n}}$ on the left of $\tau^q (\tmmathbf{n})$:
the opposite of the slope of $I_{\tmmathbf{n}}$ on the left of $\tau^q
(\tmmathbf{n})$ is exactly
\begin{equation}
  \alpha_{\tmmathbf{n}} = \sup \left\{ \lambda > 0 : \tau^{\lambda}
  (\tmmathbf{n}) = \lambda \tau^q (\tmmathbf{n}) \right\},
\end{equation}
with the convention that $\sup \emptyset = 0$.

Finally, as in the non-random case, the homogeneous extension of each of these
notions of surface tension $\tau^q$, $\tau^{\lambda}$, $\tau^{\min}$ and
$\tau^{\max}$ are convex and continuous (Proposition \ref{prop-conv-tau}).

\subsection{Low temperature asymptotics}

\label{sec-intro-lowt}The low temperature asymptotics of surface tension
permit to give a more precise insight into the properties of surface tension
in random media. First we need to introduce the concept of {\tmem{maximal
flow}} through the capacities $J$, where $J = (J_e)_{e \in E (\mathbbm{Z}^d)}$
is the family of random couplings introduced in the former section. Here we
give only a brief overview of maximal flows. The reader is invited to consult
{\cite{N244,N245}} for a pedagogical introduction. Recent results on maximal
flows, including large deviations, can be found in {\cite{N273,N370,N352}}.

We will use an analogy for describing maximal flows. Imagine a liquid which
has to cross a lattice made of tubes with limited capacity. Then, the maximal
flow, in a given direction, is the quantity of liquid that can flow through
the lattice, per unit of surface.

Given a rectangular parallelepiped $\mathcal{R}$ as in (\ref{eq-def-R}) and $I
\subset E ( \hat{\mathcal{R}})$, we consider the event that $\omega$ is closed
on $I$:
\[ \mathcal{Z}_I = \left\{ \omega_e = 0, \forall e \in I \right\} . \]
We say that $I$ is an \tmtextit{interface} for $\mathcal{R}$ if $\mathcal{Z}_I
\subset \mathcal{D}_{\mathcal{R}}$ and $\forall e \in I, \mathcal{Z}_{I
\setminus \{e\}} \not{\subset} \mathcal{D}_{\mathcal{R}}$. In other words, $I$
is an interface for $\mathcal{R}$ if the disconnection on $I$ is enough for
disconnecting $\partial^+ \hat{\mathcal{R}}$ from $\partial^-
\hat{\mathcal{R}}$ and if there is no superfluous edge in $I$. This notion of
interface corresponds to the geometrical notion of interface if, to the edges
of $I$ we associate their dual $d - 1$ dimensional facets.

According to the max-flow min-cut Theorem {\cite{N386}}, the maximum flow
from $\partial^- \hat{\mathcal{R}}$ to $\partial^+ \hat{\mathcal{R}}$ by the
edges of capacities $J_e$ is also the flow through the interface of minimal
capacity. We use this characterization for our definition. Given a rectangular
parallelepiped $\mathcal{R}=\mathcal{R}_{x, L, H} (\mathcal{S}, \tmmathbf{n})$
as in (\ref{eq-def-R}) we call $\mathcal{I}(\mathcal{R})$ the set of
interfaces for $\mathcal{R}$ and define the maximal flow in $\mathcal{R}$, for
a realization $J$ of the media, as
\begin{equation}
  \mu^J_{\mathcal{R}} = \frac{1}{L^{d - 1}} \inf_{I \in
  \mathcal{I}(\mathcal{R})} \sum_{e \in I} J_e . \label{eq-def-muJ}
\end{equation}
This quantity has the same properties as surface tension since it is as well
sub-additive. In particular, the maximal flow in $\mathcal{R}_{0, N, \delta N}
(\mathcal{S}, \tmmathbf{n})$ converges in $\mathbbm{P}$-probability, upper
deviations occur at volume order and lower deviations occur at surface order
{\cite{N274,N350}}. We will make use of the following results:

\begin{theorem}
  \label{thm-conv-mu}There exists $\mu (\tmmathbf{n}) \in [0, + \infty)$, the
  maximal flow for the distribution $\mathbbm{P}$ in the direction
  $\tmmathbf{n} \in S^{d - 1}$, such that, for any $\delta > 0$ and
  $\mathcal{S} \in \mathbbm{S}_{\tmmathbf{n}}$,
  \[ \mu^J_{\mathcal{R}_{0, N, \delta N} (\tmmathbf{n}, \mathcal{S})}
     \underset{N \rightarrow \infty}{\longrightarrow} \mu (\tmmathbf{n})
     \text{ \ in $\mathbbm{P}$-probability} . \]
  It is positive if and only if $\mathbbm{P}(J_e > 0) > p_c (d)$. Furthermore,
  \begin{equation}
    J^{\min} \|\tmmathbf{n}\|_1 \leqslant \mu (\tmmathbf{n})
    \label{eq-mu-Jmin}
  \end{equation}
  and the inequality is strict when $\mathbbm{P}(J_e > J^{\min}) > p_c (d)$.
\end{theorem}

The convergence of the maximal flow is a consequence of the sub-additivity. It
is shown in {\cite{N337}} that the maximal flow is 0 when $\mathbbm{P}(J_e >
0) \leqslant p_c (d)$, while its positivity was established in {\cite{N338}},
under the conjecture that the critical threshold for percolation and slab
percolation do coincide, proved later on in {\cite{N170}}. The inequality
(\ref{eq-mu-Jmin}) is easily obtained from the remark that minimal interfaces
have cardinal of order $N^{d - 1} \|\tmmathbf{n}\|_1$. When $\mathbbm{P}(J_e >
J^{\min}) > p_c (d)$, for small $\varepsilon > 0$ there is a percolating net
of edges with values $J_e \geqslant J^{\min} + \varepsilon$ {\cite{N170,N09}},
which is responsible for the strict inequality. See also Proposition 4.1 in
{\cite{N370}}.

It turns out that the maximal flow determines the asymptotics of the quenched
value of surface tension at low temperatures. Precisely, we show that:

\begin{proposition}
  \label{prop-tauq-mu}Let $\mathbbm{P}$ be a product measure on $[0, 1]^d$
  such that $\mathbbm{P}(J_e > 0) = 1$. Then, uniformly over $\tmmathbf{n} \in
  S^{d - 1}$,
  \begin{equation}
    \lim_{\beta \rightarrow \infty} \frac{\tau^q_{\beta}
    (\tmmathbf{n})}{\beta} = \mu (\tmmathbf{n}) . \label{eq-tauq-mu}
  \end{equation}
\end{proposition}

Clearly, (\ref{eq-tauq-mu}) also holds in the case $\mathbbm{P}(J_e > 0)
\leqslant p_c (d)$ since $\tau^q_{\beta} (\tmmathbf{n}) \leqslant \beta \mu
(\tmmathbf{n}) = 0$ (cf. Lemma \ref{lem-tau-lt} and Theorem
\ref{thm-conv-mu}). When $\mathbbm{P}(J_e > 0) > p_c (d)$ a renormalization
argument allows us to prove that:

\begin{proposition}
  \label{prop-tauqlt-b}Assume that $\mathbbm{P}(J_e > 0) > p_c (d)$. Then,
  \begin{equation}
    \liminf_{\beta \rightarrow + \infty} \frac{\tau^q_{\beta}
    (\tmmathbf{n})}{\beta} > 0 \label{eq-conv-tauq-infty},
  \end{equation}
  uniformly over $\tmmathbf{n} \in S^{d - 1}$.
\end{proposition}

On the other hand, the surface tension under the averaged Gibbs measure is
asymptotically determined by $J^{\min}$:

\begin{proposition}
  \label{prop-taul-Jmin}For all product measure $\mathbbm{P}$ on $[0, 1]^d$
  and all $\lambda > 0$, uniformly over $\tmmathbf{n} \in S^{d - 1}$,
  \begin{equation}
    \lim_{\beta \rightarrow + \infty} \frac{\tau^{\lambda}_{\beta}
    (\tmmathbf{n})}{\beta} = \lambda J^{\min} \|\tmmathbf{n}\|_1
    \label{eq-conv-taul-Jmin} .
  \end{equation}
\end{proposition}

In the case $J^{\min} = 0$ and $\mathbbm{P}(J_e > 0) = 1$, an equivalent to
$\tau^{\lambda}_{\beta} (\tmmathbf{n})$ is given in Proposition
\ref{prop-tault-Js1}.

These asymptotics have consequences on the shape of the crystals under both
the quenched and the averaged Gibbs measure (see Proposition
\ref{prop-Wulff-lowt} below). They also immediately imply that the inequality
$\tau^{\lambda}_{\beta} (\tmmathbf{n}) \leqslant \lambda \tau^q_{\beta}
(\tmmathbf{n})$ is strict at low temperatures in a number of cases:

\begin{corollary}
  \label{cor-taulq-lowt}Assume that $\mathbbm{P}(J_e > J^{\min}) > p_c (d)$.
  Then, for any $\lambda > 0$ there is $\beta_c^{\lambda} < \infty$ such that
  \begin{equation}
    \tau^{\lambda}_{\beta} (\tmmathbf{n}) < \lambda \tau^q_{\beta}
    (\tmmathbf{n}) \text{, \ \ \ } \forall \tmmathbf{n} \in S^{d - 1}, \forall
    \beta > \beta_c^{\lambda} . \label{ineq-taulq-st}
  \end{equation}
\end{corollary}

In particular if $J^{\min} = 0$ and if there is still a phase transition (i.e.
$\mathbbm{P}(J_e > 0) > p_c (d)$), then the inequality is always strict at low
temperatures.

One consequence of (\ref{ineq-taulq-st}) is the strict inequality
$\tau^{\min}_{\beta} (\tmmathbf{n}) < \tau^q_{\beta} (\tmmathbf{n})$ under the
same assumptions than in the Corollary, as $\lambda \tau^{\min}_{\beta}
(\tmmathbf{n}) \leqslant \tau^{\lambda}_{\beta} (\tmmathbf{n})$ (Proposition
\ref{prop-taul}).

Let us conclude on a comparison with the directed polymer model in $1 + 1$
dimensions: for the latter model, it was proved recently {\cite{N145}} that
the Lyapunov exponent is positive at all $\beta \geqslant 0$, which
corresponds in our settings to the strict inequality $\tau_{\beta}^a
(\tmmathbf{n}) = \tau^{\lambda = 1}_{\beta} < \tau^q_{\beta} (\tmmathbf{n})$.

\subsection{Phase coexistence}

We describe finally the phenomenon of phase coexistence in the dilute Ising
model. Phase coexistence occurs when both the plus and the minus phase are
present at the same time and occupy (distinct) regions of the domain. This
phenomenon does not occur naturally in the Ising model. One way of obtaining
phase coexistence is by conditioning the measure $\mu^J_{\Lambda}$ on the
event that the overall magnetization
\begin{equation}
  m_{\Lambda} = \frac{1}{| \Lambda |}  \sum_{x \in \Lambda} \sigma_x
\end{equation}
is smaller than $m < m_{\beta}$. Under this conditional measure, we will show
that the two phases do coexist and that the minus phases occupies a fraction
of the volume $v$ such that $(1 - 2 v) m_{\beta} = m$. Furthermore, the shape
$U$ of the region containing the minus phase is deterministic : if $\tau$ is
the surface tension of the model, the observed shape minimizes the
{\tmem{surface energy}}
\[ \mathcal{F}(U) = \int_{\partial U} \tau (\tmmathbf{n}) d s \]
under the volume constraint $\tmop{Vol} (U) \geqslant v$, and this implies
that $U$ is a translated of $v^{1 / d} \mathcal{W}$ where $\mathcal{W}$ is the
renormalized Wulff crystal associated to $\tau$:
\begin{equation}
  \mathcal{W}= \lambda \left\{ x \in \mathbbm{R}^d : x \cdot \tmmathbf{n}
  \leqslant \tau (\tmmathbf{n}), \forall \tmmathbf{n} \in S^{d - 1} \right\}
  \label{eq-Wulff}
\end{equation}
where $\lambda > 0$ is chosen such that $\tmop{Vol} (\mathcal{W}) = 1$.

Before we state our results, let us recall that $\mathcal{N}_I$ stands for the
at-most-countable set of $\beta$ at which lower large deviations for surface
tension are possible at less than surface order (Corollary \ref{cor-NI-count})
and $\hat{\beta}_c$ is the slab percolation threshold (\ref{eq-betasp}).
Another important notation is
\begin{equation}
  \mathcal{N}= \left\{ \beta \geqslant 0 : \lim_{N \rightarrow \infty}
  \mathbbm{E} \Phi^{J, f}_{\hat{\Lambda}_N} \neq \lim_{N \rightarrow \infty}
  \mathbbm{E} \Phi^{J, w}_{\hat{\Lambda}_N} \right\} \label{eq-def-N}
\end{equation}
the set of $\beta$ such that infinite volume averaged FK measures are not
unique. $\mathcal{N}$ is at most countable (Theorem 2.3 in {\cite{M01}}).

We denote by $\mathcal{W}^q$ (resp. $\mathcal{W}^{\lambda}$) the Wulff crystal
associated with the surface tension $\tau^q$ (resp. $\tau^{\lambda}$) as in
(\ref{eq-Wulff}), and $v = \alpha^d$ the fraction of the volume occupied by
the minus phase. The Wulff crystal $\alpha \mathcal{W}$ of volume $v$ fits
into the unit box $[0, 1]^d$ only if $\alpha \tmop{diam}_{\infty}
(\mathcal{W}) \leqslant 1$, where
\[ \tmop{diam}_{\infty} (A) = \sup_{x, y \in A} \|x - y\|_{\infty}, \text{ \ \
   \ } A \subset \mathbbm{R}^d . \]
Our first theorem concerns the cost of the lower large deviations for the
magnetization. In the sequel, $\Lambda_N =\{1, \ldots, N\}^d$.

\begin{theorem}
  \label{thm-cost-phco}Assume $\beta > \hat{\beta}_c$ with $\beta \notin
  \mathcal{N}$. Then, for all $0 \leqslant \alpha < 1 / \tmop{diam}_{\infty}
  (\mathcal{W}^q)$,
  \begin{equation}
    \frac{1}{N^{d - 1}} \log \mu^{J, +}_{\Lambda_N} \left(
    \frac{m_{\Lambda_N}}{m_{\beta}} \leqslant 1 - 2 \alpha^d \right)
    \underset{N \rightarrow \infty}{\longrightarrow} -\mathcal{F}^q (\alpha
    \mathcal{W}^q) \text{ \ \ in } \mathbbm{P} \text{-probability.}
  \end{equation}
\end{theorem}

Then we describe the geometry of the two phases. We consider a mesoscopic
scale $K \in \mathbbm{N}^{\star}$ and define the magnetization profile
$\mathcal{M}_K$ as
\begin{equation}
  \begin{array}{cccc}
    \mathcal{M}_K : & [0, 1]^d & \longrightarrow & [- 1, 1]\\
    & x & \longmapsto & \frac{1}{K^d} \sum_{z \in \Lambda_N \cap \Delta_{i
    (x)}} \sigma_z
  \end{array} \label{eq-def-MK}
\end{equation}
where
\begin{equation}
  i (x) = \left( \left[ \frac{Nx_1}{K} \right], \ldots, \left[ \frac{Nx_d}{K}
  \right] \right) \text{ \ and \ } \Delta_i = Ki +\{1, \ldots, K\}^d .
\end{equation}
Hence, unless $x$ is too close to the border of $[0, 1]^d$, $\mathcal{M}_K
(x)$ is the magnetization in a block of side-length $K$ that contains $Nx$.
Theorem 5.7 in {\cite{M01}} provides a strong stochastic control on
$\mathcal{M}_K$ when $\beta > \hat{\beta}_c$. In particular, when $K$ is large
enough, at every $x$ the probability that $\mathcal{M}_K (x)$ is close to
either $m_{\beta}$ or $- m_{\beta}$ is close to one under the averaged measure
$\mathbbm{E} \mu^{J, +}_{\Lambda_N}$. Hence $\mathcal{M}_K / m_{\beta}$
describes the geometry of the phases in the Ising model: it is close to one on
the plus phase region, close to minus one on the minus phase region.

We need a few more notations. To $U \subset \mathbbm{R}^d$ Borel measurable,
we associate the profile
\begin{equation}
  \chi_U : x \in \mathbbm{R}^d \mapsto \left\{ \begin{array}{ll}
    1 & \text{if } x \notin U\\
    - 1 & \text{else}
  \end{array} \label{chi} \right.
\end{equation}
and denote by $\|.\|_{L^1}$ the norm of the $L^1$-space $L^1 \left( [0, 1]^d ;
\mathbbm{R} \right)$. We also consider the set of vectors $z$ such that the
translate $z + U$ fits into $[0, 1]^d$:
\begin{equation}
  \mathcal{T} \left( U \right) = \left\{ z \in \mathbbm{R}^d : z + U \subset
  [0, 1]^d \right\} . \label{T}
\end{equation}
Our second theorem describes the geometrical structure of the two phases when
they coexist:

\begin{theorem}
  \label{thm-shape-phco}Assume that $\beta > \hat{\beta}_c$ and $\beta \notin
  \mathcal{N}$. For all $0 \leqslant \alpha < 1 / \tmop{diam}_{\infty}
  (\mathcal{W}^q)$ and $\varepsilon > 0$, for any $K$ large enough one has
  \begin{equation}
    \lim_{N \rightarrow \infty} \mu^{J, +}_{\Lambda_N} \left( \inf_{z \in
    \mathcal{T}(\alpha \mathcal{W}^q)} \left. \left\|
    \frac{\mathcal{M}_K}{m_{\beta}} - \chi_{z + \alpha \mathcal{W}^q}
    \right\|_{L^1} \leqslant \varepsilon \right|
    \frac{m_{\Lambda_N}}{m_{\beta}} \leqslant 1 - 2 \alpha^d \right) = 1
    \label{eq-shape-phco-q}
  \end{equation}
  in $\mathbbm{P}$-probability ($\mathbbm{P}$-a.s. when $\beta \notin
  \mathcal{N}_I$).
\end{theorem}

Note that, although we state our theorems for the Ising model, they could
easily be adapted to the Potts model with random interactions, or to
random-cluster models, as the two fundamental tools for the study of phase
coexistence, the coarse graining {\cite{M01}} and the study of surface
tension, were developed in the more general setting of the random-cluster
model ($q \geqslant 1$) with random couplings.

The fact that we consider $K$ large but finite is a slight improvement with
respect to former works. In general, one can take any $K = K_N$ such that $1
\ll K_N \ll N$ because on the one hand, $\mathcal{M}_{K_N}$ is close to the
local mean of $\mathcal{M}_K$ as $K_N \gg 1$, and this local mean is close to
$\chi_{z + \alpha \mathcal{W}}$ because the $K_N$-blocks intersecting $N
\partial (z + \alpha \mathcal{W})$ contribute to a negligible volume as $K_N
\ll N$.

Let us conclude this paragraph on a first consequence of the presence of
random couplings : the limit shape of the droplet at low temperatures is
smoother. In the case of the pure Ising model, the Wulff crystal converges to
the unit hypercube $[\pm 1 / 2]^d$ as the temperature goes to zero. Here,
$\mathcal{W}^q$ converges to the Wulff crystal associated with the maximal
flow $\mu$ when $\mathbbm{P}(J_e > 0) = 1$ (Proposition
\ref{prop-Wulff-lowt}). Little is known on the crystal $\mathcal{W}^{\mu}$
associated to the maximal flow $\mu$. Yet, as discussed in Section
\ref{sec-shape-Wmu}, an argument by Durrett and Liggett {\cite{N257}} shows
that $\mathcal{W}^{\mu}$ is not a square when, for instance, $d = 2$,
\[ \mathbbm{P} \left( J_e = 1 / 2 \right) = p \text{ \ and \ } \mathbbm{P}
   \left( J_e = 1 \right) = 1 - p \]
with $\overrightarrow{p_c} < p < 1$, where $\overrightarrow{p_c}$ is the
critical threshold for oriented bond percolation.

\subsection{Phase coexistence under averaged Gibbs measures}

Now we consider the issue of phase coexistence under averaged Gibbs measures,
that is, when phase coexistence is imposed on both the spin configuration and
the random couplings. Before we go further, we would like to remark that
averaged Gibbs measures do not have the physical meaning of the quenched
measure: in quenched ferromagnets, the disorder is frozen and thus cannot be
influenced by the spin configuration itself. However, the analysis presented
here gives an insight on the phenomenon of {\tmem{localization}} which can
occur in models with media randomness.

First we remark that the cost for phase coexistence is here determined by the
surface tension $\tau^{\lambda} (\tmmathbf{n})$:

\begin{theorem}
  \label{thm-cost-lambda}For all $\lambda > 0$ and $0 \leqslant \alpha < 1 /
  \tmop{diam}_{\infty} (\mathcal{W}^{\lambda})$,
  \begin{equation}
    \frac{1}{N^{d - 1}} \log \mathbbm{E} \left[ \left( \mu^{J, +}_{\Lambda_N}
    \left( \frac{m_{\Lambda_N}}{m_{\beta}} \leqslant 1 - 2 \alpha^d \right)
    \right)^{\lambda} \right] \underset{N \rightarrow \infty}{\longrightarrow}
    -\mathcal{F}^{\lambda} (\alpha \mathcal{W}^{\lambda}) .
    \label{eq-cost-phco-l}
  \end{equation}
\end{theorem}

The inequality $\tau^{\lambda} < \lambda \tau^q$ at low temperatures
(Corollary \ref{cor-taulq-lowt}) implies that
\[ \mathcal{F}^{\lambda} (\mathcal{W}^{\lambda}) \leqslant
   \mathcal{F}^{\lambda} (\mathcal{W}^q) < \lambda \mathcal{F}^q
   (\mathcal{W}^q), \]
in other words the cost for phase coexistence is {\tmem{strictly smaller}}
under the averaged Gibbs measure than under the quenched Gibbs measure. One
can go further and analyze the cost for reducing the cost for phase
coexistence under averaged Gibbs measures: the functional
\[ \mathcal{J} \left( f \right) = \sup_{\lambda > 0} \{\mathcal{F}^{\lambda}
   (\mathcal{W}^{\lambda}) - \lambda f\} \in [0, \infty] \text{, \ \ \ } f \in
   \mathbbm{R} \]
is the rate function for lower deviations of the cost for phase coexistence.
If $\mathcal{W}^{\min}$ and $\mathcal{F}^{\min}$ stand respectively for the
Wulff crystal and the surface energy associated to $\tau^{\min}$, then
$\mathcal{J}$ is infinite on the left of $\mathcal{F}^{\min}
(\mathcal{W}^{\min})$, finite on the right of $\mathcal{F}^{\min}
(\mathcal{W}^{\min})$ and zero on the right of $\mathcal{F}^q
(\mathcal{W}^q)$, and:

\begin{corollary}
  For any $f \neq \mathcal{F}^{\min} (\mathcal{W}^{\min})$ and $\alpha
  \geqslant 0$ small enough,
  \[ \lim_N \frac{1}{N^{d - 1}} \log \mathbbm{P} \left( \frac{1}{N^{d - 1}}
     \log \mu^{J, +}_{\Lambda_N} \left( \frac{m_{\Lambda_N}}{m_{\beta}}
     \leqslant 1 - 2 \alpha^d \right) \geqslant - \alpha^{d - 1} f \right) = -
     \alpha^{d - 1} \mathcal{J}(f) . \]
\end{corollary}

Upper deviations for the cost of phase coexistence, on the other hand, happen
at volume order (cf. the proof of Proposition \ref{prop-upb-phco-final}).

The shape of crystals under averaged Gibbs measures is as well determined by
the surface tension $\tau^{\lambda} (\tmmathbf{n})$:

\begin{theorem}
  \label{thm-Wulff-lambda1}For any $0 \leqslant \alpha < 1 /
  \tmop{diam}_{\infty} (\mathcal{W}^{\lambda = 1})$ and $\varepsilon > 0$, for
  any $K$ large enough one has
  \begin{equation}
    \lim_{N \rightarrow \infty} \left( \mathbbm{E} \mu^{J, +}_{\Lambda_N}
    \right) \left( \inf_{z \in \mathcal{T}(\alpha \mathcal{W}^{\lambda = 1})}
    \left. \left\| \frac{\mathcal{M}_K}{m_{\beta}} - \chi_{z + \alpha
    \mathcal{W}^{\lambda = 1}} \right\|_{L^1} \leqslant \varepsilon \right| 
    \frac{m_{\Lambda_N}}{m_{\beta}} \leqslant 1 - 2 \alpha^d \right) = 1.
    \label{eq-shape-phco-l}
  \end{equation}
\end{theorem}

This result extends in fact to all $\lambda > 0$, at the price however of
heavier notations, because $\mathbbm{E}[(\mu^{J, +}_{\Lambda_N}
(.))^{\lambda}]$ is not a measure when $\lambda \neq 1$:

\begin{theorem}
  \label{thm-Wulff-lambda}For any $\lambda > 0$, any $0 \leqslant \alpha < 1 /
  \tmop{diam}_{\infty} (\mathcal{W}^{\lambda})$ and $\varepsilon > 0$, for any
  $K$ large enough one has
  \[ \lim_{N \rightarrow \infty} \frac{\mathbbm{E} \left[ \left( \mu^{J,
     +}_{\Lambda_N} \left( \inf_{z \in \mathcal{T}(\alpha
     \mathcal{W}^{\lambda})} \left. \left\| \frac{\mathcal{M}_K}{m_{\beta}} -
     \chi_{z + \alpha \mathcal{W}^{\lambda}} \right\|_{L^1} \leqslant
     \varepsilon \text{ and } \frac{m_{\Lambda_N}}{m_{\beta}} \leqslant 1 - 2
     \alpha^d \right) \right)^{\lambda} \right] \right.}{\mathbbm{E} \left[
     \left( \mu^{J, +}_{\Lambda_N} \left( \frac{m_{\Lambda_N}}{m_{\beta}}
     \leqslant 1 - 2 \alpha^d \right) \right)^{\lambda} \right]} = 1. \]
\end{theorem}

We conclude the summary of our results with a description of a phenomenon of
{\tmem{localization}}. First, let us characterize the typical value of the
surface tension $\tau^J_{\mathcal{R}}$ under the averaged measure, conditioned
to phase coexistence. For any $\beta \geqslant 0, \lambda > 0$ and
$\tmmathbf{n} \in S^{d - 1}$, this value stands between
\begin{eqnarray}
  \hat{\tau}^{\lambda, -} (\tmmathbf{n}) & = & \inf \left\{ \tau \geqslant 0 :
  I_{\tmmathbf{n}} (\tau) + \lambda \tau = \tau^{\lambda} (\tmmathbf{n})
  \right\} \\
  \text{and \ } \hat{\tau}^{\lambda, +} (\tmmathbf{n}) & = & \sup \left\{ \tau
  \geqslant 0 : I_{\tmmathbf{n}} (\tau) + \lambda \tau = \tau^{\lambda}
  (\tmmathbf{n}) \right\} . 
\end{eqnarray}
The equality $\hat{\tau}^{\lambda, -} (\tmmathbf{n}) = \hat{\tau}^{\lambda, +}
(\tmmathbf{n})$ holds whenever there is at most one $\tau$ at which the slope
of $I_{\tmmathbf{n}}$ equals $\lambda$, that is, for all but at most countably
many values of $\lambda > 0$. Note also that the strict inequality
$\tau^{\lambda} (\tmmathbf{n}) < \lambda \tau^q (\tmmathbf{n})$ implies
$\hat{\tau}^{\lambda, +} (\tmmathbf{n}) < \tau^q (\tmmathbf{n})$. We also
consider a similar quantity for the quenched value of surface tension:
\begin{equation}
  \tilde{\tau}^q (\tmmathbf{n}) = \inf \left\{ \tau : I_{\tmmathbf{n}} (\tau)
  = 0 \right\} \label{eq-def-taut}
\end{equation}
which coincides with $\tau^q (\tmmathbf{n})$ for all $\tmmathbf{n} \in S^{d -
1}$, for all but at most countably many $\beta \geqslant 0$, see Corollary
\ref{cor-NI-count}.

Our last Theorem describes the value of the surface tension $\tau^J$
conditionally on the position of the crystal: we prove that the typical value
of surface tension is $\tau^q$ outside the boundary of the crystal, while on
the boundary it is reduced to $\hat{\tau}^{\lambda}$. When $\tau^{\lambda}
(\tmmathbf{n}) < \lambda \tau^q (\tmmathbf{n})$ for some $\tmmathbf{n} \in
S^{d - 1}$, the location of the Wulff crystal under averaged Gibbs measures is
thus {\tmem{determined}} by the realization of the media: the boundary of the
crystal coincides with the place where surface tension is reduced.

\begin{theorem}
  \label{thm-phco-tauJ}Let $\lambda > 0$, $0 \leqslant \alpha < 1 /
  \tmop{diam}_{\infty} (\mathcal{W}^{\lambda})$ and $z \in \mathcal{T}(\alpha
  \mathcal{W}^{\lambda})$. Consider $h, \delta, \gamma > 0$ and a
  parallelepiped rectangle $\mathcal{R}=\mathcal{R}_{x, h, \delta h}
  (\tmmathbf{n}, \mathcal{S}) \subset (0, 1)^d$ as in (\ref{eq-def-R}). Call
  $\mathcal{R}^N = N\mathcal{R}+ z_N (\mathcal{R})$ where $z_N (\mathcal{R})
  \in (- 1 / 2, 1 / 2]^d$ is chosen such that the center of $\mathcal{R}^N$
  belongs to $\mathbbm{Z}^d$. For $\varepsilon > 0$ small enough and $K$ large
  enough,
  \[ \lim_{N \rightarrow \infty} \frac{\mathbbm{E} \left[ \left( \mu^{J,
     +}_{\Lambda_N} \left( \tau^J_{\mathcal{R}^N} \in \mathcal{A} \text{ and }
     \left\| \frac{\mathcal{M}_K}{m_{\beta}} - \chi_{z + \alpha
     \mathcal{W}^{\lambda}} \right\|_{L^1} \leqslant \varepsilon \right)
     \right)^{\lambda} \right]}{\mathbbm{E} \left[ \left( \mu^{J,
     +}_{\Lambda_N} \left( \left\| \frac{\mathcal{M}_K}{m_{\beta}} - \chi_{z +
     \alpha \mathcal{W}^{\lambda}} \right\|_{L^1} \leqslant \varepsilon
     \right) \right)^{\lambda} \right]} = 1 \]
  when
  \begin{enumerateroman}
    \item $\mathcal{R} \cap z + \alpha \partial \mathcal{W}^{\lambda} =
    \emptyset$ and $\mathcal{A}= \left[ \tilde{\tau}^q (\tmmathbf{n}) -
    \gamma, \tau^q (\tmmathbf{n}) + \gamma \right]$
    
    \item or $x \in z + \alpha \partial \mathcal{W}^{\lambda}$, $\tmmathbf{n}$
    is the outer local normal to $z + \alpha \mathcal{W}^{\lambda}$ at $x$,
    $h$ is small enough and $\mathcal{A}= \left[ \hat{\tau}^{\lambda, -}
    (\tmmathbf{n}) - \gamma, \hat{\tau}^{\lambda, +} (\tmmathbf{n}) + \gamma
    \right]$.
  \end{enumerateroman}
\end{theorem}

\subsection{Acknowledgments}

During the elaboration of this work I enjoyed numerous stimulating discussions
with Thierry Bodineau. Most of the results presented here were obtained during
a PhD Thesis at Universit\'e Paris Diderot {\cite{M00}}. I am also grateful to
Marie Theret and Rapha\"el Rossignol for useful and pleasant discussion about
maximal flows and concentration.

\section{Surface tension}

\label{sec-tau}As announced in the former Section, surface tension is a
fundamental tool for understanding the mechanism of phase coexistence. It
quantifies the free energy per surface unit of an interface separating the
plus and minus phases in the dilute Ising model. In this Section, we prove the
convergence of surface tension in dilute models and study its large
deviations.

\subsection{Sub-additivity and convergence}

\label{sec-ts-add}In many aspects the surface tension for the dilute Ising
model is similar to the one of the Ising model with deterministic couplings.
It has the crucial property of being {\tmem{sub-additive}}, as in the uniform
case {\cite{N08}}: this is shown in Theorem \ref{thm-add-tauJ} below. We
present here the proof of Proposition \ref{prop-tauJ-mon}, Theorem
\ref{thm-add-tauJ} and finally Theorem \ref{thm-conv-tauq}. We also explain
why surface tension is positive under the assumption that $\beta >
\hat{\beta}_c$ (Proposition \ref{prop-tauq-pos}).

\begin{proof}
  (Proposition \ref{prop-tauJ-mon}). The surface tension
  $\tau^J_{\mathcal{R}}$ is a non-decreasing function of $J$ and $\beta$
  because $\mathcal{D}_{\mathcal{R}}$ is a decreasing event while the measure
  $\Phi^J_{\mathcal{R}}$ stochastically increases with $p = 1 - \exp (- \beta
  J_e)$.
  
  Now we consider $H' \geqslant H$ and call $\mathcal{R}=\mathcal{R}_{x, L,
  H} (\mathcal{S}, \tmmathbf{n})$ and $\mathcal{R}' =\mathcal{R}_{x, L, H'}
  (\mathcal{S}, \tmmathbf{n})$. In view of the DLR equation and of the
  monotonicity of $\Phi_{\mathcal{R}}^{J, \pi}$ along $\pi$, the measure
  $\Phi_{\mathcal{R}'}^{J, w}$ restricted to $E ( \hat{\mathcal{R}})$ is
  stochastically smaller than $\Phi_{\mathcal{R}}^{J, w}$. On the other hand,
  it is clear that $\mathcal{D}_{\mathcal{R}} \subset
  \mathcal{D}_{\mathcal{R}'}$, and because $\mathcal{D}_{\mathcal{R}}$ is a
  decreasing event we conclude that
  \[ \Phi_{\mathcal{R}}^{J, w} \left( \mathcal{D}_{\mathcal{R}} \right)
     \leqslant \Phi_{\mathcal{R}'}^{J, w} \left( \mathcal{D}_{\mathcal{R}}
     \right) \leqslant \Phi_{\mathcal{R}'}^{J, w} \left(
     \mathcal{D}_{\mathcal{R}'} \right), \]
  which shows that $\tau^J_{\mathcal{R}}$ is a non-increasing function of $H$.
  
  It is clear from the definition that $\tau^{\min}_{\mathcal{R}} \geqslant
  0$. The inequality $\tau^{\min}_{\mathcal{R}} \leqslant \tau^J_{\mathcal{R}}
  \leqslant \tau^{\max}_{\mathcal{R}}$ is a consequence of the monotony in
  $J$. We conclude with the upper bound on $\tau^{\max}_{\mathcal{R}}$.
  Because of the monotony in $H$ we can take $H = 2 \sqrt{d}$ (which ensures
  that disconnection is still possible). We have: $\tau^{\max}_{\mathcal{R}}
  \leqslant \tau^{\max}_{\mathcal{R}'}$ where $\mathcal{R}' =\mathcal{R}_{x,
  L, 2 \sqrt{d}} (\mathcal{S}, \tmmathbf{n})$. It is enough to close all the
  edges of $\widehat{\mathcal{R}'}$ to realize the disconnection in
  $\widehat{\mathcal{R}'}$. The DLR equation, combined with the monotonicity
  of $\Phi_{\{e\}}^{J^{\max}, \pi}$ along the boundary condition $\pi$ yields:
  \[ \tau^{\max}_{\mathcal{R}'} \leqslant - \frac{1}{L^{d - 1}} \log \prod_{e
     \in E ( \widehat{\mathcal{R}'})} \Phi_{\{e\}}^{J^{\max}, w} \left(
     \{\omega_e = 0\} \right) = \beta J^{\max}  \frac{|E (
     \widehat{\mathcal{R}'}) |}{L^{d - 1}} . \]
  Finally, $|E ( \widehat{\mathcal{R}'}) |$ is not larger than $2 d$ times the
  cardinal of $\widehat{\mathcal{R}'}$, which is itself not larger than the
  volume of $V = \bigcup_{x \in \widehat{\mathcal{R}'}} \left( x + [0, 1]^d
  \right) \subset \mathcal{R}_{0, L + 2 \sqrt{d}, 3 \sqrt{d}} (\mathcal{S},
  \tmmathbf{n})$. Consequently,
  \[ \tau^{J^{\max}}_{\mathcal{R}'} \leqslant \beta J^{\max} \times 2 d \times
     \frac{\left( L + 2 \sqrt{d} \right)^{d - 1} \times 6 \sqrt{d}}{L^{d - 1}}
     \leqslant \beta J^{\max} \times 6 \times 2^d d^{3 / 2} . \]
\end{proof}

Now we address the issue of sub-additivity. It is a fundamental tool not only
for proving the convergence of surface tension, but also for establishing the
large deviations principles in the next Section.

\begin{theorem}
  \label{thm-add-tauJ}Consider $\tmmathbf{n} \in S^{d - 1}$, $\mathcal{S},
  \mathcal{S}' \subset \mathbbm{S}_{\tmmathbf{n}}$ and $H, l \geqslant 2
  \sqrt{d}$, $L \geqslant 4 \sqrt{d} l$. Let $\mathcal{R}=\mathcal{R}_{0, L, H
  + \sqrt{d} / 2} (\mathcal{S}, \tmmathbf{n})$. There is a collection
  $(\mathcal{R}_i)_{i \in \mathcal{C}}$ of rectangular parallelepipeds
  $\mathcal{R}_i =\mathcal{R}_{z_i, l, H} (\mathcal{S}', \tmmathbf{n})$ that
  are disjoint subsets of $\mathcal{R}$, centered at $z_i \in \mathbbm{Z}^d$,
  with
  \begin{equation}
    1 - c_d  \left( \frac{l}{L} + \frac{1}{l} \right) \leqslant \left(
    \frac{l}{L} \right)^{d - 1} |\mathcal{C}| \leqslant 1 \label{eq-card-C}
  \end{equation}
  such that, for any $J : E ( \hat{\mathcal{R}}) \rightarrow [0, 1]$:
  \begin{equation}
    \tau^J_{\mathcal{R}} \leqslant \frac{1}{|\mathcal{C}|} \sum_{i \in
    \mathcal{C}} \tau^J_{\mathcal{R}_i} + \beta c_d \left( \frac{l}{L} +
    \frac{1}{l} \right) \label{eq-tauJ-sa}
  \end{equation}
  where $c_d < \infty$ is a constant that depends on $d$ only.
\end{theorem}

Let us make a few comments on this Theorem. First, a key feature of the
sub-additivity as formulated in Theorem \ref{thm-add-tauJ} is the
\tmtextit{independence} of the $\tau^J_{\mathcal{R}_i}$ under $\mathbbm{P}$
since the $\mathcal{R}_i$ are disjoint. Note that as well, the
$\tau^J_{\mathcal{R}_i}$ have the same law as the $\mathcal{R}_i$ are all
centered at lattice points. Three error terms appear in Theorem
\ref{thm-add-tauJ}. Their origins are as follows (see also Figure
\ref{fig-subadd-tauJ}):
\begin{enumerateroman}
  \item the term $\beta c_d / l$ stands for the cost of disconnection in the
  middle section of $\mathcal{R}$ between adjacent $\mathcal{R}_i$,
  
  \item the term $\beta c_d l / L$ represents the cost of disconnection in the
  area not covered by the $\mathcal{R}_i$
  
  \item and the increase of $H$ by $\sqrt{d} / 2$ for $\mathcal{R}$ with
  respect to the $\mathcal{R}_i$ is a consequence of the requirement that the
  $\mathcal{R}_i$ be all centered at lattice points.
\end{enumerateroman}
The last error term could be avoided for \tmtextit{rational} directions
$\tmmathbf{n} \in S^{d - 1}$, yet (as the two others) it will soon disappear
when we take the limit $H \rightarrow \infty$.

\begin{figure}[h!]
\begin{center}
  \includegraphics[width=8cm]{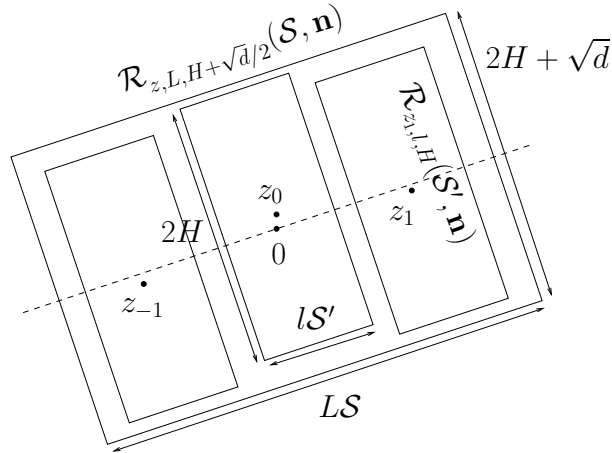}
  \end{center}
  \caption{\label{fig-subadd-tauJ}The rectangular parallelepiped $\mathcal{R}$
  and the collection $(\mathcal{R}_i)_{i \in \mathcal{C}}$ in
  Theorem~\ref{thm-add-tauJ}.}
\end{figure}

The reader will notice that the use of the FK representation permits to give a
relatively short proof of Theorem \ref{thm-add-tauJ}.

\begin{proof}
  (Theorem \ref{thm-add-tauJ}). We begin with the definition of $z_i$ and
  $\mathcal{C}$. We call $(\tmmathbf{e}'_k)_{k = 1 \ldots d - 1}$ the edges of
  $\mathcal{S}'$ and $\tmmathbf{e}'_d =\tmmathbf{n}$, so that
  $(\tmmathbf{e}'_k)_{k = 1 \ldots d}$ is an orthonormal basis of
  $\mathbbm{R}^d$. For all $i = (i_k)_{k = 1 \ldots d - 1} \in \mathbbm{Z}^{d
  - 1}$ we define $z_i$ as the unique point of $\mathbbm{Z}^d$ such that
  \[ \left( l + \sqrt{d} \right) \sum_{k = 1}^{d - 1} i_k \tmmathbf{e}'_k \in
     z_i + \left[ - \frac{1}{2}, \frac{1}{2} \right)^d \]
  and call
  \[ \mathcal{C}= \left\{ i \in \mathbbm{Z}^{d - 1} : \mathcal{R}_i \subset
     \mathcal{R} \right\} \]
  letting
  \[ \mathcal{R}=\mathcal{R}_{0, L, H + \sqrt{d} / 2} (\mathcal{S},
     \tmmathbf{n}) \text{ \ \ and \ \ } \mathcal{R}_i =\mathcal{R}_{z_i, l, H}
     (\mathcal{S}', \tmmathbf{n}) . \]
  We proceed with the proof of (\ref{eq-card-C}) first. We call
  $\mathcal{H}_{\tmmathbf{n}}$ the hyperplane of $\mathbbm{R}^d$ orthogonal to
  $\tmmathbf{n}$ that contains $0$ and remark that the orthogonal projections
  of $z_i + l\mathcal{S}'$ (for all $i \in \mathcal{C}$) on
  $\mathcal{H}_{\tmmathbf{n}}$ are disjoint and all included in
  $L\mathcal{S}$. Hence their total surface $|\mathcal{C}|l^{d - 1}$ does not
  exceed the surface of $L\mathcal{S}$, namely $L^{d - 1}$, and the upper
  bound in (\ref{eq-card-C}) follows. Reusing the previous notations we call
  \[ z_i' = \left( l + \sqrt{d} \right) \sum_{k = 1}^{d - 1} i_k
     \tmmathbf{e}'_k \text{, \ \ \ } \forall i \in \mathbbm{Z}^{d - 1} \]
  so that $z'_i \in \mathcal{H}_{\tmmathbf{n}}$. We consider then
  \[ \mathcal{C}' = \left\{ i \in \mathbbm{Z}^{d - 1} : z'_i + \left( l +
     \sqrt{d} \right) \mathcal{S}' \subset L\mathcal{S} \right\} . \]
  In view of the inequality $d (z_i, z'_i) \leqslant \sqrt{d} / 2$ it follows
  that $z_i + l\mathcal{S}' \subset \mathcal{R}$, for all $i \in
  \mathcal{C}'$, hence $\mathcal{C}' \subset \mathcal{C}$. On the other hand,
  for any $i \in \mathbbm{Z}^{d - 1}$ such that $z'_i + (l +
  \sqrt{d})\mathcal{S}' \cap (L - 2 \sqrt{d} (l + \sqrt{d}))\mathcal{S} \neq
  \emptyset$ we have $i \in \mathcal{C}'$, hence
  \[ \left( \frac{L - 2 \sqrt{d} (l + \sqrt{d})}{l + \sqrt{d}} \right)^{d - 1}
     \leqslant |\mathcal{C}' | \leqslant |\mathcal{C}| \]
  and
  \begin{eqnarray*}
    \left( \frac{l}{L} \right)^{d - 1} |\mathcal{C}| & \geqslant & \left(
    \frac{l}{l + \sqrt{d}} - 2 \sqrt{d}  \frac{l}{L} \right)^{d - 1}\\
    & \geqslant & \left( 1 - \frac{\sqrt{d}}{l} - 2 \sqrt{d}  \frac{l}{L}
    \right)^{d - 1}\\
    & \geqslant & 1 - (d - 1) \left( \frac{\sqrt{d}}{l} + 2 \sqrt{d} 
    \frac{l}{L} \right)
  \end{eqnarray*}
  which yields the lower bound for (\ref{eq-card-C}). We pass now to the proof
  of (\ref{eq-tauJ-sa}) and call
  \[ \mathcal{E}= \left\{ e \in E ( \hat{\mathcal{R}}) \setminus \bigcup_{i
     \in \mathcal{C}} E ( \widehat{\mathcal{R}^{}}_i) : d (e,
     \mathcal{H}_{\tmmathbf{n}}) \leqslant \frac{\sqrt{d}}{2} \right\} \]
  where $d (e, \mathcal{H}_{\tmmathbf{n}})$ stands for the shortest distance
  between one extremity of $e$ and $\mathcal{H}_{\tmmathbf{n}}$. The inclusion
  \[ \left( \bigcap_{i \in \mathcal{C}} \mathcal{D}_{\mathcal{R}_i}^{} \right)
     \bigcap \left\{ \omega_e = 0, \forall e \in \mathcal{E} \right\} \subset
     \mathcal{D}_{\mathcal{R}}^{} \]
  holds: consider $\omega$ that belongs to the left-hand side and let $c$ an
  $\omega$-open path issued from $\partial^+ \hat{\mathcal{R}}$. Every times
  $c$ enters some $\hat{\mathcal{R}}_i$ by the upper boundary $\partial^+
  \hat{\mathcal{R}}$, it also exits by the same upper boundary since $\omega
  \in \mathcal{D}_{\hat{\mathcal{R}}_i}^{}$. As $c$ cannot use the edges of
  $\mathcal{E}$ it is not able to cross the middle hyperplane
  $\mathcal{H}_{\tmmathbf{n}}$ elsewhere than in the $\hat{\mathcal{R}}_i$,
  and in particular it cannot reach $\partial^- \hat{\mathcal{R}}$. Since the
  $\mathcal{D}_{\hat{\mathcal{R}}_i}$ as well as the $\left\{ \omega_e = 0
  \right\}$ are decreasing events, the DLR equations and the monotonicity
  along the boundary condition for $\Phi^J$ imply that
  \begin{eqnarray}
    \Phi^{J, w}_{\hat{\mathcal{R}}} (\mathcal{D}_{\mathcal{R}}) & \geqslant &
    \prod_{i \in \mathcal{C}} \Phi^{J, w}_{\mathcal{R}_i}
    (\mathcal{D}_{\mathcal{R}_i}) \times \prod_{e \in \mathcal{E}} \Phi^{J,
    w}_{\{e\}} (\{\omega_e = 0\}) \nonumber\\
    & \geqslant & \prod_{i \in \mathcal{C}} \Phi^{J, w}_{\mathcal{R}_i}
    (\mathcal{D}_{\mathcal{R}_i}) \times \exp \left( - \beta |\mathcal{E}|
    \right)  \label{ineq-binf-PhiD}
  \end{eqnarray}
  as $\Phi^{J, w}_{\{e\}} (\{\omega_e = 0\}= 1 - p_e = \exp \left( - \beta J_e
  \right) \geqslant \exp (- \beta)$. We proceed then with an estimate over the
  cardinality of $\mathcal{E}$: we call $F = \left\{ x \in \mathbbm{Z}^d :
  \exists y, \{x, y\} \in \mathcal{E} \right\}$ the set of extremities of some
  $e \in \mathcal{E}$ and remark that $|\mathcal{E}| \leqslant d \tmop{Vol}
  \left( V \right)$ where $V = \bigcup_{x \in F} x + [0, 1]^d$. We have
  \[ V \subset \mathcal{R}_{0, L + 2 \sqrt{d}, 3 \sqrt{d} / 2} (\mathcal{S},
     \tmmathbf{n}) \text{ \ \ while \ \ } V \cap \mathcal{R}_{z_i, l - 2
     \sqrt{d}, \infty} (\mathcal{S}', \tmmathbf{n}) = \emptyset \text{, \ }
     \forall i \in \mathcal{C}, \]
  hence
  \[ \left| \mathcal{E} \right| \leqslant d \times \frac{3 \sqrt{d}}{2} \times
     \left(  \left( L + 2 \sqrt{d} \right)^{d - 1} - |\mathcal{C}| \left( l -
     2 \sqrt{d} \right)^{d - 1} \right) \leqslant c_d L^{d - 1}  \left(
     \frac{l}{L} + \frac{1}{l} \right) \]
  in view of the lower bound in (\ref{eq-card-C}). Taking logarithms in
  (\ref{ineq-binf-PhiD}) and dividing by $- L^{d - 1}$ we obtain the
  inequality
  \[ \tau^J_{\mathcal{R}} \leqslant \left( \frac{l}{L} \right)^{d - 1} \sum_{i
     \in \mathcal{C}} \tau^J_{\mathcal{R}_i} + c_d \beta \left( \frac{l}{L} +
     \frac{1}{l} \right) \]
  and (\ref{eq-tauJ-sa}) follows from the upper bound in (\ref{eq-card-C}).
  
  We conclude with a word on the structure of the sequence
  $(\tau^J_{\mathcal{R}_i})_{i \in \mathcal{C}}$. The $\mathcal{R}_i$ are
  disjoint by construction, hence so are the edge sets $E (
  \hat{\mathcal{R}}_i)$, hence the $\tau^J_{\mathcal{R}_i}$ are independent.
  They are identically distributed as the $\mathcal{R}_i$ are all centered at
  lattice points, $\mathbbm{P}$ being translation invariant as a product
  measure.
\end{proof}

Now we establish the convergence for surface tension and prove Theorem
\ref{thm-conv-tauq}. The proof of this Theorem is based on the sub-additivity
of surface tension. We do not apply directly Kingman's sub-additive Theorem
{\cite{N299}} as we want to show that $\tau^q$ does not depend on
$\mathcal{S}$, nor on $\delta$.

\begin{proof}
  Taking the expectation $\mathbbm{E}$ in the sub-additivity inequality
  (\ref{eq-tauJ-sa}) we get
  \[ \mathbbm{E} \tau_{\mathcal{R}_{0, L, H + \sqrt{d} / 2} (\mathcal{S},
     \tmmathbf{n})}^J \leqslant \mathbbm{E} \tau_{\mathcal{R}_{0, l, H}
     (\mathcal{S}', \tmmathbf{n})}^J + \beta c_d \left( \frac{l}{L} +
     \frac{1}{l} \right) . \]
  Applying $\limsup_{L \rightarrow \infty}$, then $\liminf_{l \rightarrow
  \infty}$ and taking the decreasing limit in $H$ we obtain
  \[ \lim_{H \rightarrow \infty} \limsup_{L \rightarrow \infty} \mathbbm{E}
     \tau_{\mathcal{R}_{0, L, H} (\mathcal{S}, \tmmathbf{n})}^J \leqslant
     \lim_{H \rightarrow \infty} \liminf_{L \rightarrow \infty} \mathbbm{E}
     \tau_{\mathcal{R}_{0, L, H} (\mathcal{S}', \tmmathbf{n})}^J \]
  which proves that
  \begin{equation}
    \tau^q (\tmmathbf{n}) = \lim_{H \rightarrow \infty} \liminf_{L \rightarrow
    \infty} \mathbbm{E} \tau_{\mathcal{R}_{0, L, H} (\mathcal{S},
    \tmmathbf{n})}^J = \lim_{H \rightarrow \infty} \limsup_{L \rightarrow
    \infty} \mathbbm{E} \tau_{\mathcal{R}_{0, L, H} (\mathcal{S},
    \tmmathbf{n})}^J \label{eq-def-tauq-LH}
  \end{equation}
  exists and does not depend on $\mathcal{S} \in \mathbbm{S}_n$.
  
  We prove now the convergence $\tau_{\mathcal{R}^N}^J \rightarrow \tau^q
  (\tmmathbf{n})$ in $\mathbbm{P}$-probability, where $\mathcal{R}^N
  =\mathcal{R}_{0, N, \delta N} (\mathcal{S}, \tmmathbf{n})$. The
  sub-additivity (\ref{eq-tauJ-sa}) yields: for any $\delta > 0$ and $N$ large
  enough,
  \[ \tau_{\mathcal{R}^N}^J \leqslant \tau_{\mathcal{R}_{0, N, H + \sqrt{d} /
     2} (\mathcal{S}, \tmmathbf{n})}^J \leqslant \frac{1}{|\mathcal{C}|}
     \sum_{i \in \mathcal{C}} \tau^J_{\mathcal{R}_{z_i, L, H}} + \beta c_d
     \left( \frac{L}{N} + \frac{1}{L} \right) \]
  Taking $\limsup_{N \rightarrow \infty}$ and applying the strong law of large
  numbers give:
  \[ \limsup_{N \rightarrow \infty} \tau_{\mathcal{R}^N}^J \leqslant
     \mathbbm{E} \tau_{\mathcal{R}_{0, L, H} (\mathcal{S}, \tmmathbf{n})}^J +
     \frac{\beta c_d }{L} \text{ \ \ \ } \mathbbm{P} \text{-a.s.} \]
  and after $\liminf_{L \rightarrow \infty}$ and $\lim_{H \rightarrow \infty}$
  we see that, for all $\mathcal{S} \in \mathbbm{S}_{\tmmathbf{n}}$ and
  $\delta > 0$,
  \begin{equation}
    \limsup_{N \rightarrow \infty} \tau_{\mathcal{R}^N}^J \leqslant \tau^q
    (\tmmathbf{n}) \text{ \ \ \ } \mathbbm{P} \text{-a.s.}
    \label{eq-upb-tauJ-tauq}
  \end{equation}
  On the other hand, the sub-additivity (\ref{eq-tauJ-sa}) is also responsible
  for the convergence of $\mathbbm{E} \tau_{\mathcal{R}^N}^J$: remark that
  \[ \mathbbm{E} \tau_{\mathcal{R}_{0, L, \delta N + \sqrt{d} / 2}
     (\mathcal{S}, \tmmathbf{n})}^J \leqslant \mathbbm{E}
     \tau_{\mathcal{R}^N}^J + \beta c_d \left( \frac{N}{L} + \frac{1}{N}
     \right), \]
  hence $\limsup_{L \rightarrow \infty}$ followed by $\liminf_{N \rightarrow
  \infty}$ give:
  \begin{equation}
    \tau^q (\tmmathbf{n}) \leqslant \liminf_{N \rightarrow \infty} \mathbbm{E}
    \tau_{\mathcal{R}^N}^J . \label{eq-binf-EtauJ}
  \end{equation}
  Together with (\ref{eq-upb-tauJ-tauq}) and (\ref{eq-binf-EtauJ}), the
  boundedness of $\tau_{\mathcal{R}^N}^J$ ensures the convergence in
  probability.
\end{proof}

Let us sketch now a proof of Proposition \ref{prop-tauq-pos}, namely that the
quenched surface tension $\tau^q (\tmmathbf{n})$ is \tmtextit{positive} for
any $\beta > \hat{\beta}_c$: thanks to the renormalization argument of
{\cite{M01}}, one can compare the surface tension $\tau^a = \tau^{\lambda =
1}$ under the averaged Gibbs measure to the surface tension of high density
site percolation, which is positive. The claim follows as $\tau^q \geqslant
\tau^a$ by Jensen's inequality.

\subsection{Upper large deviations}

Due to the presence of the random couplings, surface tension can
\tmtextit{fluctuate} around its typical value. The sub-additivity permits to
study the order of the cost of large deviations. First, we examine upper
deviations and prove Theorem \ref{thm-updev-tauJ}. The proof is based on the
following argument: we split $\mathcal{R}_{0, N, \delta N} (\mathcal{S},
\tmmathbf{n})$ into $cN$ rectangular parallelepipeds $\mathcal{R}_i$ with
finite height $H$. In order to increase $\tau^J_{\mathcal{R}_{0, N, \delta N}
(\mathcal{S}, \tmmathbf{n})}$ one has to increase surface tension in each
$\mathcal{R}_i$, but the cost of increasing one $\tau^J_{\mathcal{R}_i}$ is
already of surface order by sub-additivity.

\begin{proof}
  (Theorem \ref{thm-updev-tauJ}). As a first step towards the proof we
  estimate the cost for upper deviations of surface tension in a rectangular
  parallelepiped of fixed height, using the sub-additivity of $\tau^J$. From
  the definition of $\tau^q (\tmmathbf{n})$ at (\ref{eq-def-tauq-LH}) it
  follows that for any $H$ large enough,
  \[ \limsup_L \mathbbm{E} \tau^J_{\mathcal{R}_{0, L, H} (\mathcal{S},
     \tmmathbf{n})} \leqslant \tau^q (\tmmathbf{n}) + \frac{\varepsilon}{6} .
  \]
  Given such an $H$ we fix $l$ large enough such that $\mathbbm{E}
  \tau^J_{\mathcal{R}_{0, l, H} (\mathcal{S}, \tmmathbf{n})} \leqslant \tau^q
  (\tmmathbf{n}) + \varepsilon / 3$ and $c_d \beta / l \leqslant \varepsilon /
  4$, where $c_d$ refers to the constant in the sub-additivity equation. With
  the notations of Theorem \ref{thm-add-tauJ} we have:
  \begin{equation}
    \tau^J_{\mathcal{R}_{0, L, H + \sqrt{d} / 2} (\mathcal{S}, \tmmathbf{n})}
    \leqslant \frac{1}{|\mathcal{C}|} \sum_{i \in \mathcal{C}}
    \tau^J_{\mathcal{R}_{z_i, l, H} (\mathcal{S}, \tmmathbf{n})} +
    \frac{\varepsilon}{4} + \beta c_d \frac{l}{L} \label{eq-subadd-tauJ-eps}
  \end{equation}
  and the $\tau^J_{\mathcal{R}_{z_i, l, H} (\mathcal{S}, \tmmathbf{n})}$ are
  i.i.d. variables of mean not larger than $\tau^q (\tmmathbf{n}) +
  \varepsilon / 3$. Hence, Cram\'er's Theorem tells that
  \[ \mathbbm{P} \left( \frac{1}{|\mathcal{C}|} \sum_{i \in \mathcal{C}}
     \tau^J_{\mathcal{R}_{z_i, l, H} (\mathcal{S}, \tmmathbf{n})} \geqslant
     \tau^q (\tmmathbf{n}) + \frac{\varepsilon}{2} \right) \leqslant \exp (-
     c|\mathcal{C}|) \]
  for some $c > 0$. Reporting in (\ref{eq-subadd-tauJ-eps}) proves that for
  any $\varepsilon > 0$, for any $H$ large enough:
  \begin{equation}
    \limsup_{L \rightarrow \infty} \frac{1}{L^{d - 1}} \log \mathbbm{P} \left(
    \tau^J_{\mathcal{R}_{0, L, H} (\mathcal{S}, \tmmathbf{n})} \geqslant
    \tau^q (\tmmathbf{n}) + \varepsilon \right) < 0 \label{eq-cost-inc-tauJLH}
  \end{equation}
  -- that is, the cost for increasing $\tau^J_{\mathcal{R}_{0, L, H}
  (\mathcal{S}, \tmmathbf{n})}$ is of surface order. We fix such an $H$ and
  decompose now the rectangular parallelepiped $\mathcal{R}=\mathcal{R}_{0, N,
  \delta N} (\mathcal{S}, \tmmathbf{n})$ in the direction $\tmmathbf{n}$.
  Precisely, we let
  \[ \tilde{x}_i = 2 \left( H + \frac{\sqrt{d}}{2} \right) i\tmmathbf{n}
     \text{, \ \ \ } \forall i \in \mathbbm{Z} \text{ \ \ and \ \ }
     \tilde{\mathcal{R}}_i =\mathcal{R}_{\tilde{x}_i, N, H + \sqrt{d} / 2}
     (\mathcal{S}, \tmmathbf{n}) . \]
  We call $\mathcal{G}$ the set of $i \in \mathbbm{Z}$ such that
  $\tilde{\mathcal{R}}_i \subset \mathcal{R}$ and consider, for all $i \in
  \mathcal{G}$, $x_i$ the point of $\mathbbm{Z}^d$ such that $\tilde{x}_i \in
  x_i + [- 1 / 2, 1 / 2)^d$ and let
  \[ \mathcal{R}_i =\mathcal{R}_{x_i, N - \sqrt{d}, H} (\mathcal{S},
     \tmmathbf{n}) . \]
  The rectangular parallelepipeds $\mathcal{R}_i$ are disjoint subsets of
  $\mathcal{R}=\mathcal{R}_{0, N, \delta N} (\mathcal{S}, \tmmathbf{n})$, all
  centered at lattice points. Furthermore, if we call
  $\mathcal{E}_{\tmop{lat}}$ the set of edges in $E ( \hat{\mathcal{R}})$ with
  one extremity at distance at most $\sqrt{d}$ from the lateral boundary of
  $\mathcal{R}$, we have:
  \[ \omega \in \bigcup_{i \in \mathcal{G}} \mathcal{D}_{\mathcal{R}_i} \text{
     \ and \ } \omega_e = 0, \forall e \in \mathcal{E}_{\tmop{lat}} \text{ \ }
     \Rightarrow \text{ \ } \omega \in \mathcal{D}_{\mathcal{R}} . \]
  Hence the DLR equation yields:
  \begin{eqnarray*}
    \Phi^{J, w}_{\mathcal{R}} \left( \mathcal{D}_{\mathcal{R}} \right) &
    \geqslant & \max_{i \in \mathcal{G}} \Phi^{J, w}_{\mathcal{R}} \left(
    \omega_e = 0, \forall e \in \mathcal{E}_{\tmop{lat}} \text{ \ and \ }
    \omega \in \mathcal{D}_{\mathcal{R}_i} \right)\\
    & \geqslant & e^{- \beta |\mathcal{E}_{\tmop{lat}} |} \times \max_{i \in
    \mathcal{G}} \Phi^{J, w}_{\mathcal{R}_i} \left( \omega \in
    \mathcal{D}_{\mathcal{R}_i} \right) .
  \end{eqnarray*}
  As $|\mathcal{E}_{\tmop{lat}} | \leqslant c_d \delta N^{d - 1}$ we conclude
  finally to the inequality
  \begin{eqnarray}
    \tau^J_{\mathcal{R}} & \leqslant & c_d \delta \beta + \min_{i \in
    \mathcal{G}} \tau^J_{\mathcal{R}_i} .  \label{eq-tauJ-inf-min}
  \end{eqnarray}
  Inequality (\ref{eq-tauJ-inf-min}) states that in order to increase
  significantly $\tau^J_{\mathcal{R}}$, one must increase each
  $\tau^J_{\mathcal{R}_i}$. Yet, the cost for increasing one of the
  $\tau^J_{\mathcal{R}_i}$ is of surface order (\ref{eq-cost-inc-tauJLH}), and
  the $\tau^J_{\mathcal{R}_i}$ are independent variables. Hence for any
  $\delta > 0$ such that $c_d \delta \beta < \varepsilon$,
  \[ \limsup_{N \rightarrow \infty} \frac{1}{N^d} \log \mathbbm{P} \left(
     \tau^J_{\mathcal{R}_{0, N, \delta N} (\mathcal{S}, \tmmathbf{n})}
     \geqslant \tau^q (\tmmathbf{n}) + 2 \varepsilon \right) < 0. \]
  As $\tau^J_{\mathcal{R}_{0, N, \delta N} (\mathcal{S}, \tmmathbf{n})}$
  decreases with $\delta$, the claim follows for arbitrary $\delta > 0$.
\end{proof}

\subsection{Lower large deviations}

\label{sec-tau-LD}Contrary to upper deviations, lower large deviations occur
at surface order. Here we consider the rate function $I_{\tmmathbf{n}}$ for
lower large deviations. The fact that deviations occur at the same order as
the disconnecting event defining surface tension is responsible for the
distinct behavior of surface tension under quenched and averaged measures.
Explicit bounds on the rate function $I_{\tmmathbf{n}}$ will be derived in
Sections \ref{sec-conc-lowt} and \ref{sec-conc-gal}.

\begin{proof}
  We begin with the definition of the rate function $I_{\mathcal{R}}$ in a
  rectangular parallelepiped $\mathcal{R}=\mathcal{R}_{0, L, H} (\mathcal{S},
  \tmmathbf{n})$ as the surface cost for reducing $\tau^J_{\mathcal{R}}$ to
  $\tau$:
  \[ I_{\mathcal{R}} \left( \tau \right) = - \frac{1}{L^{d - 1}} \log
     \mathbbm{P} \left( \tau^J_{\mathcal{R}} \leqslant \tau \right) . \]
  According to Proposition \ref{prop-tauJ-mon}, $I_{\mathcal{R}_{0, L, H}
  (\mathcal{S}, \tmmathbf{n})} \left( \tau \right)$ is a non-increasing
  function of $\tau$ and $H$. Hence the limit
  \begin{equation}
    I_{(\mathcal{S}, \tmmathbf{n})} \left( \tau \right) = \lim_{\varepsilon
    \rightarrow 0^+} \inf_H \limsup_L I_{\mathcal{R}_{0, L, H} (\mathcal{S},
    \tmmathbf{n})} \left( \tau + \varepsilon \right) \in [0, \infty]
    \label{eq-def-In}
  \end{equation}
  exists -- we introduce the parameter $\varepsilon > 0$ in order to
  compensate for the error terms in (\ref{eq-tauJ-sa}). It is clearly a
  non-increasing function of $\tau$. We prove now that it is also convex in
  $\tau$ and that it does not depend on $\mathcal{S} \in
  \mathbbm{S}_{\tmmathbf{n}}$: let $\mathcal{S}' \in \mathbbm{S}_n$,
  $\varepsilon > 0$ and $\alpha \in [0, 1]$. Using the notations
  $\mathcal{R}=\mathcal{R}_{0, L, H + \sqrt{d} / 2} (\mathcal{S},
  \tmmathbf{n})$, $\mathcal{R}_i =\mathcal{R}_{z_i, l, H} (\mathcal{S}',
  \tmmathbf{n})$ and $\mathcal{C}$ of the sub-additivity Theorem (Theorem
  \ref{thm-add-tauJ}), we have
  \[ \tau^J_{\mathcal{R}} \leqslant \frac{|\mathcal{C}^1 |}{|\mathcal{C}|}
     \tau^1 + \frac{|\mathcal{C}^2 |}{|\mathcal{C}|} \tau^2 + \varepsilon +
     \beta c_d  \left( \frac{l}{L} + \frac{1}{l} \right) \]
  if $\mathcal{C}^1 \sqcup \mathcal{C}^2$ is a partition of $\mathcal{C}$ such
  that
  \begin{equation}
    \tau^J_{\mathcal{R}_i} \leqslant \left\{ \begin{array}{ll}
      \tau^1 + \varepsilon & \text{if } i \in \mathcal{C}^1\\
      \tau^2 + \varepsilon & \text{if } i \in \mathcal{C}^2 .
    \end{array} \label{eq-cond-In-conv} \right.
  \end{equation}
  The probability for realizing condition (\ref{eq-cond-In-conv}) equals
  \[ \exp \left( - |\mathcal{C}^1 |l^{d - 1} I_{\mathcal{R}_{0, l, H}
     (\mathcal{S}', \tmmathbf{n})} \left( \tau^1 + \varepsilon \right) -
     |\mathcal{C}^2 |l^{d - 1} I_{\mathcal{R}_{0, l, H} (\mathcal{S}',
     \tmmathbf{n})} \left( \tau^2 + \varepsilon \right) \right) \]
  and letting $|\mathcal{C}^1 | / |\mathcal{C}| \rightarrow \alpha$ and $L
  \rightarrow \infty$ we see that
  \begin{equation}
    \begin{array}{l}
      \limsup_L I_{\mathcal{R}_{0, L, H + \sqrt{d} / 2} (\mathcal{S},
      \tmmathbf{n})} \left( \alpha \tau^1 + (1 - \alpha) \tau^2 + 2
      \varepsilon + \beta c_d / l \right) \leqslant\\
      \alpha I_{\mathcal{R}_{0, l, H} (\mathcal{S}', \tmmathbf{n})} \left(
      \tau^1 + \varepsilon \right) + (1 - \alpha) I_{\mathcal{R}_{0, l, H}
      (\mathcal{S}', \tmmathbf{n})} \left( \tau^2 + \varepsilon \right) .
    \end{array} \label{eq-conv-ISn}
  \end{equation}
  Taking the superior limit in $l$, then the limit in $H$, then $\varepsilon
  \rightarrow 0^+$ we obtain
  \[ I_{(\mathcal{S}, \tmmathbf{n})} \left( \alpha \tau^1 + (1 - \alpha)
     \tau^2 \right) \leqslant \alpha I_{(\mathcal{S}', \tmmathbf{n})} \left(
     \tau^1 \right) + (1 - \alpha) I_{(\mathcal{S}', \tmmathbf{n})} \left(
     \tau^2 \right) \]
  which proves both the independence of $I_{(\mathcal{S}, \tmmathbf{n})}$ with
  respect to $\mathcal{S}$ (take $\alpha = 1$) and the convexity along $\tau$.
  We let now $I_{\tmmathbf{n}} = I_{(\mathcal{S}, \tmmathbf{n})}$ and postpone
  the proof of (\ref{eq-conv-Itaud}) for a while. The continuity of
  $I_{\tmmathbf{n}}$ on the interior of the domain of finiteness of
  $I_{\tmmathbf{n}}$ is a consequence of its convexity. Hence we examine the
  domain of finiteness of $I_{\tmmathbf{n}}$. Let first $\tau < \tau^{\min}
  (\tmmathbf{n})$. If $\varepsilon > 0$ is small enough, the event
  $\tau^J_{\mathcal{R}_{0, L, H} (\mathcal{S}, \tmmathbf{n})} \leqslant \tau +
  \varepsilon < \tau^{\min} (\tmmathbf{n})$ has a probability zero and
  consequently, $I_{\tmmathbf{n}} (\tau) = + \infty$. The second easy regime
  is $\tau \geqslant \tau^q (\tmmathbf{n})$: from Proposition
  \ref{thm-updev-tauJ} we infer that $\lim_{L \rightarrow \infty}
  \mathbbm{P}(\tau^J_{\mathcal{R}_{0, L, H} (\mathcal{S}, \tmmathbf{n})}
  \leqslant \tau + \varepsilon) = 1$ provided that $H$ is large enough and
  this implies $I_{\tmmathbf{n}} (\tau) = 0$. If at last $\tau > \tau^{\min}
  (\tmmathbf{n})$, there is $H$ such that
  \[ \limsup_L \tau^{J^{\min}}_{\mathcal{R}_{0, L, H} (\mathcal{S},
     \tmmathbf{n})} < \tau . \]
  We will prove that, for $\delta > 0$ small enough we still have:
  \begin{equation}
    \limsup_L \tau^{J^{\min} + \delta}_{\mathcal{R}_{0, L, H} (\mathcal{S},
    \tmmathbf{n})} < \tau \label{eq-tau-Jmind} .
  \end{equation}
  If we let $\mathcal{R}=\mathcal{R}_{0, L, H} (\mathcal{S}, \tmmathbf{n})$
  and differentiate along $\delta$, we obtain
  \begin{eqnarray*}
    \frac{\partial \tau^{J^{\min} + \delta}_{\mathcal{R}}}{\partial \delta} &
    = & \sum_{e \in E ( \hat{\mathcal{R}})} \left. \frac{\partial
    \tau^J_{\mathcal{R}}}{\partial J_e} \right|_{J = J^{\min} + \delta}
  \end{eqnarray*}
  yet, (\ref{eq-def-aeJ}) and Proposition \ref{prop-ctrl-aeJ} indicate that
  for any $J \in \mathcal{J}$,
  \[ \frac{L^{d - 1}}{\beta}  \frac{\partial \tau^J_{\mathcal{R}}}{\partial
     J_e} \leqslant 1. \]
  As a consequence, $\tau^{J^{\min} + \delta}_{\mathcal{R}}$ is a $c_d \beta
  H$-Lipschitz function of $\delta$. The same is true for $\limsup_L
  \tau^{J^{\min} + \delta}_{\mathcal{R}_{0, L, H} (\mathcal{S},
  \tmmathbf{n})}$, thus (\ref{eq-tau-Jmind}) holds true for $\delta > 0$ small
  enough. Now we write, for any $L$ large enough:
  \begin{eqnarray*}
    I_{\mathcal{R}_{0, L, H} (\mathcal{S}, \tmmathbf{n})} \left( \tau \right)
    & = & - \frac{1}{L^{d - 1}} \log \mathbbm{P} \left(
    \tau^J_{\mathcal{R}_{0, L, H} (\mathcal{S}, \tmmathbf{n})} \leqslant \tau
    \right)\\
    & \leqslant & - \frac{1}{L^{d - 1}} \log \mathbbm{P} \left( J_e \leqslant
    J^{\min} + \delta, \text{ \ } \forall e \in E \left( \hat{\mathcal{R}}_{0,
    L, H} (\mathcal{S}, \tmmathbf{n}) \right) \right)\\
    & \leqslant & c_d H \times (- \log \mathbbm{P}(J_e \in [J^{\min},
    J^{\min} + \delta]))
  \end{eqnarray*}
  which is finite thanks to the definition of $J^{\min}$. This ends the proof
  that $I_{\tmmathbf{n}} (\tau) < \infty$, for any $\tau > \tau^{\min}
  (\tmmathbf{n})$.
  
  We address at last the convergence (\ref{eq-conv-Itaud}). The inequality
  $I_{\mathcal{R}_{0, N, \delta N} (\mathcal{S}, \tmmathbf{n})} \left( \tau
  \right) \leqslant I_{\mathcal{R}_{0, N, H} (\mathcal{S}, \tmmathbf{n})}
  \left( \tau \right)$ when $N \delta \geqslant H$ yields an upper bound on
  the superior limit:
  \[ \limsup_N I_{\mathcal{R}_{0, N, \delta N} (\mathcal{S}, \tmmathbf{n})}
     \left( \tau \right) \leqslant \inf_H \limsup_L I_{\mathcal{R}_{0, L, H}
     (\mathcal{S}, \tmmathbf{n})} \left( \tau \right) \leqslant
     I_{\tmmathbf{n}} (\tau^-) = I_{\tmmathbf{n}} (\tau) \]
  for all $\tau > \tau^{\min} (\tmmathbf{n})$, thanks to the continuity of
  $I_{\tmmathbf{n}}$. For the lower bound we use the sub-additivity of surface
  tension. Applying (\ref{eq-conv-ISn}) with $\alpha = 1$, $l = N$, $H =
  \delta N$ yields: for any $\varepsilon > 0$ and $N$ large enough,
  \[ \limsup_L I_{\mathcal{R}_{0, L, \delta N + \sqrt{d} / 2} (\mathcal{S},
     \tmmathbf{n})} \left( \tau + 3 \varepsilon \right) \leqslant
     I_{\mathcal{R}_{0, N, \delta N} (\mathcal{S}, \tmmathbf{n})} \left( \tau
     + \varepsilon \right) \]
  and replacing $\tau + \varepsilon$ with $\tau$, we obtain after the limits
  $N \rightarrow \infty$ and $\varepsilon \rightarrow 0^+$ the lower bound
  \[ I_{\tmmathbf{n}} (\tau) \leqslant \liminf_N I_{\mathcal{R}_{0, N, \delta
     N} (\mathcal{S}, \tmmathbf{n})} \left( \tau \right) \text{, \ \ \ }
     \forall \tau \in \mathbbm{R}. \]
\end{proof}

\subsection{Surface tension under averaged Gibbs measures}

The rate function $I_{\tmmathbf{n}}$ can be analyzed through a dual quantity:
the surface tension under the averaged Gibbs measure defined at
(\ref{eq-dual-Itaul}). The duality of Fenchel-Legendre transforms for convex
functions (Lemma 4.5.8 in {\cite{N71}}) implies that $\lambda \mapsto
\tau^{\lambda} (\tmmathbf{n})$ is concave and that
\begin{equation}
  I_{\tmmathbf{n}} (\tau) = \sup_{\lambda > 0} \{\tau^{\lambda} (\tmmathbf{n})
  - \lambda \tau\} \label{eq-rec-Itau} .
\end{equation}
As we said at (\ref{eq-cvg-taul}), $\tau^{\lambda} (\tmmathbf{n})$ can be
interpreted as the surface tension under an average of $\Phi^{J,
w}_{\mathcal{R}}$. Indeed, if we let
\begin{equation}
  \tau^{\lambda}_{\mathcal{R}} = - \frac{1}{L^{d - 1}} \log \mathbbm{E} \left(
  \left[ \Phi^{J, w}_{\mathcal{R}} \left( \mathcal{D}_{\mathcal{R}} \right)
  \right]^{\lambda} \right) = - \frac{1}{L^{d - 1}} \log \mathbbm{E} \left(
  \exp \left( - \lambda L^{d - 1} \tau^J_{\mathcal{R}} \right) \right),
  \label{eq-def-taul-R}
\end{equation}
for any rectangular parallelepiped $\mathcal{R}$ of side-length $L$ as in
(\ref{eq-def-R}), then Varadhan's Lemma yields:

\begin{proposition}
  \label{prop-conv-taul}For any $\lambda > 0$ and $\tmmathbf{n} \in S^{d -
  1}$, for any sequence of rectangular parallelepipeds $\mathcal{R}^N
  =\mathcal{R}_{0, N, \delta N} (\mathcal{S}, \tmmathbf{n})$ with $\delta > 0$
  and $\mathcal{S} \in \mathbbm{S}_{\tmmathbf{n}}$, the quantity
  $\tau^{\lambda}_{\mathcal{R}^N}$ converges to $\tau^{\lambda}
  (\tmmathbf{n})$:
  \begin{equation}
    \lim_N \tau^{\lambda}_{\mathcal{R}^N} = \tau^{\lambda} (\tmmathbf{n})
    \label{eq-taul-st} .
  \end{equation}
  Thus, the limit does not depend on $\delta > 0$ nor on $\mathcal{S} \in
  \mathbbm{S}_n$.
\end{proposition}

We defined at (\ref{eq-def-taut}) the value $\tilde{\tau}^q (\tmmathbf{n})$ of
the surface tension at which $I_{\tmmathbf{n}} (\tau)$ becomes zero. Below are
some immediate consequences of the definition of $\tau^{\lambda}
(\tmmathbf{n})$ at (\ref{eq-dual-Itaul}) together with (\ref{eq-taul-st}),
which allow to sketch the graph of $\lambda \mapsto \tau^{\lambda}
(\tmmathbf{n})$ on Figure \ref{fig-gr-taul}:

\begin{proposition}
  \label{prop-taul}The following inequalities hold:
  \begin{equation}
    \lambda \tau^{\min} (\tmmathbf{n}) \leqslant \tau^{\lambda} (\tmmathbf{n})
    \leqslant \lambda \tilde{\tau}^q (\tmmathbf{n}) \label{ineq-taul} \text{,
    \ \ } \forall \tmmathbf{n} \in S^{d - 1}, \lambda > 0
  \end{equation}
  while:
  \begin{equation}
    \frac{\tau^{\lambda} (\tmmathbf{n})}{\lambda}  \underset{\lambda
    \rightarrow 0^+}{\longrightarrow}  \tilde{\tau}^q (\tmmathbf{n}) \text{ \
    \ and \ \ } \frac{\tau^{\lambda} (\tmmathbf{n})}{\lambda} 
    \underset{\lambda \rightarrow + \infty}{\longrightarrow} \tau^{\min}
    (\tmmathbf{n}) \label{eq-prop-taul} \text{, \ \ \ } \forall \tmmathbf{n}
    \in S^{d - 1} .
  \end{equation}
  Hence, $\tau^{\lambda} (\tmmathbf{n})$ is positive if and only if
  $\tilde{\tau}^q (\tmmathbf{n}) > 0$. Furthermore:
  \begin{equation}
    \tau^{\lambda} (\tmmathbf{n}) \underset{\lambda \rightarrow +
    \infty}{\longrightarrow} \lim_{\tau \rightarrow 0^+} I_{\tmmathbf{n}}
    (\tau) \in [0, \infty] .
  \end{equation}
\end{proposition}

\begin{figure}[h!]
  \begin{center}
  \includegraphics[width=8cm]{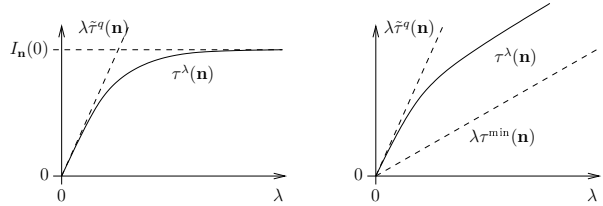}
  \end{center}
  \caption{\label{fig-gr-taul}The graph of $\lambda \mapsto \tau^{\lambda}
  (\tmmathbf{n})$ in the case of dilution ($\tau^{\min} = 0$ and
  $I_{\tmmathbf{n}} (0) < \infty$, left) and distributions with $\tau^{\min} >
  0$ (right).}
\end{figure}

Another important yet classical fact is the \tmtextit{convexity} of surface
tension {\cite{N08}}. The proposition below is a consequence of the weak
triangle inequality for $\tau^J_{\mathcal{R}}$ (see {\cite{N08}} or
{\cite{N70}} for the uniform case, or Appendix 2.5.2 in {\cite{M00}}).

\begin{proposition}
  \label{prop-conv-tau}Let $f^q$ be the homogeneous extension of $\tau^q$ to
  $\mathbbm{R}^d$, namely:
  \[ f^q (x) = \left\{ \begin{array}{ll}
       \|x\| \tau^q (x /\|x\|) & \text{if } x \in \mathbbm{R}^d \setminus
       \{0\}\\
       0 & \text{if } x = 0,
     \end{array} \right. \]
  and let $f^{\lambda}$ (resp. $\tilde{f}^q$) be the homogeneous extension of
  $\tau^{\lambda}$ (resp. $\tilde{\tau}^q$) to $\mathbbm{R}^d$. Then, $f^q$,
  $f^{\lambda}$ and $\tilde{f}^q$ are convex and $\tau^q$, $\tau^{\lambda}$
  and $\tilde{\tau}^q$ are continuous on $S^{d - 1}$.
\end{proposition}

\subsection{Concentration at low temperatures}

\label{sec-conc-lowt}In this Section and the next one we establish
respectively Theorems \ref{thm-Iquad-lowt} and \ref{thm-Iquad-gal}. In both
cases we use concentration of measure theory, which is a very efficient tool
for analyzing the fluctuations of product measures. In the case of polymers or
even spin glasses it yields relevant bounds on the probabilities of
deviations, see {\cite{N216}} for a review. Concerning the Ising (or
random-cluster) model with random couplings, its application to the deviations
of surface tension requires a control over the surface of the interface, and
this is the point where the proofs of Theorems \ref{thm-Iquad-lowt} and
\ref{thm-Iquad-gal} differ: at low temperatures one can control rather easily
the length of the interface, while under the only assumptions of Theorem
\ref{thm-Iquad-gal} the same control is not immediate.

The surface tension $\tau^{\lambda} (\tmmathbf{n})$ under averaged Gibbs
measure plays an important role here, as well as the modified measure
$\mathbbm{E}_{\lambda}$ defined at (\ref{eq-def-Elambda}) below. We will
obtain lower bounds on $\tau^{\lambda} (\tmmathbf{n})$, which correspond to
lower bounds on $I_{\tmmathbf{n}} (\tau)$ by (\ref{eq-rec-Itau}).

Rather than making the assumption that the product measure $\mathbbm{P}$
satisfies a logarithmic Sobolev inequality as in {\cite{M00}}{\footnote{Usual
measures such as dilution $\mathbbm{P}(J_e \in \{0, 1\}) = 1$, or $J_e$ with
positive density on $[0, 1]$ do satisfy a logarithmic Sobolev inequality, cf.
{\cite{N216}} or Theorems 4.2, 6.6 and Section 6.3 in {\cite{N187}}.}}, we use
general bounds on product measure (Corollary 5.8 in {\cite{N216}}). The author
thanks Rapha\"el Rossignol for pointing out this improvement. The proof of
Theorem \ref{thm-Iquad-lowt} is made of four steps, the first three being
common with the proof of Theorem \ref{thm-Iquad-gal}.

The first step consists in relating the derivative of the surface tension
$\tau^{\lambda}_{\mathcal{R}} (\tmmathbf{n})$ in a rectangular parallelepiped
$\mathcal{R}$ as in (\ref{eq-def-R}), with a basis of side-length $L$, to the
entropy of the positive function $\exp (f_{\lambda})$ where
\[ f_{\lambda} = - \lambda L^{d - 1} \tau_{\mathcal{R}}^J . \]
We recall that the \tmtextit{entropy} of a positive measurable function $f$
with $\mathbbm{E}(f \log (1 + f)) < \infty$ is
\begin{equation}
  \tmop{Ent}_{\mathbbm{P}} (f) =\mathbbm{E}(f \log f) -\mathbbm{E}(f) \log
  \mathbbm{E}(f) . \label{eq-def-Ent}
\end{equation}
With these notations, it is immediate that:

\begin{lemma}
  \label{lem-dtau-ent}For any $\lambda > 0$,
  \begin{eqnarray}
    - \frac{\partial}{\partial \lambda} \left(
    \frac{\tau^{\lambda}_{\mathcal{R}}}{\lambda} \right) & = &
    \frac{1}{\lambda^2 L^{d - 1}}  \frac{\tmop{Ent}_{\mathbbm{P}} (\exp
    (f_{\lambda}))}{\mathbbm{E} \left( \exp \left( f_{\lambda}) \right)
    \right.} 
  \end{eqnarray}
\end{lemma}

As a second step we study the quantity
\begin{equation}
  a_e^J = \frac{L^{d - 1}}{\beta}  \frac{\partial
  \tau^J_{\mathcal{R}}}{\partial J_e} \label{eq-def-aeJ} .
\end{equation}
The proposition below provides an interpretation of $a_e^J$ as the probability
that the disconnecting interface due to the event $\mathcal{D}_{\mathcal{R}}$
passes through the edge $e$. We prove also, and this is crucial for our
construction, that the actual value of $J_e$ does not influence too much that
of $a_e^J$:

\begin{proposition}
  \label{prop-ctrl-aeJ}For any $e$, $a_e^J$ is a $\mathcal{C}^{\infty}$
  function of the $J_{e'}$. For any $J \in [0, 1]^{E (\mathbbm{Z}^d)}$, one
  has
  \begin{equation}
    a_e^J = \frac{1}{p_e} \left( \Phi^{J, w}_{\mathcal{R}} (\omega_e) -
    \Phi^{J, w}_{\mathcal{R}} (\omega_e |\mathcal{D}_{\mathcal{R}}) \right)
    \text{ \ \ if \ } J_e > 0
  \end{equation}
  together with the following inequalities:
  \begin{equation}
    0 \leqslant a_e^J \leqslant 1 \text{ \ \ \ and \ \ \ } \sup_{J_e} a^J_e
    \leqslant e^{\beta} \inf_{J_e} a^J_e . \label{eq-ctrl-aeJ}
  \end{equation}
\end{proposition}

The controls (\ref{eq-ctrl-aeJ}), together with Corollary 5.8 in
{\cite{N216}}, permit to establish the third step. Given a rectangular
parallelepiped $\mathcal{R}$ as in (\ref{eq-def-R}) and $\lambda \geqslant 0$,
we introduce the probability measure $\mathbbm{P}_{\lambda}$ that to any
bounded measurable $h : J \mapsto h (J) \in \mathbbm{R}$ gives expectation
\begin{equation}
  \mathbbm{E}_{\lambda} (h (J)) =\mathbbm{E} \left( h (J) \frac{\exp \left( -
  \lambda L^{d - 1} \tau^J_{\mathcal{R}} \right)}{\mathbbm{E} \exp \left( -
  \lambda L^{d - 1} \tau^J_{\mathcal{R}} \right)} \right) .
  \label{eq-def-Elambda}
\end{equation}
\begin{proposition}
  \label{prop-upb-ent}Denote $m_{\mathbbm{P}} =\mathbbm{E}(J_e)$. For any
  $\lambda \geqslant 0$, we have both
  \begin{eqnarray}
    \frac{\tmop{Ent}_{\mathbbm{P}} (\exp (f_{\lambda}))}{\mathbbm{E} \left(
    \exp \left( f_{\lambda}) \right) \right.} & \leqslant & \lambda^2 
    \frac{\beta^2 e^{\beta (1 + \lambda)}}{4}  \left\{ \begin{array}{l}
      \mathbbm{E}_{\lambda} \left(  \sum_{e \in E ( \hat{\mathcal{R}})} a_e^J
      \right)\\
      \frac{1}{m_{\beta}} L^{d - 1}  \frac{1}{\lambda}  \frac{\partial
      \tau^{\lambda}_{\mathcal{R}}}{\partial \beta} .
    \end{array} \label{eq-lwb-dtaul}  \right.
  \end{eqnarray}
\end{proposition}

The second majoration leads to Theorem \ref{thm-Iquad-gal}, while the first
one yields Theorem \ref{thm-Iquad-lowt} after a last step: using Peierls'
argument we show that, in the Ising model ($q = 2$) with couplings $J_e
\geqslant \varepsilon > 0$, the length of the interface is of order $N^{d -
1}$.

\begin{proposition}
  \label{prop-lint-Ising}Let $q = 2$ and $\varepsilon > 0$. There exists $c_d
  < \infty$ such that, for $\beta$ large enough, for
  $\mathcal{R}=\mathcal{R}^N =\mathcal{R}_{0, N, \delta N} (\mathcal{S},
  \tmmathbf{n})$ with $\delta \in (0, 1)$, for any realization $J$ of the
  random couplings such that $J_e \geqslant \varepsilon$ and $N$ large enough,
  \begin{eqnarray}
    \sum_{e \in E ( \widehat{\mathcal{R}^N})} a_e^J & \leqslant & \frac{c_d
    }{\varepsilon} N^{d - 1} . 
  \end{eqnarray}
\end{proposition}

We give now the proofs of all the propositions, followed by that of Theorem
\ref{thm-Iquad-lowt}.

\begin{proof}
  (Proposition \ref{prop-ctrl-aeJ}). The fact that $a_e^J$ is a
  $\mathcal{C}^{\infty}$ function of $J_{e'}$ is a consequence of the same
  property for $\tau^J_{\mathcal{R}}$, the quantity $\Phi^{J, w}_{\mathcal{R}}
  (\mathcal{D}_{\mathcal{R}})$ being always positive. We introduce next a few
  notations: we let
  \begin{equation}
    w_{\mathcal{R}}^J (\omega) = \prod_{e \in E ( \hat{\mathcal{R}})} \left(
    \frac{p_e}{1 - p_e} \right)^{\omega_e} q^{C_{E ( \hat{\mathcal{R}})}^w
    (\omega)}  \text{ \ \ and \ \ } Z_{\mathcal{R}}^J (\mathcal{A}) =
    \sum_{\omega \in \mathcal{A}} w_{\mathcal{R}}^J (\omega)
    \label{eq-def-phiw}
  \end{equation}
  for any $\omega \in \Omega_{E ( \hat{\mathcal{R}})}$ and $\mathcal{A}
  \subset \Omega_{E ( \hat{\mathcal{R}})}$, see (\ref{eq-def-FK}) for the
  definition of $C_{E ( \hat{\mathcal{R}})}^w (\omega)$. For all $J$ with $J_e
  > 0$, we have
  \[ \frac{\partial \log w_{\mathcal{R}}^J (\omega)}{\partial J_e} = \beta
     \frac{\omega_e}{p_e} \]
  and as a consequence, for all $J$ with $J_e > 0$,
  \begin{eqnarray*}
    a_e^J & = & - \frac{1}{\beta}  \frac{\partial}{\partial J_e} \log
    \frac{Z^J_{\mathcal{R}} (\mathcal{D}_{\mathcal{R}})}{Z^J_{\mathcal{R}}
    (\Omega_{E ( \hat{\mathcal{R}})})}\\
    & = & \frac{1}{p_e} \left( \Phi^{J, w}_{\mathcal{R}} (\omega_e) -
    \Phi^{J, w}_{\mathcal{R}} (\omega_e |\mathcal{D}_{\mathcal{R}}) \right) .
  \end{eqnarray*}
  Under this formulation, the FKG inequality and the bound $\Phi^{J,
  w}_{\mathcal{R}} \left( \omega_e \right) \leqslant p_e$ imply that $0
  \leqslant a_e^J \leqslant 1$ for any $J \in \mathcal{J}$ with $J_e > 0$, and
  the inequality extends by continuity to the whole of $\mathcal{J}$. We now
  calculate the derivative of $a^J_e$ along $J_e$ for $J_e > 0$ and obtain, as
  \begin{eqnarray*}
    \frac{\partial}{\partial J_e} \left[ \frac{\Phi^{J, w}_{\mathcal{R}}
    \left( \omega_e |\mathcal{A} \right)}{p_e} \right] & = & \beta \left[
    \frac{\Phi^{J, w}_{\mathcal{R}} \left( \omega_e |\mathcal{A} \right)}{p_e}
    - \frac{\Phi^{J, w}_{\mathcal{R}} \left( \omega_e |\mathcal{A}
    \right)^2}{p_e^2} \right],
  \end{eqnarray*}
  that, for any $J \in \mathcal{J} \tmop{with} J_e > 0$,
  \[ \frac{\partial a_e^J}{\partial J_e} = \beta a_e^J \left( 1 -
     \frac{\Phi^{J, w}_{\mathcal{R}} \left( \omega_e \right)}{p_e} -
     \frac{\Phi^{J, w}_{\mathcal{R}} \left( \omega_e
     |\mathcal{D}_{\mathcal{R}} \right)}{p_e} \right) . \]
  This implies in particular that
  \[ \left| \frac{\partial a_e^J}{\partial J_e} \right| \leqslant \beta a_e^J
  \]
  and the comparison $\sup_{J_e \in [0, 1]} a_e^J \leqslant e^{\beta}
  \inf_{J_e \in [0, 1]} a_e^J$ follows.
\end{proof}

\begin{proof}
  (Proposition \ref{prop-upb-ent}). According to Corollary 5.8 in
  {\cite{N216}} and to the Mean Value Theorem, we have
  \begin{eqnarray*}
    \frac{\tmop{Ent}_{\mathbbm{P}} (\exp (f_{\lambda}))}{\mathbbm{E} \left(
    \exp \left( f_{\lambda}) \right) \right.} & \leqslant & \frac{1}{4} 
    \sum_{e \in E ( \hat{\mathcal{R}})} \frac{\mathbbm{E} \left( \left(
    \sup_{J_e \in [0, 1]} \frac{\partial f_{\lambda}}{\partial J_e} \right)^2
    \exp (\sup_{J_e \in [0, 1]} f_{\lambda}) \right)}{\mathbbm{E} \left( \exp
    \left( f_{\lambda}) \right) \right.} .
  \end{eqnarray*}
  It is clear that
  \[ \frac{\partial f_{\lambda}}{\partial J_e} = - \lambda \beta a_e^J . \]
  On the other hand, Proposition \ref{prop-ctrl-aeJ} yields
  \[ \sup_{J_e \in [0, 1]} (a_e^J)^2 \leqslant \sup_{J_e \in [0, 1]} a_e^J
     \leqslant e^{\beta} \inf_{J_e \in [0, 1]} a_e^J \]
  and
  \[ \sup_{J_e \in [0, 1]} \exp (f_{\lambda}) \leqslant e^{\beta \lambda}
     \inf_{J_e \in [0, 1]} \exp (f_{\lambda}), \]
  hence
  \begin{eqnarray*}
    \frac{\tmop{Ent}_{\mathbbm{P}} (\exp (f_{\lambda}))}{\mathbbm{E} \left(
    \exp \left( f_{\lambda}) \right) \right.} & \leqslant & \frac{\lambda^2
    \beta^2 e^{\beta (1 + \lambda)}}{4}  \sum_{e \in E ( \hat{\mathcal{R}})}
    \mathbbm{E} \left( \inf_{J_e \in [0, 1]} a_e^J \times \frac{\inf_{J_e \in
    [0, 1]} \exp (f_{\lambda})}{\mathbbm{E} \left( \exp \left( f_{\lambda})
    \right) \right.} \right)
  \end{eqnarray*}
  and the first bound follows. For the second one, remark that as we take
  infimums over $J_e$ we in fact obtain a quantity that is
  \tmtextit{independent} of $J_e$. Thus
  \begin{eqnarray*}
    \frac{\tmop{Ent}_{\mathbbm{P}} (\exp (f_{\lambda}))}{\mathbbm{E} \left(
    \exp \left( f_{\lambda}) \right) \right.} & \leqslant & \frac{\lambda^2
    \beta^2 e^{\beta (1 + \lambda)}}{4}  \sum_{e \in E ( \hat{\mathcal{R}})}
    \mathbbm{E} \left( \frac{J_e}{m_{\mathbbm{P}}} \inf_{J_e \in [0, 1]} a_e^J
    \times \frac{\inf_{J_e \in [0, 1]} \exp (f_{\lambda})}{\mathbbm{E} \left(
    \exp \left( f_{\lambda}) \right) \right.} \right)\\
    & \leqslant & \frac{\lambda^2 \beta^2 e^{\beta (1 + \lambda)}}{4
    m_{\mathbbm{P}}} \mathbbm{E}_{\lambda} \left( \sum_{e \in E (
    \hat{\mathcal{R}})} J_e a_e^J \right)
  \end{eqnarray*}
  which ends the proof as
  \begin{eqnarray*}
    \frac{1}{\lambda}  \frac{\partial \tau^{\lambda}_{\mathcal{R}}}{\partial
    \beta} & = & \frac{1}{L^{d - 1}} \mathbbm{E}_{\lambda} \left( \sum_{e \in
    E ( \hat{\mathcal{R}})} J_e a_e^J \right) .
  \end{eqnarray*}
\end{proof}

\begin{proof}
  (Proposition \ref{prop-lint-Ising}). As $\mathcal{R}=\mathcal{R}_N
  =\mathcal{R}_{0, N, \delta N} (\mathcal{S}, \tmmathbf{n})$ is centered at
  the origin, we consider
  \begin{eqnarray*}
    \Sigma_{\mathcal{R}}^+ & = & \left\{ \sigma : \mathbbm{Z}^d \rightarrow
    \left\{ \pm 1 \right\} : \sigma_x = 1, \forall x \notin \hat{\mathcal{R}}
    \setminus \partial \hat{\mathcal{R}} \right\}\\
    \Sigma_{\mathcal{R}}^{\pm} & = & \left\{ \sigma : \mathbbm{Z}^d
    \rightarrow \left\{ \pm 1 \right\} : \sigma_x = \left\{ \begin{array}{ll}
      1 & \text{if } x \cdot \tmmathbf{n} \geqslant 0\\
      - 1 & \text{else}
    \end{array} \right., \forall x \notin \hat{\mathcal{R}} \setminus \partial
    \hat{\mathcal{R}} \right\}
  \end{eqnarray*}
  the set of spin configurations on $\hat{\mathcal{R}}$ with plus or mixed
  boundary conditions. The correspondence between the random-cluster
  representation (with $q = 2$) and Ising model gives
  \[ \tau^J_{\mathcal{R}} = \frac{1}{N^{d - 1}} \log \frac{Z^{J,
     +}_{\mathcal{R}}}{Z^{J, \pm}_{\mathcal{R}}} \]
  where $Z^{J, +}_{\mathcal{R}}$ and $Z^{J, \pm}_{\mathcal{R}}$ are the
  partition functions
  \begin{eqnarray*}
    Z^{J, +}_{\mathcal{R}} & = & \sum_{\sigma \in \Sigma^+_{\mathcal{R}}} \exp
    \left( \frac{\beta}{2} \sum_{e =\{x, y\} \in E ( \hat{\mathcal{R}})} J_e
    \sigma_x \sigma_y \right)\\
    \text{and \ \ } Z^{J, \pm}_{\mathcal{R}} & = & \sum_{\sigma \in
    \Sigma^{\pm}_{\mathcal{R}}} \exp \left( \frac{\beta}{2} \sum_{e =\{x, y\}
    \in E ( \hat{\mathcal{R}})} J_e \sigma_x \sigma_y \right),
  \end{eqnarray*}
  leading thus to
  \begin{equation}
    a_e^J = \mu^{J, +}_{\mathcal{R}} (\sigma_x \sigma_y) - \mu^{J,
    \pm}_{\mathcal{R}} (\sigma_x \sigma_y) \label{eq-aeJ-mu} \text{, \ }
    \forall e =\{x, y\} \in E ( \hat{\mathcal{R}})
  \end{equation}
  where $\mu^{J, \pm}_{\mathcal{R}}$ is the Ising model on $\hat{\mathcal{R}}$
  with mixed boundary condition (plus on $\partial^+ \hat{\mathcal{R}}$, minus
  on $\partial^- \hat{\mathcal{R}}$). We consider now an interface $I$ for
  $\mathcal{R}$ as in Section \ref{sec-intro-lowt}. We recall that it is a
  minimal set of edges such that connections from $\partial^+
  \hat{\mathcal{R}}$ to $\partial^- \hat{\mathcal{R}}$ through $E (
  \hat{\mathcal{R}}) \setminus I$ are impossible. We consider $I^+$ the upper
  part of the interface $I$:
  \[ I^+ =\{x : \exists y \in \mathbbm{Z}^d : \{x, y\} \in I \text{ \ and \ }
     x \nleftrightarrow \partial^- \hat{\mathcal{R}} \text{ in } E (
     \hat{\mathcal{R}}) \setminus I\} \]
  and define symmetrically the set $I^-$. We call then $\mathcal{S}_I$ the
  event that $I$ is the spin interface between $\partial^+ \hat{\mathcal{R}}$
  and $\partial^- \hat{\mathcal{R}}$ under the measure $\mu^{J,
  \pm}_{\mathcal{R}}$:
  \[ \mathcal{S}_I = \left\{ \sigma \in \Sigma^{\pm}_{\mathcal{R}} :
     \begin{array}{l}
       \sigma (x) = + 1, \forall x \in I^+\\
       \sigma (x) = - 1, \forall x \in I^-
     \end{array} \right\} . \]
  Conditionally on $\mathcal{S}_I$, the restriction of $\mu^{J,
  \pm}_{\mathcal{R}}$ to the upper (resp. lower) parts of $\hat{\mathcal{R}}$
  equals the Ising measure with uniform plus (resp. minus) boundary condition.
  Hence, for any $\{x, y\} \notin I$ we have
  \begin{equation}
    \mu^{J, \pm}_{\mathcal{R}} (\sigma_x \sigma_y |\mathcal{S}_I) \geqslant
    \mu^{J, +}_{\mathcal{R}} (\sigma_x \sigma_y), \label{eq-cond-mu-sxy}
  \end{equation}
  and consequently
  \begin{eqnarray}
    \sum_{e \in E ( \hat{\mathcal{R}})} a_e^J & \leqslant & \sum_{I \text{
    interface}} \mu^{J, \pm}_{\mathcal{R}} (\mathcal{S}_I) \times 2| I|. 
  \end{eqnarray}
  Thus it remains only to bound the average interface length, in the Ising
  sense, under $\mu^{J, \pm}_{\mathcal{R}}$. We remark that $\mu^{J,
  \pm}_{\mathcal{R}} (\mathcal{S}_I)$ can also be written as
  \begin{eqnarray*}
    \mu^{J, \pm}_{\mathcal{R}} (\mathcal{S}_I) & = & \frac{Z^{J,
    +}_{\mathcal{R} \setminus I} \exp \left( - \beta \sum_{e \in \Gamma} J_e
    \right)}{Z^{J, \pm}_{\mathcal{R}} }
  \end{eqnarray*}
  where $Z^{J, +}_{\mathcal{R} \setminus I}$ stands for the partition function
  associated to the set of configurations with plus boundary condition on
  $I^+$, $I^-$ and on $\partial \hat{\mathcal{R}}$. Thanks to the assumption
  $J_e \geqslant \varepsilon$ and to the remarks that
  \begin{eqnarray*}
    Z^{J, +}_{\mathcal{R} \setminus I} & \leqslant & Z^{J, +}_{\mathcal{R}}\\
    \text{and \ } Z^{J, \pm}_{\mathcal{R}} & \geqslant & Z^{J,
    +}_{\mathcal{R}} \exp \left( - \beta | \partial^- \hat{\mathcal{R}} |
    \right),
  \end{eqnarray*}
  we have
  \begin{eqnarray*}
    \mu^{J, \pm}_{\mathcal{R}} (\mathcal{S}_I) & \leqslant & \exp \left( -
    \beta \varepsilon |I| + \beta c_d N^{d - 1} \right)
  \end{eqnarray*}
  as $\delta < 1$. We conclude with a Peierls estimate and bound the number of
  interfaces of cardinal $n \geqslant 2 c_d N^{d - 1} / \varepsilon$ by
  $(c_d)^n$:
  \begin{eqnarray*}
    \sum_{e \in E ( \hat{\mathcal{R}})} a_e^J & \leqslant & \frac{2 c_d N^{d -
    1}}{\varepsilon} + \sum_{n \geqslant 2 c_d N^{d - 1} / \varepsilon} n
    (c_d)^n e^{- \beta \varepsilon n + \beta c_d N^{d - 1}}
  \end{eqnarray*}
  The second term goes to $0$ with $N \rightarrow \infty$ for $\beta$ large
  enough.
\end{proof}

\begin{proof}
  (Theorem \ref{thm-Iquad-lowt}). The combination of Lemma \ref{lem-dtau-ent},
  Propositions \ref{prop-upb-ent} and \ref{prop-lint-Ising} implies that in
  the setting of Theorem \ref{thm-Iquad-lowt},
  \begin{eqnarray*}
    - \frac{\partial}{\partial \lambda} \left(
    \frac{\tau^{\lambda}_{\mathcal{R}}}{\lambda} \right) & \leqslant &
    \frac{c_d \beta^2 e^{\beta (1 + \lambda)}}{4 J^{\min}} 
  \end{eqnarray*}
  for $\beta$ large enough, $\mathcal{R}=\mathcal{R}^N =\mathcal{R}_{0, N,
  \delta N} (\mathcal{S}, \tmmathbf{n})$, $\delta \in (0, 1)$ and $N$ large
  enough. Integrating over $\lambda$ we obtain, as $\lim_{\lambda \rightarrow
  0^+} \tau^{\lambda}_{\mathcal{R}} / \lambda =\mathbbm{E}
  \tau^J_{\mathcal{R}}$, the inequality
  \begin{eqnarray*}
    \tau^{\lambda}_{\mathcal{R}} & \geqslant & \lambda \mathbbm{E}
    \tau^J_{\mathcal{R}} - \lambda^2  \frac{c_d \beta^2 e^{\beta (1 +
    \lambda)}}{4 J^{\min}} .
  \end{eqnarray*}
  Letting $N \rightarrow \infty$ gives
  \begin{eqnarray*}
    \tau^{\lambda} (\tmmathbf{n}) & \geqslant & \lambda \tau^q (\tmmathbf{n})
    - \lambda^2  \frac{c_d \beta^2 e^{\beta (1 + \lambda)}}{4 J^{\min}}
  \end{eqnarray*}
  and the duality formula (\ref{eq-rec-Itau}) yields the claim with $c = c_d
  J^{\min} / (\beta^2 \exp (2 \beta))$, for large enough $\beta$.
\end{proof}

\subsection{Concentration in a general setting}

\label{sec-conc-gal}We give now the proof of Theorem \ref{thm-Iquad-gal},
which is based on Herbst's argument, together with the controls of Lemma
\ref{lem-dtau-ent} and Proposition \ref{prop-upb-ent}. We will then give the
proof of Corollary \ref{cor-NI-count}.

First we give an immediate consequence of the duality formula
(\ref{eq-rec-Itau}):

\begin{lemma}
  \label{lem-dual-as-Itau}Assume that
  \begin{equation}
    \limsup_{\lambda \rightarrow 0^+}  \frac{\tau^{\lambda} (\tmmathbf{n}) -
    \lambda \tau^q (\tmmathbf{n})}{\lambda^2} \geqslant - c \text{ \ \ for
    some } c \in [0, \infty] . \label{eq-asymp-taul-na}
  \end{equation}
  Then,
  \begin{equation}
    \limsup_{r \rightarrow 0^+} \frac{I_{\tmmathbf{n}} (\tau^q (\tmmathbf{n})
    - r)}{r^2} \geqslant \frac{1}{4 c} \label{eq-asymp-In} \in [0, \infty] .
  \end{equation}
\end{lemma}

\begin{proof}
  (Theorem \ref{thm-Iquad-gal}). Given $\delta > 0$ and $\mathcal{S} \in
  \mathbbm{S}_{\tmmathbf{n}}$, we denote $\mathcal{R}^N$ the rectangular
  parallelepiped $\mathcal{R}^N =\mathcal{R}_{0, N, \delta N} (\tmmathbf{n},
  \mathcal{S})$ and introduce
  \[ K_{\tmmathbf{n}}^{\mathbbm{P}, \beta} = \liminf_{\lambda \rightarrow 0^+}
     \liminf_{N \rightarrow \infty} \frac{1}{\lambda} \int_0^{\lambda}
     \frac{\partial \tau_{\mathcal{R}^N}^{\lambda'}}{\partial \beta}  \frac{d
     \lambda'}{\lambda'} \label{eq-def-KPB} \in [0, \infty] \]
  In view of Theorem \ref{thm-conv-tauq} and Proposition \ref{prop-conv-taul}
  we have
  \begin{eqnarray*}
    \tau^{\lambda} (\tmmathbf{n}) - \lambda \tau^q (\tmmathbf{n}) & = &
    \lim_{N \rightarrow \infty} \tau^{\lambda}_{\mathcal{R}^N} - \lambda
    \mathbbm{E} \tau^J_{\mathcal{R}^N}\\
    & = & \lim_{N \rightarrow \infty} \lambda \int_0^{\lambda}
    \frac{\partial}{\partial \lambda'} \left(
    \frac{\tau^{\lambda'}_{\mathcal{R}^N}}{\lambda'} \right) d \lambda'
  \end{eqnarray*}
  as $\mathbbm{E} \tau^J_{\mathcal{R}^N} = \lim_{\lambda \rightarrow 0^+}
  \tau^{\lambda}_{\mathcal{R}^N} / \lambda$ for any $N$ finite. Lemma
  \ref{lem-dtau-ent} and Proposition \ref{prop-upb-ent} yield, for any
  $\varepsilon > 0$:
  \begin{eqnarray*}
    \limsup_{\lambda \rightarrow 0^+}  \frac{\tau^{\lambda} (\tmmathbf{n}) -
    \lambda \tau^q (\tmmathbf{n})}{\lambda^2} & \geqslant & - \frac{\beta^2
    e^{\beta (1 + \varepsilon)}}{4 m_{\mathbbm{P}}} \liminf_{\lambda
    \rightarrow 0^+} \liminf_{N \rightarrow \infty}  \frac{1}{\lambda}
    \int_0^{\lambda} \frac{\partial
    \tau_{\mathcal{R}^N_{}}^{\lambda'}}{\partial \beta} \frac{d
    \lambda'}{\lambda'}\\
    & = & - \frac{\beta^2 e^{\beta (1 + \varepsilon)}}{4 m_{\mathbbm{P}}}
    K_{\tmmathbf{n}}^{\mathbbm{P}, \beta}
  \end{eqnarray*}
  and an immediate application of Lemma \ref{lem-dual-as-Itau} gives, after
  the limit $\varepsilon \rightarrow 0$, the lower bound:
  \begin{equation}
    \limsup_{r \rightarrow 0^+} \frac{I_{\beta, \tmmathbf{n}} (\tau^q_{\beta}
    (\tmmathbf{n}) - r)}{r^2} \geqslant \frac{m_{\mathbbm{P}}}{\beta^2
    e^{\beta} K_{\tmmathbf{n}}^{\mathbbm{P}, \beta}} .
  \end{equation}
  The lower bound is positive when $K_{\tmmathbf{n}}^{\mathbbm{P}, \beta} <
  \infty$. In order to show that this is the case for Lebesgue almost all
  $\beta$, we evaluate the integral of $K_{\tmmathbf{n}}^{\mathbbm{P}, \beta}$
  on some interval $[\beta_1, \beta_2]$. For any $\delta > 0$ and $\mathcal{S}
  \in \mathbbm{S}_{\tmmathbf{n}}$, Fatou's Lemma and Fubini Theorem imply that
  \begin{eqnarray*}
    \int_{\beta_1}^{\beta_2} K_{\tmmathbf{n}}^{\mathbbm{P}, \beta} d \beta &
    \leqslant & \liminf_{\lambda \rightarrow 0^+} \liminf_{N \rightarrow
    \infty} \frac{1}{\lambda} \int_0^{\lambda} \int_{\beta_1}^{\beta_2}
    \frac{\partial \tau_{\mathcal{R}_N}^{\lambda'}}{\partial \beta}  \frac{d
    \lambda'}{\lambda'}\\
    & = & \liminf_{\lambda \rightarrow 0^+} \liminf_{N \rightarrow \infty}
    \frac{1}{\lambda} \int_0^{\lambda} \frac{\tau_{\beta_2,
    \mathcal{R}_N}^{\lambda'} - \tau_{\beta_1,
    \mathcal{R}_N}^{\lambda'}}{\lambda'} d \lambda' .
  \end{eqnarray*}
  The convergence as $N \rightarrow \infty$ is uniformly dominated (by
  Jensen's inequality and Proposition \ref{prop-tauJ-mon}, $0 \leqslant
  \tau_{\mathcal{R}_N}^{\lambda} \leqslant \lambda c_d \beta$) hence we
  finally obtain
  \begin{eqnarray}
    \int_{\beta_1}^{\beta_2} K_{\tmmathbf{n}}^{\mathbbm{P}, \beta} d \beta &
    \leqslant & \liminf_{\lambda \rightarrow 0^+} \frac{1}{\lambda}
    \int_0^{\lambda} \frac{\tau_{\beta_2}^{\lambda'} (\tmmathbf{n}) -
    \tau_{\beta_1}^{\lambda'} (\tmmathbf{n})}{\lambda'} d \lambda' \nonumber\\
    & = & \tilde{\tau}^q_{\beta_2} (\tmmathbf{n}) - \tilde{\tau}^q_{\beta_1}
    (\tmmathbf{n}) .  \label{eq-intK}
  \end{eqnarray}
  in view of (\ref{eq-prop-taul}). In particular,
  $K_{\tmmathbf{n}}^{\mathbbm{P}, \beta}$ is \tmtextit{finite} for Lebesgue
  almost all $\beta \geqslant 0$.
\end{proof}

We would like to make a remark on $K_{\tmmathbf{n}}^{\mathbbm{P}, \beta}$. In
view of Corollary \ref{cor-NI-count}, for Lebesgue almost every $\beta_1,
\beta_2$ with $\beta_1 \leqslant \beta_2$ one can replace
$\tilde{\tau}^q_{\beta_2} (\tmmathbf{n}) - \tilde{\tau}^q_{\beta_1}
(\tmmathbf{n})$ in (\ref{eq-intK}) with $\tau^q_{\beta_2} (\tmmathbf{n}) -
\tau^q_{\beta_1} (\tmmathbf{n})$. As a consequence, whenever $\tau^q_{\beta}
(\tmmathbf{n})$ is derivable on some interval, $K_{\tmmathbf{n}}^{\mathbbm{P},
\beta} \leqslant \partial \tau^q_{\beta} (\tmmathbf{n}) / \partial \beta$ for
Lebesgue almost every $\beta$ in that interval.

\begin{proof}
  (Corollary \ref{cor-NI-count}). We denote by
  \[ \tau^q_{\beta^-} (\tmmathbf{n}) = \lim_{\varepsilon \rightarrow 0^+}
     \tau^q_{\beta - \varepsilon} (\tmmathbf{n}) \]
  the left limit of $\tau^q_{\beta} (\tmmathbf{n})$. For any $\tau \in
  \mathbbm{R}$, $\beta \mapsto I_{\tmmathbf{n}} (\tau)$ is non-decreasing
  hence $\tilde{\tau}^q_{\beta} (\tmmathbf{n})$ (defined at
  (\ref{eq-def-taut})) does not decrease with $\beta$. According to Theorem
  \ref{thm-Iquad-gal}, $\tilde{\tau}^q_{\beta} (\tmmathbf{n})$ coincides with
  $\tau^q_{\beta} (\tmmathbf{n})$ for almost all $\beta$, hence
  \[ \tau^q_{\beta^-} (\tmmathbf{n}) \leqslant \tilde{\tau}^q_{\beta}
     (\tmmathbf{n}) \leqslant \tau^q_{\beta} (\tmmathbf{n}), \forall \beta
     \geqslant 0 \]
  hence the left continuity of $\tau^q_{\beta} (\tmmathbf{n})$ at a particular
  $\beta$ implies that $\tilde{\tau}^q_{\beta} (\tmmathbf{n}) = \tau^q_{\beta}
  (\tmmathbf{n})$, in other words that lower deviations are (at least) of
  surface order. This is the first part of the claim. Now we consider
  \[ \mathcal{D}= \left\{ \beta \in \mathbbm{R}^+, \exists \tmmathbf{n} \in
     S^{d - 1} : \tau^q_{\beta^-} (\tmmathbf{n}) \neq \tau^q_{\beta}
     (\tmmathbf{n}) \right\} \]
  and prove that $\mathcal{D}$ is at most countable. The homogeneous extension
  of $\tau^q_{\beta^-} (\tmmathbf{n})$ to $\mathbbm{R}^d$ is convex as the
  pointwise limit of the $f^q_{\beta - \varepsilon}$, hence $\tau^q_{\beta^-}
  (\tmmathbf{n})$ is a continuous function of $\tmmathbf{n} \in S^{d - 1}$.
  Consequently, for any dense sequence $(\tmmathbf{n}_n)_{n \in \mathbbm{N}}$
  in $S^{d - 1}$, we have
  \begin{eqnarray*}
    \mathcal{D} & \subset & \bigcup_{n \in \mathbbm{N}} \left\{ \beta \in
    \mathbbm{R}^+ : \tau^q_{\beta} (\tmmathbf{n}_n) \neq \tau^q_{\beta^-}
    (\tmmathbf{n}_n) \right\}
  \end{eqnarray*}
  which is at most countable.
\end{proof}

\section{Low temperatures asymptotics}

\label{sec-lowt}Here we study the low temperature asymptotics of surface
tension and prove the results presented in Section \ref{sec-intro-lowt}. We
begin with upper bounds on surface tension which hold in all generality, and
then establish lower bounds with the help of Peierls arguments.

\subsection{Upper bounds on surface tension}

Relevant upper bounds on surface tension are easily established:

\begin{lemma}
  \label{lem-tau-lt}Let $\mathbbm{P}$ be a product measure $\mathbbm{P}$ on
  $[0, 1]^d$, $\tmmathbf{n} \in S^{d - 1}$ and $\lambda > 0$. Then,
  \begin{eqnarray}
    \tau^q_{\beta} (\tmmathbf{n}) & \leqslant & \beta \mu (\tmmathbf{n}) 
    \label{eq-upb-tauq-mu}\\
    \text{and \ } \tau^{\lambda}_{\beta} (\tmmathbf{n}) & \leqslant &
    \|\tmmathbf{n}\|_1 \times \log \frac{1}{\mathbbm{E} \exp \left( - \lambda
    \beta J_e \right)}  \label{eq-upb-taul-n1}
  \end{eqnarray}
\end{lemma}

\begin{proof}
  We begin with the proof of (\ref{eq-upb-tauq-mu}) and consider a rectangular
  parallelepiped $\mathcal{R}$. With the notations of Section
  \ref{sec-intro-lowt}, for all interface $I \in \mathcal{I}(\mathcal{R})$,
  the DLR equation yields
  \[ \Phi^{J, w}_{\mathcal{R}} \left( \mathcal{D}_{\mathcal{R}} \right)
     \geqslant \Phi^{J, w}_{\mathcal{R}} \left( \mathcal{Z}_I \right)
     \geqslant \prod_{e \in I} \Phi_{\{e\}}^{J, w} (\omega_e = 0) = \exp
     \left( - \beta \sum_{e \in I} J_e \right) \]
  and consequently $\tau^J_{\beta, \mathcal{R}} \leqslant \beta
  \mu_{\mathcal{R}}^J$, which implies (\ref{eq-upb-tauq-mu}) taking
  $\mathcal{R}=\mathcal{R}^N =\mathcal{R}_{0, N, \delta N} (\mathcal{S},
  \tmmathbf{n})$ and $N \rightarrow \infty$. Similarly, in view of the
  definition (\ref{eq-def-taul-R}) we have
  \begin{eqnarray*}
    \tau^{\lambda}_{\beta, \mathcal{R}} & \leqslant & - \frac{1}{L^{d - 1}}
    \log \mathbbm{E} \left( \prod_{e \in I} \Phi_{\{e\}}^{J, w} \left(
    \omega_e = 0 \right)^{\lambda} \right)\\
    & \leqslant & \frac{|I|}{L^{d - 1}} \log \frac{1}{\mathbbm{E} \exp \left(
    - \lambda \beta J_e \right)}
  \end{eqnarray*}
  which yields (\ref{eq-upb-taul-n1}) if we choose for $I$ the interface of
  smallest cardinal in $\mathcal{I}(\mathcal{R})$, which has a cardinal
  approximately $\|\tmmathbf{n}\|_1 L^{d - 1}$.
\end{proof}

\subsection{Quenched surface tension and maximal flows}

We present here the proof of Proposition \ref{prop-tauq-mu}, which is based on
a control of the length of the interface, using a Peierls argument, and then
the proof of Proposition \ref{prop-tauqlt-b} which uses a renormalization
argument.

\begin{proof}
  (Proposition \ref{prop-tauq-mu}). Given a rectangular parallelepiped
  $\mathcal{R}$, we have
  \[ \Phi^J_{\mathcal{R}} \left( \mathcal{D}_{\mathcal{R}} \right) \leqslant
     \sum_{I \in \mathcal{I}(\mathcal{R})} \Phi^J_{\mathcal{R}} \left(
     \mathcal{Z}_I \right) \leqslant \sum_{I \in \mathcal{I}(\mathcal{R})}
     \prod_{e \in I} q e^{- \beta J_e} . \]
  We decompose the sum according to the length of the interface: for any $c
  >\|\tmmathbf{n}\|_1$,
  \begin{eqnarray}
    \Phi^J_{\mathcal{R}} \left( \mathcal{D}_{\mathcal{R}} \right) & \leqslant
    & \sum_{I \in \mathcal{I}(\mathcal{R}) : |I| < c L^{d - 1}} q^{|I|} e^{-
    \beta L^{d - 1} \mu^J_{\mathcal{R}}}  \nonumber\\
    &  & + \sum_{I \in \mathcal{I}(\mathcal{R}) : |I| \geqslant cL^{d - 1}}
    q^{|I|} e^{- \beta \sum_{e \in I} J_e}  \label{eq-upb-phiDR-c}
  \end{eqnarray}
  The first term is not larger than
  \[ (c_d q)^{cL^{d - 1} } \exp \left( - \beta L^{d - 1} \mu^J_{\mathcal{R}} 
     \right) \]
  and the expectation of the second one is
  \[ \mathbbm{E} \left( \sum_{I \in \mathcal{I}(\mathcal{R}) : |I| \geqslant
     cL^{d - 1}} q^{|I|} e^{- \beta \sum_{e \in I} J_e} \right) \leqslant
     \frac{1}{1 - c_d q\mathbbm{E}(e^{- \beta J_e})} \times \left[ c_d
     q\mathbbm{E}(e^{- \beta J_e}) \right]^{cL^{d - 1}} \]
  if $\rho_{\beta} = c_d q\mathbbm{E}(e^{- \beta J_e}) < 1$, which is the case
  for $\beta$ large as $\mathbbm{P}(J_e = 0) = 0 < (c_d q)^{- 1}$. For any
  such $\beta$, applying Markov's inequality we obtain, for any $\varepsilon >
  0$:
  \begin{eqnarray*}
    \mathbbm{P} \left( \sum_{I \in \mathcal{I}(\mathcal{R}) : |I| \geqslant
    cL^{d - 1}} q^{|I|} e^{- \beta \sum_{e \in I} J_e} \geqslant
    (\rho_{\beta})^{(1 - \varepsilon) cL^{d - 1}} \right) & \leqslant &
    \frac{1}{1 - \rho_{\beta}} \times (\rho_{\beta})^{\varepsilon cL^{d - 1}}
    .
  \end{eqnarray*}
  Hence (\ref{eq-upb-phiDR-c}) shows that, for $J$ typical under $\mathbbm{P}$
  -- up to large deviations of surface order --
  \[ \Phi^J_{\mathcal{R}} \left( \mathcal{D}_{\mathcal{R}} \right) \leqslant
     (c_d q)^{cL^{d - 1} } \exp \left( - \beta L^{d - 1} \mu^J_{\mathcal{R}} 
     \right) + (\rho_{\beta})^{(1 - \varepsilon) cL^{d - 1}} \]
  which proves that
  \begin{eqnarray*}
    \tau^q_{\beta} (\tmmathbf{n}) & \geqslant & \min \left( \beta \mu
    (\tmmathbf{n}) - c \log (c_d q), c \log \frac{1}{\rho_{\beta}} \right)
  \end{eqnarray*}
  for any $\beta \geqslant 0$ such that $\rho_{\beta} < 1$. The lower bound is
  optimal for
  \[ c = \frac{\beta \mu (\tmmathbf{n})}{\log (c_d q) + \log
     \frac{1}{\rho_{\beta}}} \]
  which is negligible with respect to $\beta$ in the limit $\beta \rightarrow
  + \infty$, as $\log (1 / \rho_{\beta}) \rightarrow + \infty$. The limit
  (\ref{eq-tauq-mu}) follows -- the uniformity over $\tmmathbf{n} \in S^{d -
  1}$ is a consequence of the fact that $\mu$ is bounded. If $J_{\min} > 0$,
  then we even have, for some $C < \infty$, that for $\beta$ large enough
  (independent of $\tmmathbf{n} \in S^{d - 1}$),
  \[ \tau^q_{\beta} (\tmmathbf{n}) \geqslant \beta \mu (\tmmathbf{n}) - C. \]
\end{proof}

\begin{proof}
  (Proposition \ref{prop-tauqlt-b}). The proof for (\ref{eq-conv-tauq-infty})
  exploits a renormalization argument similar to the one used in {\cite{N27}}.
  As $\mathbbm{P}(J_e > 0) > p_c (d)$, for small enough $\varepsilon > 0$ it
  is still the case that $\mathbbm{P}(J_e \geqslant \varepsilon) > p_c (d)$.
  We say that $e \in \mathbbm{E}(\mathbbm{Z}^d)$ is open for $J$ if $J_e
  \geqslant \varepsilon$, and consider the connected components for these
  definition of open edges. A block $B_i^K = Ki +\{1, \ldots, 2 K\}^d$ ($i \in
  \mathbbm{Z}^d$) of side-length $2 K$, is said {\tmem{good}} when, in
  $B_i^K$,
  \begin{enumerateroman}
    \item there is a unique connected component of diameter larger or equal to
    $K$
    
    \item and this connected component touches all the faces of the block.
  \end{enumerateroman}
  The work of Pisztora {\cite{N09}} together with the knowledge that, in the
  case of independent bond percolation, the slab percolation threshold
  coincides with the threshold for percolation {\cite{N170}} imply that the
  sequence of random variables
  \[ \eta_i =\tmmathbf{1}_{\left\{ \text{The block } B_i^K \text{ is good}
     \right\}} \]
  stochastically dominate a site percolation process of parameter $\rho$ close
  to one, provided that $K$ is large enough. Provided that $\rho$ (hence $K$)
  is large enough, a Peierls argument shows that there is a probability $1 -
  \exp (cN^{d - 1})$ (for some $c > 0$, for $N$ large enough) that no
  $K$-block interface in $\mathcal{R}^N =\mathcal{R}_{0, N, \delta N}
  (\mathcal{S}, \tmmathbf{n})$ contains less than half of good blocks.
  
  Given a suitable $K$ we now establish a lower bound on the quenched surface
  tension. Consider a realization $J$ of the couplings with the property that
  no $K$-block interface in $\mathcal{R}^N =\mathcal{R}_{0, N, \delta N}
  (\mathcal{S}, \tmmathbf{n})$ contains less than half of good blocks. For any
  such $J$, the event of disconnection requires the choice of a block surface
  of cardinality at least $(N / K)^{d - 1}$, and that, in each good block, at
  least one edge with $J_e \geqslant \varepsilon$ be closed. Hence: for $N$
  large and $J$ typical up to surface order large deviations,
  \begin{eqnarray*}
    \Phi_{\mathcal{R}^N}^{J, w} (\mathcal{D}_{\mathcal{R}^N}) & \leqslant &
    \sum_{n \geqslant (N / K)^{d - 1}} (c_d)^n  \left[ c_d K^d qe^{- \beta
    \varepsilon} \right]^{(1 - \varepsilon) n}
  \end{eqnarray*}
  leading to $\tau^q_{\beta} \geqslant \left[ (1 - \varepsilon) \beta
  \varepsilon - \log \left( c_d ^2 K^d q) \right] / K^{d - 1} \right.$ for
  large enough $\beta$. The claim follows.
\end{proof}

\subsection{Surface tension under the averaged Gibbs measure}

Under the assumption $\mathbbm{P}(J_e > 0) = 1$, we establish a lower bound on
$\tau^{\lambda}_{\beta} (\tmmathbf{n})$ which is equivalent to the upper
bound:

\begin{proposition}
  \label{prop-tault-Js1}Assume that $\mathbbm{P}(J_e > 0) = 1$. Then,
  uniformly over $\tmmathbf{n} \in S^{d - 1}$,
  \begin{eqnarray}
    \tau^{\lambda}_{\beta} (\tmmathbf{n}) & \geqslant & (1 - o_{\beta
    \rightarrow \infty} (1)) \times \|\tmmathbf{n}\|_1 \times \log
    \frac{1}{\mathbbm{E} \exp \left( - \lambda \beta J_e \right)} . 
    \label{eq-equiv-taul-n1}
  \end{eqnarray}
\end{proposition}

Before we address the proof of Proposition \ref{prop-tault-Js1}, let us remark
that Proposition~\ref{prop-taul-Jmin} is a clear consequence of Lemma
\ref{lem-tau-lt} and Proposition \ref{prop-tault-Js1}.

\begin{proof}
  (Proposition \ref{prop-tault-Js1}). Remark that for any $I \subset
  \mathcal{I}(\mathcal{R})$,
  \[ \Phi^{J, w}_{\mathcal{R}} \left( \mathcal{Z}_I \right) \leqslant \prod_{e
     \in I} \Phi^{J, f}_{\mathcal{R}} \left( \omega_e = 0 \right) \leqslant
     \prod_{e \in I} q e^{- \beta J_e} . \]
  If $\lambda \leqslant 1$, the inequality $( \sum_{i = 1}^n x_i)^{\lambda}
  \leqslant \sum_{i = 1}^n x_i^{\lambda}$ for non-negative $x_i$ yields
  \[ \Phi^{J, w}_{\mathcal{R}} \left( \mathcal{D}_{\mathcal{R}}
     \right)^{\lambda} \leqslant \left( \sum_{I \in \mathcal{I}(\mathcal{R})}
     \Phi^{J, w}_{\mathcal{R}} \left( \mathcal{Z}_I \right) \right)^{\lambda}
     \leqslant \sum_{I \in \mathcal{I}(\mathcal{R})} \Phi^{J, w}_{\mathcal{R}}
     \left( \mathcal{Z}_I \right)^{\lambda} \]
  hence
  \begin{eqnarray*}
    \mathbbm{E} \left[ \left( \Phi^{J, w}_{\mathcal{R}} \left(
    \mathcal{D}_{\mathcal{R}} \right) \right)^{\lambda} \right] & \leqslant &
    \sum_{I \in \mathcal{I}(\mathcal{R})} \prod_{e \in I} q^{\lambda}
    \mathbbm{E}e^{- \lambda \beta J_e}\\
    & = & \sum_{I \in \mathcal{I}(\mathcal{R})} \left( q^{\lambda}
    \mathbbm{E}e^{- \lambda \beta J_e} \right)^{|I|}
  \end{eqnarray*}
  As $\mathbbm{E}e^{- \lambda \beta J_e} \rightarrow 0$ as $\beta \rightarrow
  \infty$ under the assumption $\mathbbm{P}(J_e > 0) = 1$, the Peierls
  argument gives a relevant lower bound : there is $c_d$ depending only on the
  dimension $d$, such that the number of interfaces of cardinal $n$ in
  $\mathcal{I}(\mathcal{R})$ is not larger than $(c_d)^n$. Hence,
  \begin{eqnarray*}
    \mathbbm{E} \left[ \left( \Phi^{J, w}_{\mathcal{R}} \left(
    \mathcal{D}_{\mathcal{R}} \right) \right)^{\lambda} \right] & \leqslant &
    \sum_{n \geqslant \min_{I \in \mathcal{I}(\mathcal{R})} |I|} (c_d
    q^{\lambda} \mathbbm{E}e^{- \lambda \beta J_e})^n\\
    & \leqslant & \frac{1}{1 - c_d q^{\lambda} \mathbbm{E}e^{- \lambda \beta
    J_e}} \times \left[ c_d q^{\lambda} \mathbbm{E}e^{- \lambda \beta J_e}
    \right]^{\min_{I \in \mathcal{I}(\mathcal{R})} |I|}
  \end{eqnarray*}
  for $\beta$ large enough, thus
  \begin{equation}
    \tau^{\lambda}_{\beta} \geqslant \|\tmmathbf{n}\|_1 \times \left( \log
    \frac{1}{\mathbbm{E}e^{- \lambda \beta J_e}} - (1 + \lambda) \log c_d -
    \lambda \log q \right) \label{eq-lwb-taul}
  \end{equation}
  for all $\lambda \leqslant 1$ and $\beta$ large enough. If $\lambda
  \geqslant 1$, Minkowski's inequality yields:
  \begin{eqnarray*}
    \left[ \mathbbm{E} \left[ \left( \Phi^{J, w}_{\mathcal{R}} \left(
    \mathcal{D}_{\mathcal{R}} \right) \right)^{\lambda} \right] \right]^{1 /
    \lambda} & \leqslant & \sum_{I \in \mathcal{I}(\mathcal{R})} \left[
    \mathbbm{E} \left[ \left( \Phi^{J, w}_{\mathcal{R}} \left( \mathcal{Z}_I
    \right) \right)^{\lambda} \right] \right]^{1 / \lambda}\\
    & \leqslant & \sum_{I \in \mathcal{I}(\mathcal{R})} \prod_{e \in I} q
    \left[ \mathbbm{E}e^{- \lambda \beta J_e} \right]^{1 / \lambda}
  \end{eqnarray*}
  and we conclude similarly that (\ref{eq-lwb-taul}) holds again for all
  $\lambda \geqslant 1$ and $\beta$ large enough. The claim
  (\ref{eq-equiv-taul-n1}) follows from the divergence $\lim_{\beta
  \rightarrow + \infty} \log (1 /\mathbbm{E}e^{- \lambda \beta J_e}) = +
  \infty$ under the assumption $\mathbbm{P}(J_e > 0) = 1$, and the convergence
  is uniform in $\tmmathbf{n} \in S^{d - 1}$ as (\ref{eq-lwb-taul}) holds for
  any $\beta$ large enough independent of $\tmmathbf{n}$.
\end{proof}

\subsection{Limit shapes at low temperatures}

The limit shape of Wulff crystals (\ref{eq-Wulff}) are immediately inferred
from the uniform limits for surface tension. The Proposition below is a
consequence of Proposition \ref{prop-tauq-mu} for the first point, Lemma
\ref{lem-tau-lt} and Proposition \ref{prop-tault-Js1} for the second one:

\begin{proposition}
  \label{prop-Wulff-lowt}Let $\mathbbm{P}$ be a product measure on $[0, 1]^d$
  with $\mathbbm{P}(J_e > 0) = 1$.
  \begin{enumerateroman}
    \item The Wulff crystal $\mathcal{W}^q$ converges to the Wulff crystal
    $\mathcal{W}^{\mu}$ associated with the maximal flow $\mu$.
    
    \item For any $\lambda > 0$, the Wulff crystal $\mathcal{W}^{\lambda}$
    converges to the hypercube $\mathcal{W}^{\|.\|_1} = [\pm 1 / 2]^d$.
  \end{enumerateroman}
\end{proposition}

\label{sec-shape-Wmu}We remarked above that the maximal flow determines the
limit shape of the crystal in the (quenched) dilute Ising model as the
temperature goes to zero, while the crystals under the averaged Gibbs measure
converges to the unit hypercube.

Let us explain how the result of Durrett and Liggett {\cite{N257}} for site
first passage percolation can be used to show that $\mathcal{W}^{\mu}$ is not
in general an hypercube. We consider $d = 2$ and $\mathbbm{P}$ such that
\begin{equation}
  \mathbbm{P} \left( J_e = \frac{1}{2} \right) = p \text{ \ and \ }
  \mathbbm{P} \left( J_e = 1 \right) = 1 - p \label{ex-Wmu}
\end{equation}
with $\overrightarrow{p_c} < p < 1$, where $\overrightarrow{p_c}$ is the
critical threshold for oriented bond percolation. In the two dimensional case,
one can also interpret $\mu (\tmmathbf{n})$ as the limit ratio over $N$ of the
time needed for reaching the position $N\tmmathbf{n}^{\perp}$, if to every
edge $e$ we associate a passage time $J_e$. If $\tmmathbf{n}$ belongs to the
cone of oriented percolation (modulo the symmetries of the lattice
$\mathbbm{Z}^2$) one can reach the position $N\tmmathbf{n}^{\perp}$ following
a directed path with all edges (except finitely many at the origin) satisfying
$J_e = 1 / 2$. Hence, in those directions,
\[ \mu (\tmmathbf{n}) = \frac{1}{2} \|\tmmathbf{n}\|_1 . \]
However, the argument of {\cite{N257}} (applied to bond in place of site first
passage percolation) shows that when $\tmmathbf{n}$ is close enough to the
axis, one has to use edges $J_e = 1$ with a positive frequency, thus $\mu
(\tmmathbf{n}) >\|\tmmathbf{n}\|_1 / 2$ for those directions. The reciprocity
formula
\[ \tau (\tmmathbf{n}) = \sup_{x \in \mathcal{W}^{\tau}} x \cdot \tmmathbf{n}
\]
for Wulff crystals, where $\mathcal{W}^{\tau} = \left\{ x : x \cdot
\tmmathbf{n} \leqslant \tau (\tmmathbf{n}), \forall \tmmathbf{n} \right\}$ is
the un-normalized Wulff crystal for $\tau$, shows that $\mathcal{W}^{\mu}$ is
not a square as $\mu (\tmmathbf{n})$ is not proportional to
$\|\tmmathbf{n}\|_1$.

\section{Phase coexistence}

\label{sec-phco}

\subsection{Profiles of bounded variation and surface energy}

\label{sec-BV}The coarse graining for the dilute Ising model (Theorem 5.10 in
{\cite{M01}}) implies that at every position, the local magnetization
$\mathcal{M}_K$ is close to $\pm m_{\beta}$ with large probability. In order
to describe the geometrical structure of the phases, we estimate the
probability that $\mathcal{M}_K / m_{\beta}$ be close, in $L^1$-distance, to a
given Borel measurable function $u : [0, 1]^d \rightarrow \{\pm 1\}$. As a
first step towards the description of phase coexistence, we define here the
set of profiles we consider, define surface energy and the associated
isoperimetric problem.

In the following, $\mathcal{L}^d$ stands for the Lebesgue measure on
$\mathbbm{R}^d$ and $\mathcal{H}^{d - 1}$ for the $d - 1$ dimensional
Hausdorff measure, which gives to any Borel set $X \subset \mathbbm{R}^d$ the
weight
\[ \mathcal{H}^{d - 1} \left( X \right) = \lim_{\delta \rightarrow 0^+}
   \frac{\alpha_{d - 1}}{2^{d - 1}} \inf \left\{ \sum_{i \in I} [\tmop{diam}
   (E_i)]^{d - 1} : \sup_{i \in I} \tmop{diam} (E_i) < \delta, X \subset
   \bigcup_{i \in I} E_i \right\} \]
where the infimum takes into account finite or countable coverings $(E_i)_{i
\in I}$, and $\alpha_{d - 1}$ is the volume of the unit ball of
$\mathbbm{R}^{d - 1}$. The $L^1$-distance between two Borel measurable
functions $u, v : [0, 1]^d \rightarrow \mathbbm{R}$ is
\[ \|u - v\|_{L^1} = \int_{[0, 1]^d} |u - v|d\mathcal{L}^d, \]
and the set $L^1$ is
\[ \left\{ u : [0, 1]^d \rightarrow \mathbbm{R} \text{ Borel measurable},
   \|u\|_{L^1} < \infty \right\} . \]
In order that $L^1$ be a Banach space for the $L^1$-norm, we identify $u : [0,
1]^d \rightarrow \mathbbm{R}$ with the class of functions $v : \|u - v\|_{L^1}
= 0$ that coincide with $u$ on a set of full measure. We also denote by
$\mathcal{V}(u, \delta)$ the neighborhood of radius $\delta > 0$ in $L^1$
around $u \in L^1$.

For the study of phase coexistence, we have to consider virtually any $u \in
L^1$ taking values in $\{\pm 1\}$. Before we can define the surface energy for
such profiles, a description of the boundary of these profiles is necessary.
It is done conveniently in the framework of bounded variation profiles
(Chapter 3 in {\cite{N129}}). Given a Borel subset $U \subset \mathbbm{R}^d$,
the variation (or perimeter) of $U$ is
\[ \mathcal{P}(U) = \sup \left\{ \int_U \tmop{div} f d\mathcal{L}^d, f \in
   \mathcal{C}^{\infty}_c (\mathbbm{R}^d, [- 1, 1]) \right\} \in [0, \infty]
\]
where $\mathcal{C}^{\infty}_c (\mathbbm{R}^d, [- 1, 1])$ is the set of
$\mathcal{C}^{\infty}$ functions from $\mathbbm{R}^d$ to $[- 1, 1]$ with
compact support, and $\tmop{div}$ the divergence operator:
\[ \tmop{div} f = \frac{\partial f}{\partial x_1} + \ldots + \frac{\partial
   f}{\partial x_n} . \]
To $U \subset \mathbbm{R}^d$ Borel measurable, we associate $u = \chi_U$ as in
(\ref{chi}) and define the set of bounded variation profiles $\tmop{BV}$ as
follows:
\[ \tmop{BV} = \left\{ u = \chi_U : U \subset (0, 1)^d \text{ is a Borel set
   and } \mathcal{P}(U) < \infty \right\} . \]
Bounded variations profiles $u = \chi_U \in \tmop{BV}$ have a
\tmtextit{reduced boundary} $\partial^{\star} u$ and an outer normal
$\tmmathbf{n}_.^u : \partial^{\star} u \rightarrow S^{d - 1}$ with, in
particular, $\mathcal{H}^{d - 1} (\partial^{\star} u) =\mathcal{P}(U)$.

This allows us to define the {\tmem{surface energy}} of bounded variation
profiles. As the outer normal $\tmmathbf{n}_.^u$ defined on $\partial^{\star}
u$ is Borel measurable, we can consider
\begin{equation}
  \mathcal{F}^q (u) = \int_{\partial^{\star} u} \tau^q (\tmmathbf{n}_x^u)
  d\mathcal{H}^{d - 1} (x) \label{eq-def-Ftau} \text{, \ \ \ \ } \forall u \in
  \tmop{BV}
\end{equation}
and
\begin{equation}
  \mathcal{F}^{\lambda} (u) = \int_{\partial^{\star} u} \tau^{\lambda}
  (\tmmathbf{n}_x^u) d\mathcal{H}^{d - 1} (x) \label{eq-def-Ftaul}  \text{, \
  \ \ \ } \forall u \in \tmop{BV}, \forall \lambda > 0.
\end{equation}
where $\tau^q$ (resp. $\tau^{\lambda}$) stands for the quenched surface
tension of the dilute Ising model (resp. surface tension under the averaged
Gibbs measure), see Theorem \ref{thm-conv-tauq} and (\ref{eq-dual-Itaul}).
Because the homogeneous extension of the surface tensions $\tau^q$ and
$\tau^{\lambda}$ are convex (Proposition \ref{prop-conv-tau}), $\mathcal{F}^q$
and $\mathcal{F}^{\lambda}$ are lower semi-continuous with respect to the
$L^1$-norm. See Chapter 14 in {\cite{N70}} or Theorem 2.1 in {\cite{N118}}.
For commodity, when $u = \chi_U \in \tmop{BV}$ we also denote by
$\mathcal{F}^q (U)$ (resp. $\mathcal{F}^{\lambda} (U)$) the surface energy of
$u$.

When surface tension is positive, the level sets of $\mathcal{F}^q$ and
$\mathcal{F}^{\lambda}$ are compact since, for all $a \geqslant 0$, the set
\begin{equation}
  \tmop{BV}_a = \left\{ u = \chi_U \in \tmop{BV} : \mathcal{P}(U) \leqslant a
  \right\} \label{eq-def-BVa}
\end{equation}
is itself compact for the $L^1$-norm, cf. Theorem 3.23 in {\cite{N129}}.
Consequently, $\mathcal{F}^q$ and $\mathcal{F}^{\lambda}$ are good rate
functions.

Let us conclude with a word on the solutions to the {\tmem{isoperimetric
problem}} of finding the $u \in \tmop{BV}$ such that
\begin{equation}
  \int_{[0, 1]^d} u \, d\mathcal{L}^d \leqslant \frac{m}{m_{\beta}} \text{ \
  and \ } \mathcal{F}^q (u) \text{ is minimal ?} \label{eq-uBV-opt}
\end{equation}
The renormalized Wulff crystal $\mathcal{W}^q$ (\ref{eq-Wulff}) is known to be
the solution to the same problem \tmtextit{without} the constraint that $U
\subset (0, 1)^d$. Precisely, the solutions to $U \subset \mathbbm{R}^d$ Borel
set with
\[ \mathcal{L}^d (U) = 1 \text{ \ and \ } \mathcal{F}^q (U) \text{ minimal} \]
are the translates of $\mathcal{W}^q$, as the homogeneous extension of
$\tau^q$ is convex (Proposition \ref{prop-conv-tau}) -- see {\cite{N91}},
{\cite{N89}} and {\cite{N90}}.

For $m < m_{\beta}$ not too small, $\mathcal{W}^q$ determines as well the
optimal profiles in the cube (\ref{eq-uBV-opt}). Consider $\alpha > 0$ with
\begin{equation}
 \alpha^d = \frac{1}{2}  \left( 1 - \frac{m}{m_{\beta}} \right) .
\end{equation}
The quantity $\alpha^d$ is precisely the least volume of $U$ corresponding to
$u = \chi_U \in \tmop{BV}$ with$\int_{[0, 1]^d} ud\mathcal{L}^d \leqslant m /
m_{\beta}$. If some translate of $\alpha \mathcal{W}^q$ fits into the unit
cube $[0, 1]^d$, that is if $\alpha \tmop{diam}_{\infty} (\mathcal{W}^q)
\leqslant 1$, then $\mathcal{T}(\alpha \mathcal{W}^q)$ defined at (\ref{T}) is
not empty and therefore the infimum of $\mathcal{F}^q (u)$ for $u \in
\tmop{BV}$ with $\int_{[0, 1]^d} ud\mathcal{L}^d \leqslant m / m_{\beta}$ is
exactly $\mathcal{F}^q (\alpha \mathcal{W}^q)$. As a consequence, for all
$\alpha$ satisfying $\alpha \tmop{diam}_{\infty} (\mathcal{W}^q)
\leqslant 1$ the optimal phase profiles correspond
to the translates of $\alpha \mathcal{W}^q$ that belong to $[0, 1]^d$, which
are the $z + \alpha \mathcal{W}^q$, for $z \in \mathcal{T}(\alpha
\mathcal{W}^q)$. The same remains true if we replace $\mathcal{F}^q$ and
$\mathcal{W}^q$ with $\mathcal{F}^{\lambda}$ and $\mathcal{W}^{\lambda}$, for
any $\lambda > 0$.

\subsection{Covering theorems for BV profiles}

\label{sec-cover}Covering theorems play an essential role in the study of
phase coexistence, as they allow to pass from the macroscopic scale (the phase
profile $u$) to the microscopic scale (the dilute Ising model). We give first
two definitions:

\begin{definition}
  \label{def-d-adapted-u}Let $u \in \tmop{BV}$, $\tau : \mathcal{S}^{d - 1}
  \mapsto [0, \infty]$ continuous, $\delta > 0$ and $\mathcal{R}$ a
  rectangular parallelepiped as in (\ref{eq-def-R}), included in $[0, 1]^d$.
  We say that $\mathcal{R}$ is $\delta$-adapted to $u$ and $\tau$ at $x \in
  \partial^{\star} u$ if the following holds:
  \begin{enumerateroman}
    \item If $\tmmathbf{n}=\tmmathbf{n}_x^u$ is the outer normal to $u$ at
    $x$, there are $\mathcal{S} \in \mathbbm{S}_{\tmmathbf{n}}$ and $h \in (0,
    \delta]$ such that, if $\mathcal{R} \subset (0, 1)^d$ (we say that
    $\mathcal{R}$ is interior), then
    \[ \mathcal{R}= x + h\mathcal{S}+ [\pm \delta h]\tmmathbf{n}, \]
    and if $\mathcal{R} \cap \partial [0, 1]^d \neq \emptyset$ (we say that
    $\mathcal{R}$ is on the border), then $x \in \partial [0, 1]^d$,
    $\tmmathbf{n}$ is also the outer normal to $[0, 1]^d$ at $x$ and
    \[ \mathcal{R}= x + h\mathcal{S}+ [- \delta h, 0]\tmmathbf{n}. \]
    \item We have
    \[ \mathcal{H}^{d - 1} \left( \partial^{\star} u \cap \partial \mathcal{R}
       \right) = 0, \]
    \[ \left| 1 - \frac{1}{h^{d - 1}} \mathcal{H}^{d - 1} \left(
       \partial^{\star} u \cap \mathcal{R} \right) \right| \leqslant \delta,
    \]
    and
    \[ \left| \tau (\tmmathbf{n}) - \frac{1}{h^{d - 1}} \int_{\partial^{\star}
       u \cap \mathcal{R}} \tau (\tmmathbf{n}_.^u) d\mathcal{H}^{d - 1}
       \right| \leqslant \delta . \]
    \item If $\chi : \mathbbm{R}^d \rightarrow \{\pm 1\}$ is the
    characteristic function of $\mathcal{R}$ defined by
    \[ \chi (z) = \left\{ \begin{array}{ll}
         + 1 & \text{if } (z - x) \cdot \tmmathbf{n} \geqslant 0\\
         - 1 & \text{else}
       \end{array}, \forall z \in \mathbbm{R}^d, \right. \]
    then
    \[ \frac{1}{2 \delta h^d} \int_{\mathcal{R}} | \chi - u|d\mathcal{H}^d
       \leqslant \delta . \]
  \end{enumerateroman}
\end{definition}

\begin{definition}
  \label{def-d-cover}Let $u \in \tmop{BV}$, $\tau : \mathcal{S}^{d - 1}
  \mapsto [0, \infty]$ continuous and $\delta > 0$. A finite sequence
  $(\mathcal{R}_i)_{i = 1 \ldots n}$ of disjoint rectangular parallelepipeds
  included in $[0, 1]^d$ is said to be a $\delta$-covering for
  $\partial^{\star} u$ and $\tau$ if each $\mathcal{R}_i$ is $\delta$-adapted
  to $u$ and $\tau$ and if
  \begin{equation}
    \mathcal{H}^{d - 1} \left( \partial^{\star} u \setminus \bigcup_{i = 1}^n
    \mathcal{R}_i \right) \leqslant \delta \label{eq-cov-lwb-left-delta} .
  \end{equation}
\end{definition}

The Vitali covering theorem (Theorem 13.3 in {\cite{N70}}) is especially well
adapted to our purpose. Given a Borel set $E \subset \mathbbm{R}^d$, we say
that a collection of sets $\mathcal{U}$ is a Vitali class for $E$ if, for each
$x \in E$ and $\delta > 0$, there is $U \in \mathcal{U}$ with $0 < \tmop{diam}
U < \delta$ containing $x$.

\begin{theorem}
  \tmtextup{[Vitali]} \label{thm-Vitali}Let $E \subset \mathbbm{R}^d$ be
  $\mathcal{H}^{d - 1}$-measurable and consider $\mathcal{U}$ a Vitali class
  of closed sets for $E$. Then, there is a countable disjoint sequence
  $(U_i)_{i \in I}$ in $\mathcal{U}$ such that
  \[ \text{either \ \ } \sum_{i \in I} (\tmop{diam} U_i)^{d - 1} = \infty
     \text{ \ \ or \ \ } \mathcal{H}^{d - 1} \left( E \setminus \bigcup_{i \in
     I} U_i \right) = 0. \]
\end{theorem}

The Vitali Theorem allows us to state a short proof of the following:

\begin{theorem}
  \label{thm-d-cover}For any $u \in \tmop{BV}$, $\tau : \mathcal{S}^{d - 1}
  \mapsto [0, \infty]$ continuous and $\delta, h > 0$, there is a
  $\delta$-covering $(\mathcal{R}_i)_{i = 1 \ldots n}$ for $\partial^{\star}
  u$ and $\tau$.
\end{theorem}

Before we give the proof of Theorem \ref{thm-d-cover} we recall a property of
the reduced boundary (see Theorem 3.59 in {\cite{N129}}):

\begin{lemma}
  \label{lem-conv-duR}Let $u \in \tmop{BV}$. For all $x \in \partial^{\star}
  u$, for all $\delta \in (0, 1)$, all $\mathcal{S} \in
  \mathbbm{S}_{\tmmathbf{n}_x^u}$ one has
  \[ \lim_{h \rightarrow 0^+} \frac{1}{h^{d - 1}} \mathcal{H}^{d - 1} \left(
     \partial^{\star} u \cap \dot{\mathcal{R}}_{x, h, \delta h} (\mathcal{S},
     \tmmathbf{n}_x^u) \right) = 1. \]
\end{lemma}

\begin{proof}
  (Theorem \ref{thm-d-cover}). We design a set $E$ that has zero
  $\mathcal{H}^{d - 1}$-measure and such that the collection of closed
  rectangular parallelepipeds
  \[ \mathcal{U}_{\delta} = \left\{ \mathcal{R} \text{ $\delta$-adapted to $u$
     and $\tau$ at } x \in \partial^{\star} u \setminus E \right\} \]
  is a Vitali class for $\partial^{\star} u \setminus (E)$. This is enough to
  prove the claim: thanks to the Vitali covering Theorem, this implies the
  existence of a countable disjoint sequence $(\mathcal{R}_i)_{i \in I}$ of
  $\delta$-adapted rectangular parallelepipeds with either
  \[ \sum_{i \in I} (\tmop{diam} \mathcal{R}_i)^{d - 1} = \infty \text{ \ or \
     } \mathcal{H}^{d - 1} \left( \partial^{\star} u \setminus \bigcup_{i \in
     I} \mathcal{R}_i \right) = 0. \]
  The first case is in contradiction with the inequalities $1 / h_i^{d - 1}
  \mathcal{H}^{d - 1} \left( \partial^{\star} u \cap \mathcal{R}_i \right)
  \geqslant 1 - \delta$ and $\mathcal{H}^{d - 1} \left( \partial^{\star} u
  \right) < \infty$, hence the second is realized and the Theorem is proved.
  
  We define the set $E$ by its complement in $\partial^{\star} u$:
  $\partial^{\star} u \setminus E$ is the set of all $x \in \partial^{\star}
  u$ such that, for all $\mathcal{S} \in \mathbbm{S}_{\tmmathbf{n}^u_x}$, the
  following holds:
  \begin{enumerateroman}
    \item If $x \in \partial [0, 1]^d$, then $\tmmathbf{n}_x^u$ is the outer
    normal to $[0, 1]^d$ at $x$.
    
    \item The set $\left\{ h > 0 : \mathcal{H}^{d - 1} \left( \partial^{\star}
    u \cap \partial \mathcal{R}_{x, h, \delta h} (\mathcal{S},
    \tmmathbf{n}_x^u) \right) > 0 \right\}$ is at most countable.
    
    \item $\lim_{h \rightarrow 0^+} \frac{1}{h^{d - 1}} \mathcal{H}^{d - 1}
    \left( \partial^{\star} u \cap \dot{\mathcal{R}}_{x, h, \delta h}
    (\mathcal{S}, \tmmathbf{n}_x^u) \right) = 1$.
    
    \item $\lim_{h \rightarrow 0^+} \frac{1}{h^{d - 1}} \int_{\partial^{\star}
    u \cap \dot{\mathcal{R}}_{x, h, \delta h} (\mathcal{S}, \tmmathbf{n}_x^u)}
    \tau (\tmmathbf{n}_.^u) d\mathcal{H}^{d - 1} = \tau (\tmmathbf{n}_x^u)$.
    
    \item $\lim_{h \rightarrow 0^+} \frac{1}{h^d} \int_{\mathcal{R}_{x, h,
    \delta h} (\mathcal{S}, \tmmathbf{n}_x^u)} \left| \chi_{x,
    \tmmathbf{n}^u_x} - u_{} \right|_{} d\mathcal{L}^d = 0$.
  \end{enumerateroman}
  This definition for $E$ implies that $\mathcal{U}_{\delta}$ is a Vitali
  class of closed sets for $\partial^{\star} u \setminus E$. We conclude the
  proof of Theorem \ref{thm-d-cover} showing that $E$ has zero $\mathcal{H}^{d
  - 1}$-measure, and more precisely that each of conditions (i)-(v) is true
  for (at least) $\mathcal{H}^{d - 1}$-almost all $x \in \partial^{\star} u$:
  \begin{enumerateroman}
    \item This condition holds for all $x \in \partial^{\star} u$ because of
    the inclusion $U \subset (0, 1)^d$ if $u = \chi_U$, cf. Theorem 3.59 in
    {\cite{N129}}.
    
    \item Since the volume of $\partial^{\star} u$ is zero, (ii) holds for all
    $x$.
    
    \item Condition (iii) holds for all $x \in \partial^{\star} u$ in view of
    Lemma \ref{lem-conv-duR}.
    
    \item It is a consequence of the strong form of the Besicovitch derivation
    theorem (Theorem 5.52 in {\cite{N129}}) together with Lemma
    \ref{lem-conv-duR}, that condition (iv) holds for $\mathcal{H}^{d -
    1}$-almost all $x \in \partial^{\star} u$.
    
    \item Condition (v) holds for all $x \in \partial^{\star} u$, cf. Theorem
    3.59 in {\cite{N129}}.
  \end{enumerateroman}
\end{proof}

\subsection{Lower bound for phase coexistence}

Here we establish lower bounds for the probability of phase coexistence. In
view of the applications, in particular to the control of the dynamics
{\cite{M03}} or Chapter 4 in {\cite{M00}}, we establish it for a large class
of profiles, that include Wulff crystals and shapes with $C^1$ boundary.

Proposition \ref{prop-lwb-disc} below relates the probability of an event of
disconnection along the boundary of a given profile to the surface tension
$\tau^J$, for a given realization of the media. In Proposition
\ref{prop-lwb-cond-phco} we show that conditionally on this event of
disconnection, phase coexistence has large probability. Then we state in
Proposition \ref{prop-lwb-phco-final} a lower bound on the probability of
phase coexistence for both quenched and averaged measures.

Given some region $U \subset (0, 1)^d$, $N \in \mathbbm{N}^{\star}$ and
$\delta > 0$, we consider $\mathcal{E}^{N, \delta}_U$ the set of edges at
distance at most $N \sqrt{d} \delta$ from $N \partial U$:
\[ \mathcal{E}^{N, \delta}_U = \left\{ e \in E^w (\Lambda_N), d (e, N \partial
   U) \leqslant N \sqrt{d} \delta \right\} \]
(see Figure \ref{fig-binf-phco}) and call
\[ \mathcal{D}_U^{N, \delta} = \left\{ \omega \in \Omega_{E^w (\Lambda_N)} :
   x \overset{\omega}{\nleftrightarrow} y, \begin{array}{l}
     \forall x \in \Lambda_N \setminus NU, y \in \Lambda_N \cap NU \text{
     with}\\
     d (x / N, \partial U) > \sqrt{d} \delta \text{ and } d (y / N, \partial
     U) > \sqrt{d} \delta
   \end{array} \right\} \]
the event that disconnection occurs around $\partial U$. In order to be able
to control the probability of $\mathcal{D}_U^{N, \delta}$, we introduce the
following definition:

\begin{definition}
  \label{def-u-regular}We say that a profile $u = \chi_U$ is regular if
  \begin{enumerateroman}
    \item $U$ is open and at positive distance from the boundary $\partial [0,
    1]^d$ of the unit cube,
    
    \item $\partial U$ is $d - 1$ rectifiable and
    
    \item for small enough $r > 0$, $[0, 1]^d \setminus \left( \partial U + B
    (0, r) \right)$ has exactly two connected components.
  \end{enumerateroman}
\end{definition}

We recall that $E \subset \mathbbm{R}^d$ is a $d - 1$ rectifiable set if there
exists a Lipschitzian function mapping some bounded subset of $\mathbbm{R}^{d
- 1}$ onto $E$ (Definition 3.2.14 in {\cite{N304}}). It is the case in
particular of the boundary of non-empty Wulff crystals (Theorem 3.2.35 in
{\cite{N304}}) and of bounded polyhedral sets. It follows from Proposition
3.62 in {\cite{N129}} that any $u = \chi_U$ regular belongs to $\tmop{BV}$ and
that $\partial U = \partial^{\star} u$ up to a $\mathcal{H}^{d -
1}$-negligible set, so that the covering Theorem applies as well to $\partial
U$. Assumption (ii) in Definition \ref{def-u-regular} has the following
consequence:

\begin{lemma}
  \label{lem-mink-cont}Let $u = \chi_U \in \tmop{BV}$ be a regular profile.
  Then, for any $\delta > 0$, for any $\delta$-covering $(\mathcal{R}_i)_{i =
  1 \ldots n}$ of $u$, one has
  \[ \limsup_{r \rightarrow 0} \frac{\mathcal{L}^d \left( \left( \partial U
     \setminus \bigcup_{i = 1}^n \mathcal{R}_i \right) + B \left( 0, r \right)
     \right)}{r} \leqslant 2 \delta . \]
\end{lemma}

\begin{proof}
  Clearly, the set
  \[ E = \partial U \setminus \bigcup_{i = 1}^n \dot{\mathcal{R}}_i \]
  is a closed, $d - 1$ rectifiable set. Thus, the $d - 1$ Minkowski content of
  $E$ equals the $d - 1$ dimensional Hausdorff measure of $E$ (Theorem 3.2.39
  in {\cite{N304}}). In other words:
  \[ \lim_{r \rightarrow 0} \frac{\mathcal{L}^d \left( E + B \left( 0, r
     \right) \right)}{2 r} =\mathcal{H}^{d - 1} \left( E \right) \leqslant
     \delta \]
  and the claim follows.
\end{proof}

Before we state Propositions \ref{prop-lwb-disc} and \ref{prop-lwb-cond-phco}
we give one more notation. The analysis of surface tension has been done for
rectangular parallelepiped centered \tmtextit{at lattice points}. Changing the
center of the parallelepipeds does not modify the behavior of surface tension,
but this would have led to heavier notations. We prefer to proceed to a small
adjustment here: given a macroscopic rectangular parallelepiped $\mathcal{R}
\subset (0, 1)^d$ and $N \in \mathbbm{N}^{\star}$, we let
\begin{equation}
  \mathcal{R}^N = N\mathcal{R}+ z_N (\mathcal{R}) \label{eq-def-RN}
\end{equation}
where $z_N (\mathcal{R}) \in (- 1 / 2, 1 / 2]^d$ is chosen such that the
center of $\mathcal{R}^N$ belongs to $\mathbbm{Z}^d$. Still, for any finite
collection $(\mathcal{R}_i)_{i = 1 \ldots n}$ of disjoint rectangular
parallelepipeds in $(0, 1)^d$ and large enough $N$, the collection
$(\mathcal{R}_i^N)_{i = 1 \ldots n}$ is disjoint and included in $(0, N)^d$.

\begin{proposition}
  \label{prop-lwb-disc}Consider a regular $u = \chi_U$. For any $\delta > 0$
  and any $\delta$-covering $(\mathcal{R}_i)_{i = 1 \ldots n}$ for $u$, we
  have
  \begin{equation}
    \frac{1}{N^{d - 1}} \log \Phi^{J, w}_{\Lambda_N} \left( \mathcal{D}_U^{N,
    \delta} \right) \geqslant - \sum_{i = 1}^n h_i^{d - 1}
    \tau^J_{\mathcal{R}_i^N} - c \beta \delta \label{eq-lwb-phi-decU}
  \end{equation}
  for any $N$ large enough, where $c < \infty$ depends on $d$ and $u$.
\end{proposition}

\begin{proposition}
  \label{prop-lwb-cond-phco}Assume $\beta > \hat{\beta}_c$ and $\beta \notin
  \mathcal{N}$, and let $u = \chi_U$ regular. For any $\varepsilon > 0$, for
  small enough $\delta > 0$ there are $K \in \mathbbm{N}^{\star}$ and $c > 0$
  such that, for large enough $N$:
  \begin{equation}
    \mathbbm{P} \left( \inf_{\pi \in \mathcal{D}_U^{N, \delta}} \Psi^{J, w,
    +}_{\Lambda_N} \left( \left. \frac{\mathcal{M}_K}{m_{\beta}} \in
    \mathcal{V}(u, \varepsilon) \right| \omega = \pi \text{ on }
    \mathcal{E}_U^{N, \delta} \right) \leqslant \frac{1}{2} - e^{- c \sqrt{N}}
    \right) \leqslant e^{- c \sqrt{N}} . \label{eq-condproba-McloseW}
  \end{equation}
\end{proposition}

\begin{proof}
  (Proposition \ref{prop-lwb-disc}). To realize the event of disconnection
  $\mathcal{D}_U^{N, \delta}$, it is enough to realize all the
  $\mathcal{D}_{\mathcal{R}_i^N}$ and to close all the edges that are at
  distance less than $1 + \sqrt{d}$ from
  \[ N \left[  \left( \partial U \setminus \bigcup_{i = 1}^n
     \dot{\mathcal{R}}_i \right) \cup \bigcup_{i = 1}^n \partial_{\tmop{lat}}
     \mathcal{R}_i \right] \]
  where $\partial_{\tmop{lat}} \mathcal{R}$ stands for the lateral boundary of
  $\mathcal{R}$, that is the faces of $\partial \mathcal{R}$ that are parallel
  to the orientation $\tmmathbf{n}$ of $\mathcal{R}$. Thanks to Lemma
  \ref{lem-mink-cont} and Definition \ref{def-d-adapted-u}, there are at most
  $\delta c_d N^{d - 1} \left( 1 +\mathcal{H}^{d - 1} (\partial U) \right)$
  such edges for large enough $N$. An immediate application of the DLR
  equation yields (\ref{eq-lwb-phi-decU}).
\end{proof}

\begin{figure}[h!]
  \begin{center}
  \includegraphics[width=8cm]{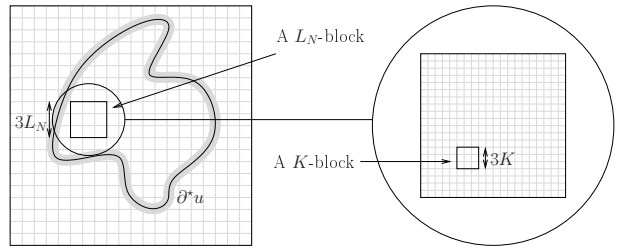}
  \end{center}
  \caption{\label{fig-binf-phco}The scales $K$ and $L_N$.}
\end{figure}

\begin{proof}
  (Proposition \ref{prop-lwb-cond-phco}). In order to obtain the claim for a
  mesoscopic scale $K$ that does not depend on $N$, we proceed to a coarse
  grained analysis at two characteristic scales $K$ and $L_N = [ \sqrt{N}]$.
  Given $K \in \mathbbm{N}^{\star}$, we consider $(\Delta_i, \Delta_i')_{i \in
  I_{\Lambda_N, K}}$ the $(K, K)$-covering of $\Lambda_N$ as in Definition 5.1
  in {\cite{M01}} as well as the phase indicator
  \[ (\phi_i)_{i \in I_{\Lambda_N, K}} \]
  given by Theorem 5.10 in {\cite{M01}}, for the tolerance $\delta$. We call
  $F = \left\{ 0, 1 \right\}^{I_{\Lambda_N, K}}$ the set of site
  configurations on the index of blocks $I_{\Lambda_N, K}$. In order to apply
  the stochastic domination Theorem 5.10 (iv) in {\cite{M01}}, we will define
  an increasing function $f : F \rightarrow \{0, 1\}$ with the appropriate
  properties. First, we need to describe the $L_N$-blocks: we call $(
  \tilde{\Delta}_j, \tilde{\Delta}'_j)_{j \in J_{N, K}}$ the $(L_N,
  L_N)$-covering for $I_{\Lambda_N, K}$ as in Definition 5.1 in {\cite{M01}}.
  Then we let
  \[ J = \left\{ j \in J_{N, K} : \forall i \in \tilde{\Delta}'_j, E^w \left(
     \Delta'_i \right) \cap \mathcal{E}^{N, \delta}_U = \emptyset \right\} \]
  and
  \[ I = \bigcup_{j \in J} \tilde{\Delta}'_j . \]
  Given $\rho \in F$ a site configuration on $I_{\Lambda_N, K}$ and $j \in J$,
  we say that the $L_N$-block $\tilde{\Delta}'_j$ is good if there is a
  crossing cluster of open sites for $\rho$ in $\tilde{\Delta}'_j$, of density
  at least $1 - \delta$. Then we define $f : F \rightarrow \{0, 1\}$ letting
  \[ f (\rho) =\tmmathbf{1}_{\left\{ \text{For all $j \in J$, }
     \tilde{\Delta}'_j \text{ is good} \right\}} . \]
  Clearly, $f$ is an increasing function. We prove now that its expectation is
  close to $1$ under high-parameter site percolation. Consider
  $\mathcal{B}_p^I$ the site percolation process on $I$ of density $p \in (0,
  1)$. According to Theorem 1.1 in {\cite{N72}}, for large enough $p < 1$
  there is $c > 0$ such that, for large enough $N$, for all $j \in J$:
  \[ \mathcal{B}_p^I \left( \left\{ \tilde{\Delta}'_j \text{ is good} \right\}
     \right) \geqslant 1 - \exp \left( - 2 c L_N^{d - 1} \right) \]
  and consequently (the cardinal of $J$ is bounded by $N^d$), for $p < 1$
  close enough to $1$, for large enough $N$,
  \[ \mathcal{B}_p^I \left( f \right) \geqslant 1 - \exp \left( - c \sqrt{N}
     \right) . \]
  Consequently, the stochastic domination for $(| \phi_i |)_{i \in
  I_{\Lambda_N, K}}$ (see Theorem 5.10 (iv) in {\cite{M01}}) yields the same
  lower bound on the expectation of $f ( (| \phi_i |)_{i \in I})$: for large
  enough $K$ (depending on $\delta$), there is $c > 0$ such that, for any $N$
  large enough:
  \begin{equation}
    \mathbb{E} \inf_{\pi} \Psi_{\Lambda_N, \beta}^{J, +}  \left( f \left(
    \left( | \phi_i | \right)_{i \in I} \right) \left| \omega = \pi \text{ on
    } E^w (\Lambda_N) \setminus \bigcup_{i \in I} E^w (\Delta'_i) \right.
    \right) \geqslant 1 - e^{- c \sqrt{N}} . \label{eq-lwb-EPsif}
  \end{equation}
  The event that $f \left( \left( | \phi_i | \right)_{i \in I} \right) = 1$
  gives a control on the magnetization. For large enough $N$, the blocks
  $(\Delta_i)_{i \in I}$ cover a fraction of $\Lambda_N$ that is close to $1
  -\mathcal{L}^d \left( \partial U + B (0, c_d \delta) \right)
  \underset{\delta \rightarrow 0^+}{\longrightarrow} 1$. This and the
  properties of $(\phi_i)_{i \in I_{\Lambda_N, K}}$ (Theorem 5.10 (i) and (ii)
  in {\cite{M01}}) imply that, for small enough $\delta > 0$, for large enough
  $N$:
  \[ f \left( \left( | \phi_i | \right)_{i \in I} \right) = 1 \Rightarrow
     \frac{\mathcal{M}_K}{m_{\beta}} \in \mathcal{V}(u, \varepsilon) \text{ or
     } \frac{\mathcal{M}_K}{m_{\beta}} \in \mathcal{V}(\tmmathbf{1},
     \varepsilon) . \]
  We now consider a boundary condition $\pi \in \mathcal{D}_U^{N, \delta}$.
  Because of the $\omega$-disconnection, the spin of the clusters touching
  some $\Delta_i \subset NU$ with $i \in I$ has a symmetric distribution under
  the conditional measure
  \[ \Psi_{\Lambda_N, \beta}^{J, +}  \left( . \left| f \left( \left( | \phi_i
     | \right)_{i \in I} \right) = 1 \text{ and } \omega = \pi \text{ on }
     \mathcal{E}^{N, \delta}_U \right. \right) . \]
  Hence, one has
  \[ \begin{array}{l}
       \inf_{\pi \in \mathcal{D}_U^{N, \delta}} \Psi_{\Lambda_N, \beta}^{J, +}
       \left( \left. \frac{\mathcal{M}_K}{m_{\beta}} \in \mathcal{V}(u,
       \varepsilon) \right| \omega = \pi \text{ on } \mathcal{E}^{N, \delta}_U
       \right)\\
       \hspace{2cm} \geqslant \frac{1}{2} \inf_{\pi \in \mathcal{D}_U^{N,
       \delta}} \Psi_{\Lambda_N, \beta}^{J, +}  \left( f \left( \left( |
       \phi_i | \right)_{i \in I} \right) \left| \omega = \pi \text{ on }
       \mathcal{E}^{N, \delta}_U \right. \right)
     \end{array} \]
  The claim follows as (\ref{eq-lwb-EPsif}) implies, as $\mathcal{E}^{N,
  \delta}_U \subset E^w (\Lambda_N) \setminus \bigcup_{i \in I} E^w
  (\Delta'_i)$, that
  \[ \mathbbm{P} \left( \inf_{\pi \in \mathcal{D}_U^{N, \delta}}
     \Psi_{\Lambda_N, \beta}^{J, +}  \left( f \left( \left( | \phi_i |
     \right)_{i \in I} \right) \left| \omega = \pi \text{ on } \mathcal{E}^{N,
     \delta}_U \right. \right) \leqslant 1 - e^{- c / 2 \sqrt{N}} \right)
     \leqslant e^{- c / 2 \sqrt{N}} . \]
\end{proof}

The final formulation of the lower bound for phase coexistence is the
following:

\begin{proposition}
  \label{prop-lwb-phco-final}Assume $\beta > \hat{\beta}_c$ and $\beta \notin
  \mathcal{N}$. For any $0 \leqslant \alpha < \text{$1 / \tmop{diam}_{\infty}
  (\mathcal{W}^q)$}$ and $\varepsilon > 0$ there exists $K \in
  \mathbbm{N}^{\star}$ such that,
  \begin{equation}
    \liminf_{N \rightarrow \infty} \frac{1}{N^{d - 1}} \log \mu^{J,
    +}_{\Lambda_N} \left( \frac{\mathcal{M}_K}{m_{\beta}} \in
    \mathcal{V}(\chi_{z_0 + \alpha \mathcal{W}^q}, \varepsilon) \right)
    \geqslant -\mathcal{F}^q (\chi_{\alpha \mathcal{W}^q}) \text{ \ \ \ \ }
    \mathbbm{P} \text{-a.s.} \label{eq-lwb-phcoex-q}
  \end{equation}
  where $z_0 = (1 / 2, \ldots 1 / 2)$. Similarly, for any $\lambda > 0$ and $0
  \leqslant \alpha < \text{$1 / \tmop{diam}_{\infty}
  (\mathcal{W}^{\lambda})$}$,
  \begin{equation}
    \liminf_{N \rightarrow \infty} \frac{1}{N^{d - 1}} \log \mathbbm{E} \left[
    \left( \mu^{J, +}_{\Lambda_N} \left( \frac{\mathcal{M}_K}{m_{\beta}} \in
    \mathcal{V}(\chi_{z_0 + \alpha \mathcal{W}^{\lambda}}, \varepsilon)
    \right) \right)^{\lambda} \right] \geqslant -\mathcal{F}^{\lambda}
    (\chi_{\alpha \mathcal{W}^{\lambda}}) \label{eq-lwb-phcoex-l} .
  \end{equation}
\end{proposition}

\begin{proof}
  Let $U = z_0 + \alpha \mathcal{W}^q$. According to Theorem 3.2.35 in
  {\cite{N304}}, $\partial U$ is rectifiable, hence the profile $u = \chi_U$
  is regular. Let $\varepsilon, \delta > 0$. Thanks to Theorem
  \ref{thm-d-cover} there exists a $\delta$-covering $(\mathcal{R}_i)_{i =
  1}^n$ adapted to the profile $\chi_U$ and $\tau^q$. Proposition
  \ref{prop-lwb-disc} applies and gives, for $\delta > 0$ small enough:
  \begin{eqnarray}
    \mu^{J, +}_{\Lambda_N} \left( \frac{\mathcal{M}_K}{m_{\beta}} \in
    \mathcal{V}(\chi_U, \varepsilon) \right) & \geqslant & \inf_{\pi \in
    \mathcal{D}_U^{N, \delta}} \Psi^{J, w, +}_{\Lambda_N} \left(
    \frac{\mathcal{M}_K}{m_{\beta}} \in \mathcal{V}(\chi_U, \varepsilon) |
    \omega = \pi \text{ on } \mathcal{E}_U^{N, \delta} \right) \nonumber\\
    &  & \times \exp \left( - N^{d - 1} \left( \sum_{i = 1}^n h_i^{d - 1}
    \tau^J_{\mathcal{R}_i^N} + c \beta \delta \right) \right) 
    \label{eq-lwb-coU}
  \end{eqnarray}
  where $c < \infty$ depends on $d$ and $u$. An important remark is that the
  two factors are \tmtextit{independent} under the product measure
  $\mathbbm{P}$. Proposition \ref{prop-lwb-cond-phco} yields:
  \begin{equation}
    \mathbbm{P} \left( \inf_{\pi \in \mathcal{D}_U^{N, \delta}} \Psi^{J, w,
    +}_{\Lambda_N} \left( \frac{\mathcal{M}_K}{m_{\beta}} \in
    \mathcal{V}(\chi_U, \varepsilon) | \omega = \pi \text{ on }
    \mathcal{E}_U^{N, \delta} \right) \leqslant \frac{1}{3} \right) \leqslant
    e^{- c \sqrt{N}} . \label{eq-P-MclU}
  \end{equation}
  We prove first (\ref{eq-lwb-phcoex-q}) and consider $\gamma, \xi > 0$. If
  $\delta > 0$ is small enough, Theorem \ref{thm-updev-tauJ} tells that the
  $\mathbbm{P}$-probability that $\tau^J_{\mathcal{R}_i^N} > \tau^q
  (\tmmathbf{n}_i) + \gamma$ for some $i \in \{1, \ldots, n\}$ decays like
  $\exp (- cN^d)$ where $c > 0$. Hence, with $\mathbbm{P}$-probability at
  least $1 - e^{- c \sqrt{N} / 3}$ we have
  \begin{eqnarray*}
    \frac{1}{N^{d - 1}} \log \mu^{J, +}_{\Lambda_N} \left(
    \frac{\mathcal{M}_K}{m_{\beta}} \in \mathcal{V}(\chi_{z_0 + \alpha
    \mathcal{W}^{\tau}}, \varepsilon) \right) & \geqslant & - \sum_{i = 1}^n
    h_i^{d - 1} (\tau^q (\tmmathbf{n}_i)^{} + \gamma) - c \beta \delta\\
    & \geqslant & -\mathcal{F}^q (\chi_{\alpha \mathcal{W}^q}) - \xi
  \end{eqnarray*}
  for small enough $\delta > 0$ and $\gamma > 0$. Borel-Cantelli Lemma ensures
  that $\mathbbm{P}$-almost surely,
  \begin{eqnarray*}
    \liminf_{N \rightarrow \infty} \frac{1}{N^{d - 1}} \log \mu^{J,
    +}_{\Lambda_N} \left( \frac{\mathcal{M}_K}{m_{\beta}} \in
    \mathcal{V}(\chi_{z_0 + \alpha \mathcal{W}^{\tau}}, \varepsilon) \right) &
    \geqslant & -\mathcal{F}^q (\chi_{\alpha \mathcal{W}^q}) - \xi
  \end{eqnarray*}
  and (\ref{eq-lwb-phcoex-q}) follows letting $\xi \rightarrow 0^+$. We
  conclude with the proof of (\ref{eq-lwb-phcoex-l}), take $\lambda > 0$ and
  denote here $U = z_0 + \alpha \mathcal{W}^{\lambda}$. Again, there exists a
  $\delta$-covering $(\mathcal{R}_i)_{i = 1}^n$ adapted to the profile
  $\chi_U$ and $\tau^{\lambda}$. For $N$ large enough, the $\mathcal{R}_i^N$
  are disjoint and hence the $\tau^J_{\mathcal{R}_i^N}$ are independent under
  $\mathbbm{P}$. Consequently, for $N$ large enough and $\lambda > 0$,
  (\ref{eq-lwb-coU}) and (\ref{eq-P-MclU}) give
  \begin{eqnarray*}
    \mathbbm{E} \left[ \left( \mu^{J, +}_{\Lambda_N} \left(
    \frac{\mathcal{M}_K}{m_{\beta}} \in \mathcal{V}(\chi_U, \varepsilon)
    \right) \right)^{\lambda} \right] & \geqslant & \frac{1}{2 \times
    3^{\lambda}} \times \prod_{i = 1}^l \mathbbm{E} \exp \left( - \lambda N^{d
    - 1} h_i^{d - 1} \tau^J_{\mathcal{R}_i^N} \right)\\
    &  & \times \exp \left( - \lambda N^{d - 1} c \beta \delta \right) .
  \end{eqnarray*}
  In view of Proposition \ref{prop-conv-taul}, this means
  \begin{eqnarray*}
    \liminf_{N \rightarrow \infty} \frac{1}{N^{d - 1}} \log \mathbbm{E} \left[
    \left( \mu^{J, +}_{\Lambda_N} \left( \frac{\mathcal{M}_K}{m_{\beta}} \in
    \mathcal{V}(\chi_U, \varepsilon) \right) \right)^{\lambda} \right] &
    \geqslant & - \sum_{i = 1}^n h_i^{d - 1} \tau^{\lambda} (\tmmathbf{n}_i) -
    \lambda c \beta \delta
  \end{eqnarray*}
  and the claim follows as $\delta \rightarrow 0$.
\end{proof}

\subsection{Upper bound for phase coexistence}

Here we address the opposite problem of providing an upper bound on the
probability of phase coexistence along a given phase profile. Our analysis
follows the same line as {\cite{N06,N07,N12}}. The cost of phase coexistence
is easily related (Proposition \ref{prop-upb-phco}) to another notion of
surface tension (\ref{eq-def-tauJL1}), that uses a $L^1$-characterization of
phase coexistence. Then the $L^1$-notion of surface tension is related to a
percolative definition of surface tension with {\tmem{free boundary}}
conditions, with the help of the minimal section argument (Proposition
\ref{prop-minsect-arg}). As in the uniform setting {\cite{N12}}, the surface
tension with free boundary condition differs very slightly from the usual
notion of surface tension (Proposition \ref{prop-comp-tauJf-tauJ}).

The $L^1$-definition of surface tension is as follows. Given $\delta > 0$, a
rectangular parallelepiped $\mathcal{R} \subset [0, 1]^d$ as in Definition
\ref{def-d-adapted-u} (i) and $K, N \in \mathbbm{N}^{\star}$ we define
\begin{equation}
  \tilde{\tau}^{J, \delta, K}_{N\mathcal{R}} = - \frac{1}{(h N)^{d - 1}} \log
  \sup_{\bar{\sigma} \in \Sigma^+_{\widehat{N\mathcal{R}}}} \mu^{J,
  \bar{\sigma}}_{\widehat{N\mathcal{R}}} \left( \left\|
  \frac{\mathcal{M}_K}{m_{\beta}} - \chi \right\|_{L^1 \left( \mathcal{R}
  \right)} \leqslant 2 \delta \mathcal{L}^d \left( \mathcal{R} \right) \right)
  \label{eq-def-tauJL1}
\end{equation}
where $\chi$ is the characteristic function of $\mathcal{R}$ as in Definition
\ref{def-d-adapted-u} (iii), and $\mu^{J,
\bar{\sigma}}_{\widehat{N\mathcal{R}}}$ the Gibbs measure on
$\widehat{N\mathcal{R}}$ with boundary condition $\bar{\sigma}$. We have:

\begin{proposition}
  \label{prop-upb-phco}Let $u \in \tmop{BV}$, $\delta > 0$ and assume that
  $(\mathcal{R}_i)_{i = 1 \ldots n}$ is a $\delta$-covering for $u$. Then, for
  any $\varepsilon > 0$ small enough, any $K, N \in \mathbbm{N}^{\star}$ one
  has:
  \begin{equation}
    \frac{1}{N^{d - 1}} \log \mu^{J, +}_{\Lambda_N} \left(
    \frac{\mathcal{M}_K}{m_{\beta}} \in \mathcal{V}(u, \varepsilon) \right)
    \leqslant - \sum_{i = 1}^n h_i^{d - 1} \tilde{\tau}^{J, \delta,
    K}_{N\mathcal{R}_i} \label{eq-upb-mag-u} .
  \end{equation}
\end{proposition}

\begin{proof}
  For $\varepsilon > 0$ small enough, the implication
  \[ \frac{\mathcal{M}_K}{m_{\beta}} \in \mathcal{V}(u, \varepsilon)
     \Rightarrow \left\| \frac{\mathcal{M}_K}{m_{\beta}} - u \right\|_{L^1
     (\mathcal{R}_i)} \leqslant \delta \mathcal{L}^d (\mathcal{R}_i) \text{, \
     \ \ } \forall i \in \{1, \ldots, n\} \]
  holds. Thanks to (iii) in Definition \ref{def-d-adapted-u}, for such
  $\varepsilon$ we have
  \[ \frac{\mathcal{M}_K}{m_{\beta}} \in \mathcal{V}(u, \varepsilon)
     \Rightarrow \left\| \frac{\mathcal{M}_K}{m_{\beta}} - \chi_i
     \right\|_{L^1 (\mathcal{R}_i)} \leqslant 2 \delta \mathcal{L}^d
     (\mathcal{R}_i) \text{, \ \ \ } \forall i \in \{1, \ldots, n\}. \]
  Now, the Gibbs property for $\mu^{J, +}_{\Lambda_N}$ implies that
  \begin{eqnarray*}
    \mu^{J, +}_{\Lambda_N} \left( \frac{\mathcal{M}_K}{m_{\beta}} \in
    \mathcal{V}(u, \varepsilon) \right) & \leqslant & \mu^{J, +}_{\Lambda_N}
    \left( \left\| \frac{\mathcal{M}_K}{m_{\beta}} - \chi_i  \right\|_{L^1
    (\mathcal{R}_i)} \leqslant 2 \delta \mathcal{L}^d (\mathcal{R}_i), \forall
    i \in \{1, \ldots, n\} \right)\\
    & = & \mu^{J, +}_{\Lambda_N} \left( \prod_{i = 1}^n \mu^{J,
    \sigma}_{\widehat{N\mathcal{R}_i}} \left( \left\|
    \frac{\mathcal{M}_K}{m_{\beta}} - \chi_i  \right\|_{L^1 (\mathcal{R}_i)}
    \leqslant 2 \delta \mathcal{L}^d (\mathcal{R}_i) \right) \right)\\
    & \leqslant & \exp \left( - h_i^{d - 1} N^{d - 1}  \tilde{\tau}^{J,
    \delta, K}_{N\mathcal{R}_i} \right)
  \end{eqnarray*}
  thanks to (\ref{eq-def-tauJL1}), and the claim is proved.
\end{proof}

Using the minimal section argument as in {\cite{N06}} one can compare the
$L^1$-surface tension to the surface tension under free boundary condition in
$\mathcal{R}=\mathcal{R}_{x, L, H} (\mathcal{S}, \tmmathbf{n})$, defined as
\begin{equation}
  \tilde{\tau}^J_{\mathcal{R}} = - \frac{1}{L^{d - 1}} \log \Phi^{J,
  f}_{\mathcal{R}} \left( \mathcal{D}_{\tilde{\mathcal{R}}} \right)
  \label{eq-def-tauJt}
\end{equation}
where $\tilde{\mathcal{R}} =\mathcal{R}_{x, L, H / 2} (\mathcal{S},
\tmmathbf{n})$ is a rectangular parallelepiped twice finer than $\mathcal{R}$.

\begin{proposition}
  \label{prop-minsect-arg}Assume $\beta > \hat{\beta}_c$ with $\beta \notin
  \mathcal{N}$. Then, there exists $c_{d, \delta} \in (0, \infty)$ with
  $\lim_{\delta \rightarrow 0} c_{d, \delta} = 0$ such that, for any
  $\mathcal{R}$ as in Definition \ref{def-d-adapted-u} (i), for any $\delta >
  0$, if $K$ is large enough then:
  \begin{equation}
    \limsup_N \frac{1}{N^d} \log \mathbbm{P} \left( \tilde{\tau}^{J, \delta,
    K}_{N\mathcal{R}} \leqslant \tilde{\tau}^J_{\mathcal{R}^N} - c_{d, \delta}
    \right) < 0.
  \end{equation}
\end{proposition}

We do not detail here the proof of Proposition \ref{prop-minsect-arg} as it is
easily adapted from {\cite{N06}}. Then, the argument of {\cite{N12}} let us
quantify the influence of the boundary condition on the value of surface
tension:

\begin{proposition}
  \label{prop-comp-tauJf-tauJ}Assume $\beta > \hat{\beta}_c$ and $\beta \notin
  \mathcal{N}$. Let $\mathcal{R}$ be a rectangular parallelepiped
  $\mathcal{R}$ as in Definition \ref{def-d-adapted-u} (i), with $\delta \in
  (0, 1)$. Then,
  \begin{equation}
    \limsup_N \frac{1}{N^d} \log \mathbbm{P} \left(
    \tilde{\tau}^J_{\mathcal{R}^N} \leqslant \tau^J_{\mathcal{R}^N} - c_d
    \delta \right) < 0
  \end{equation}
  where $c_d < \infty$ depends on $d$ only.
\end{proposition}

We cannot afford to give here the proof of Proposition
\ref{prop-comp-tauJf-tauJ} as the generalization to the random case of the
argument of {\cite{N12}} makes it far too long. However, no new ingredient
needs to be introduced with respect to the original construction {\cite{N12}},
and the interested reader can consult the PhD thesis {\cite{M00}} for a
complete development of the proofs of both Propositions \ref{prop-minsect-arg}
and \ref{prop-comp-tauJf-tauJ}.

The consequence of the three last Propositions, together with Varadhan's
Lemma, is a lower bound on the probability of phase coexistence along a given
profile under quenched and averaged measures:

\begin{proposition}
  \label{prop-upb-phco-final}For all $\beta > \hat{\beta}_c$ with $\beta
  \notin \mathcal{N}$, for every $u \in \tmop{BV}$ and $\xi, \lambda > 0$,
  there exists $\varepsilon > 0$ such that, for $K \in \mathbbm{N}^{\star}$
  large enough,
  \begin{equation}
    \limsup_N \frac{1}{N^{d - 1}} \log \mu^{J, +}_{\Lambda_N} \left(
    \frac{\mathcal{M}_K}{m_{\beta}} \in \mathcal{V}(u, \varepsilon) \right)
    \leqslant -\mathcal{F}^q (u) + \xi \label{eq-upb-phco-q}
  \end{equation}
  in $\mathbbm{P}$-probability (and $\mathbbm{P}$-almost surely if $\beta
  \notin \mathcal{N}_I$) and
  \begin{equation}
    \limsup_N \frac{1}{N^{d - 1}} \log \mathbbm{E} \left[ \mu^{J,
    +}_{\Lambda_N} \left( \frac{\mathcal{M}_K}{m_{\beta}} \in \mathcal{V}(u,
    \varepsilon) \right) \right]^{\lambda} \leqslant -\mathcal{F}^{\lambda}
    (u) + \xi .
  \end{equation}
\end{proposition}

\begin{proof}
  We fix $\delta \in (0, 1)$ and a $\delta$-covering $(\mathcal{R}_i)_{i = 1
  \ldots n}$ for $u$ as in Definition \ref{def-d-cover}. We examine first the
  quenched convergence: according to Propositions \ref{prop-minsect-arg} and
  \ref{prop-comp-tauJf-tauJ} there is $c > 0$ such that
  \begin{equation}
    \mathbbm{P} \left( \tilde{\tau}^{J, \delta, K}_{N\mathcal{R}_i} \geqslant
    \tau^J_{\mathcal{R}^N} - c_{d, \delta} - c_d \delta \right) \geqslant 1 -
    \exp (- cN^d) \text{, \ \ } \forall i = 1 \ldots n \label{eq-lwb-tauJt}
  \end{equation}
  for $K$ and $N$ large enough. On the other hand, for any $\varepsilon > 0$
  small enough Propositions \ref{prop-upb-phco} yields
  \[ \frac{1}{N^{d - 1}} \log \mu^{J, +}_{\Lambda_N} \left(
     \frac{\mathcal{M}_K}{m_{\beta}} \in \mathcal{V}(u, \varepsilon) \right)
     \leqslant - \sum_{i = 1}^n h_i^{d - 1} \tilde{\tau}^{J, \delta,
     K}_{N\mathcal{R}_i} \]
  and hence, for $K$ and $N$ large enough,
  \[ \frac{1}{N^{d - 1}} \log \mu^{J, +}_{\Lambda_N} \left(
     \frac{\mathcal{M}_K}{m_{\beta}} \in \mathcal{V}(u, \varepsilon) \right)
     \leqslant - \sum_{i = 1}^n h_i^{d - 1}  \left[ \tau^J_{\mathcal{R}^N} -
     c_{d, \delta} - c_d \delta \right] \]
  with $\mathbbm{P}$-probability greater than $1 - n \exp (- cN^d)$. This
  implies (\ref{eq-upb-phco-q}) for $\delta > 0$ small enough in view of the
  convergence $\tau^J_{\mathcal{R}^N_i} \rightarrow \tau^q (\tmmathbf{n}_i)$
  in $\mathbbm{P}$-probability (Theorem \ref{thm-conv-tauq}) or of the
  almost-sure convergence if $\beta \notin \mathcal{N}_I$ (Corollary
  \ref{cor-NI-count}). We examine now the averaged convergence: consider
  $\lambda > 0$ and again, a $\delta$-covering $(\mathcal{R}_i)_{i = 1 \ldots
  n}$ for $u$. For $K, N$ large enough and $\varepsilon > 0$ small enough we
  have
  \[ \mathbbm{E} \left( \left[ \mu^{J, +}_{\Lambda_N} \left(
     \frac{\mathcal{M}_K}{m_{\beta}} \in \mathcal{V}(u, \varepsilon) \right)
     \right]^{\lambda} \right) \hspace{6cm} \]
  \begin{eqnarray*}
    & \leqslant & \mathbbm{E} \exp \left( - \sum_{i = 1}^n \lambda (h_i N)^{d
    - 1} \tilde{\tau}^{J, \delta, K}_{N\mathcal{R}_i} \right)\\
    & \leqslant & n \exp (- cN^d) +\mathbbm{E} \exp \left( - \sum_{i = 1}^n
    \lambda (h_i N)^{d - 1} \tau^J_{\mathcal{R}^N} \right)\\
    &  & \times \exp \left( \lambda \sum_{i = 1}^n h_i^{d - 1} N^{d - 1} 
    \left( c_{d, \delta} + c_d \delta \right) \right)
  \end{eqnarray*}
  in view of (\ref{eq-lwb-tauJt}). Varadhan's Lemma (Proposition
  \ref{prop-conv-taul}) yields: for any $\varepsilon > 0$ small enough, any
  $K$ large enough,
  \[ \limsup_N \frac{1}{N^{d - 1}} \log \mathbbm{E} \left( \left[ \mu^{J,
     +}_{\Lambda_N} \left( \frac{\mathcal{M}_K}{m_{\beta}} \in \mathcal{V}(u,
     \varepsilon) \right) \right]^{\lambda} \right) \hspace{4cm} \]
  \[ \leqslant - \sum_{i = 1}^l h_i^{d - 1}  \left[ \tau^{\lambda}
     (\tmmathbf{n}_i) - c_{d, \delta} - c_d \delta \right] \]
  and the conclusion follows for $\delta > 0$ small enough.
\end{proof}

\subsection{Exponential tightness}

The last step towards the proofs of Theorems \ref{thm-cost-phco} and
\ref{thm-shape-phco} is the exponential tightness property. Note that the
compact set $\tmop{BV}_a$ was defined at (\ref{eq-def-BVa}).

\begin{proposition}
  \label{prop-exp-tight}For any $\beta > \hat{\beta}_c$ with $\beta \notin
  \mathcal{N}$, there exists $C > 0$ and for every $\delta > 0$, for any $K
  \in \mathbbm{N}^{\star}$ large enough one has
  \begin{equation}
    \limsup_N \frac{1}{N^{d - 1}} \log \mathbbm{E} \mu^{J, +}_{\Lambda_N}
    \left( \frac{\mathcal{M}_K}{m_{\beta}} \notin \mathcal{V}(\tmop{BV}_a,
    \delta)^c \right) \leqslant - Ca.
  \end{equation}
\end{proposition}

The proof of Bodineau, Ioffe and Velenik given in {\cite{N07}} applies as well
in the present case.

\subsection{Proofs of Theorems \ref{thm-cost-phco} to \ref{thm-Wulff-lambda}}

Theorems \ref{thm-cost-phco} and \ref{thm-shape-phco} are consequences of the
large deviations estimates (Propositions \ref{prop-lwb-phco-final} and
\ref{prop-upb-phco-final}) together with the exponential tightness
(Proposition \ref{prop-exp-tight}) in view of the compactness of
$\tmop{BV}_a$. The case of averaged Gibbs measures (Theorems
\ref{thm-cost-lambda}, \ref{thm-Wulff-lambda1} and \ref{thm-Wulff-lambda})
presents complete similarity with the non-random case and for this reason we
focus here only on the quenched case. Furthermore, the proof of Theorem
\ref{thm-cost-phco} is similar to that of Theorem \ref{thm-shape-phco}, which
is the reason for which we give the proof of (\ref{eq-shape-phco-q}) only.

\begin{proof}
  (First half of Theorem \ref{thm-shape-phco}). First we establish the lower
  bound
  \begin{equation}
    \liminf_N  \frac{1}{N^{d - 1}} \log \mu^{J, +}_{\Lambda_N} \left(
    \frac{m_{\Lambda_N}}{m_{\beta}} \leqslant 1 - 2 \alpha^d \right) \geqslant
    -\mathcal{F}^q (\chi_{\alpha \mathcal{W}^q}) \text{, \ \ \ } \mathbbm{P}
    \text{-almost surely} . \label{eq-lwb-lowmag}
  \end{equation}
  The proof goes as follows: for any $\alpha' > \alpha$, for small enough
  $\varepsilon > 0$ one has
  \[ \frac{\mathcal{M}_K}{m_{\beta}} \in \mathcal{V}(\chi_{z_0 + \alpha'
     \mathcal{W}^q}, \varepsilon) \Rightarrow \frac{m_{\Lambda_N}}{m_{\beta}}
     \leqslant 1 - 2 \alpha^d \]
  hence, Proposition \ref{prop-upb-phco-final} gives: for any $\alpha' >
  \alpha$,
  \[ \liminf_N  \frac{1}{N^{d - 1}} \log \mu^{J, +}_{\Lambda_N} \left(
     \frac{m_{\Lambda_N}}{m_{\beta}} \leqslant 1 - 2 \alpha^d \right)
     \geqslant -\mathcal{F}^q (\chi_{\alpha' \mathcal{W}^q}) \text{, \ \ \ }
     \mathbbm{P} \text{-almost surely} . \]
  The lower bound (\ref{eq-lwb-lowmag}) follows if we let $\alpha' \rightarrow
  \alpha$.
  
  Now we establish the following upper bound: for any $\varepsilon > 0$, there
  is $\delta > 0$ such that
  \[ \limsup_N  \frac{1}{N^{d - 1}} \log \mu^{J, +}_{\Lambda_N} \left(
     \frac{\mathcal{M}_K}{m_{\beta}} \notin \bigcup_{x \in
     \mathcal{T}^q_{\alpha}} \mathcal{V}(\chi_{x + \alpha \mathcal{W}^q},
     \varepsilon) \text{ and } \frac{m_{\Lambda_N}}{m_{\beta}} \leqslant 1 - 2
     \alpha^d \right) \]
  \begin{equation}
    \hspace{6cm} \leqslant -\mathcal{F}^q (\chi_{\alpha \mathcal{W}^q}) -
    \delta \label{eq-upb-less-Fq}
  \end{equation}
  in $\mathbbm{P}$-probability ($\mathbbm{P}$-almost surely if $\beta \notin
  \mathcal{N}_I$). To begin with, we choose $a > 0$ so large that $Ca$ in
  Proposition \ref{prop-exp-tight} is larger than $2\mathcal{F}^q
  (\chi_{\alpha \mathcal{W}^q}) + 2$. Thanks to Markov's inequality, this
  implies that, for any $\gamma > 0$, for large enough $K$,
  \begin{equation}
    \limsup_N  \frac{1}{N^{d - 1}} \log \mu^{J, +}_{\Lambda_N} \left(
    \frac{\mathcal{M}_K}{m_{\beta}} \notin \mathcal{V}(\tmop{BV}_a, \gamma)
    \right) \leqslant -\mathcal{F}^q (\chi_{\alpha \mathcal{W}^q}) - 1,
    \label{eq-aplc-exptight}
  \end{equation}
  $\mathbbm{P}$-almost surely (see (\ref{eq-def-BVa}) for the definition of
  $\tmop{BV}_a$). Consider $\eta > 0$ and let
  \[ F = \left\{ u \in \tmop{BV}_a : \int_{[0, 1]^d} u \leqslant 1 - 2
     \alpha^d + \eta \text{ \ and \ } u \notin \bigcup_{x \in
     \mathcal{T}^q_{\alpha}} \mathcal{V} \left( \chi_{x + \alpha
     \mathcal{W}^q}, \frac{\varepsilon}{2} \right) \right\} . \]
  For $\gamma > 0$ small enough, for large enough $N$ the event
  \[ \frac{\mathcal{M}_K}{m_{\beta}} \notin \bigcup_{x \in
     \mathcal{T}^q_{\alpha}} \mathcal{V}(\chi_{x + \alpha \mathcal{W}^q},
     \varepsilon) \text{ and } \frac{m_{\Lambda_N}}{m_{\beta}} \leqslant 1 - 2
     \alpha^d \]
  implies that
  \[ \frac{\mathcal{M}_K}{m_{\beta}} \notin \mathcal{V}(\tmop{BV}_a, \gamma)
     \text{ or } \frac{\mathcal{M}_K}{m_{\beta}} \in \mathcal{V}(F, \gamma) .
  \]
  The probability of the first event is under control (\ref{eq-aplc-exptight})
  for any $\gamma > 0$ (and large enough $K$), hence we focus on the
  probability of the second one. Given $\xi > 0$, applying Proposition
  \ref{prop-upb-phco-final} we obtain $\varepsilon : u \in \tmop{BV} \mapsto
  \varepsilon (u) \in (0, \xi)$ \ such that, for any $u \in \tmop{BV}$ and any
  $K$ large enough:
  \begin{equation}
    \limsup_{N \rightarrow \infty} \frac{1}{N^{d - 1}} \log \mu^{J,
    +}_{\Lambda_N} \left( \frac{\mathcal{M}_K}{m_{\beta}} \in \mathcal{V}(u,
    \varepsilon (u)) \right) \leqslant -\mathcal{F}^q (u) + \xi
    \label{eq-upb-phco-eu}
  \end{equation}
  in $\mathbbm{P}$-probability ($\mathbbm{P}$-almost surely if $\beta \notin
  \mathcal{N}_I$). The set $\tmop{BV}_a$ is compact for the $L^1$-norm, thus
  it can be covered by a finite union $\tmop{BV}_a \subset \bigcup_{i = 1}^n
  \mathcal{V}(u_i, \varepsilon (u_i))$ with $u_i \in \tmop{BV}_a$, $i = 1
  \ldots n$. Since the right-hand side term is open, for $\gamma > 0$ small
  enough we still have
  \[ \mathcal{V}(\tmop{BV}_a, \gamma) \subset \bigcup_{i = 1}^n
     \mathcal{V}(u_i, \varepsilon (u_i)) . \]
  We consider $(u_i')_{i = 1 \ldots l}$ the subsequence of the $u_i$ such that
  $\mathcal{V}(u_i, \varepsilon (u_i))$ intersects $\mathcal{V}(F, \gamma)$.
  Thanks to the inclusion
  \[ \mathcal{V}(F, \gamma) \subset \bigcup_{i = 1}^l \mathcal{V}(u_i',
     \varepsilon (u_i')) \]
  and to (\ref{eq-upb-phco-eu}), we have: for small enough $\gamma$, for large
  enough $K$:
  \[ \limsup_{N \rightarrow \infty} \frac{1}{N^{d - 1}} \log \mu^{J,
     +}_{\Lambda_N} \left( \frac{\mathcal{M}_K}{m_{\beta}} \in \mathcal{V}(F,
     \gamma) \right) \leqslant - \inf_{u \in \tmop{BV} : u \in \mathcal{V}(F,
     2 \xi)} \mathcal{F}^q (u) + \xi \]
  in $\mathbbm{P}$-probability ($\mathbbm{P}$-almost surely if $\beta \notin
  \mathcal{N}_I$). Yet, the limit as $\xi \rightarrow 0$ of the right-hand
  side is bounded from above by $- \inf_{u \in F'} \mathcal{F}^q (u)$ where
  \[ F' = \left\{ u \in \tmop{BV}_a : \int_{[0, 1]^d} u \leqslant 1 - 2
     \alpha^d + 2 \eta \text{ \ and \ } u \notin \bigcup_{x \in
     \mathcal{T}^q_{\alpha}} \mathcal{V} \left( \chi_{x + \alpha
     \mathcal{W}^q}, \frac{\varepsilon}{4} \right) \right\}, \]
  for any $\eta > 0$. Yet, $- \inf_{u \in F'} \mathcal{F}^q (u)$ is strictly
  smaller, in the limit $\eta \rightarrow 0$, than $-\mathcal{F}^q
  (\chi_{\alpha \mathcal{W}^q})$ since the solutions to the isoperimetric
  problem (\ref{eq-uBV-opt}) are excluded. Together with
  (\ref{eq-aplc-exptight}), this implies (\ref{eq-upb-less-Fq}) and the
  conclusion (\ref{eq-shape-phco-q}) follows from (\ref{eq-lwb-lowmag}) and
  (\ref{eq-upb-less-Fq}).
\end{proof}

\subsection{Localization of the Wulff crystal under averaged measures}

One consequence of the introduction of the random media is the
{\tmem{localization}} of the Wulff crystal if the volume constraint acts on
the media as well: the surface tension appears to be reduced on the contour of
the crystal. Here we give the proof of Theorem \ref{thm-phco-tauJ} after a we
state the following immediate consequence of the lower large deviations
described in Theorem \ref{thm-rate-I}:

\begin{lemma}
  \label{lem-tauhat}Let $\mathcal{R}^N =\mathcal{R}_{0, N, \delta N}
  (\mathcal{S}, \tmmathbf{n})$ and $\gamma > 0$, $\mathcal{A}= [
  \hat{\tau}^{\lambda, -} (\tmmathbf{n}) - \gamma, \hat{\tau}^{\lambda, +}
  (\tmmathbf{n}) + \gamma]$. Then,
  \[ \limsup_N \frac{1}{N^{d - 1}} \log \mathbbm{E} \left[ 1_{\left\{
     \tau^J_{\mathcal{R}^N} \in \mathcal{A}^c \right\}} \times \exp \left( -
     \lambda N^{d - 1} \tau^J_{\mathcal{R}^N} \right) \right] < \tau^{\lambda}
     (\tmmathbf{n}) . \]
\end{lemma}

\begin{proof}
  (Theorem \ref{thm-phco-tauJ}). According to Theorems \ref{thm-cost-lambda}
  and \ref{thm-Wulff-lambda}, it is enough to prove that
  \[ \limsup_{N \rightarrow \infty} \frac{1}{N^{d - 1}} \log \mathbbm{E}
     \left[ \left( \mu^{J, +}_{\Lambda_N} \left( \tau^J_{\mathcal{R}^N} \in
     \mathcal{A}^c \text{ and } \left\| \frac{\mathcal{M}_K}{m_{\beta}} -
     \chi_{z + \alpha \mathcal{W}^{\lambda}} \right\|_{L^1} \leqslant
     \varepsilon \right) \right)^{\lambda} \right] \]
  \begin{equation}
    \hspace{8cm} < -\mathcal{F}^{\lambda} (\alpha \mathcal{W}^{\lambda}) .
    \label{eq-cond-tauJ}
  \end{equation}
  In the case that the parallelepiped $\mathcal{R}$ does not intersect the
  crystal $z + \alpha \partial \mathcal{W}^{\lambda}$, for $\delta > 0$ small
  enough any $\delta$-covering $(\mathcal{R}_i)_{i = 1 \ldots n}$ for $z +
  \alpha \mathcal{W}^{\lambda}$ and $\tau^q$ does not intersect $\mathcal{R}$.
  For $\varepsilon > 0$ small enough and $K$ large enough, Propositions
  \ref{prop-upb-phco}, \ref{prop-minsect-arg} and \ref{prop-comp-tauJf-tauJ},
  the definition of the $\delta$-covering and the independence of
  $\tau^J_{\mathcal{R}}$ from the $\tilde{\tau}^{J, \delta,
  K}_{N\mathcal{R}_i}$ under the product measure $\mathbbm{P}$ imply that the
  right-hand side of (\ref{eq-cond-tauJ}) is bounded from above by
  \[ -\mathcal{F}^{\lambda} (\alpha \mathcal{W}^{\lambda}) + \underset{\delta
     \rightarrow 0}{o} (1) + \limsup_N \frac{1}{N^{d - 1}} \log \mathbbm{P}
     \left( \tau^J_{\mathcal{R}^N} \in \mathcal{A}^c \right) \]
  which is strictly smaller than $-\mathcal{F}^{\lambda} (\alpha
  \mathcal{W}^{\lambda})$ for small enough $\delta$, as the last term is
  strictly negative.
  
  Now we consider the case when the parallelepiped $\mathcal{R}$ is tangent to
  the crystal. For $h > 0$ small enough, for $\varepsilon > 0$ small enough,
  the strict inequality
  \[ \limsup_N \frac{1}{N^{d - 1}} \log \mathbbm{E} \left[ \left(
     \sup_{\bar{\sigma} \in \Sigma^+_{\widehat{N\mathcal{R}}}} \mu^{J,
     \bar{\sigma}}_{\widehat{N\mathcal{R}}} \left( \tau^J_{\mathcal{R}^N} \in
     \mathcal{A}^c \text{ \ and \ } \left\| \frac{\mathcal{M}_K}{m_{\beta}} -
     \chi_{z + \alpha \mathcal{W}^{\lambda}} \right\|_{L^1 (\mathcal{R})}
     \leqslant \varepsilon \right) \right)^{\lambda} \right] \]
  
  \begin{equation}
    \hspace{6cm} < - \int_{\partial (z + \alpha \mathcal{W}^{\lambda}) \cap
    \mathcal{R}} \tau^{\lambda} (\tmmathbf{n}_.) d\mathcal{H} \label{tauJinA}
  \end{equation}
  holds according to Propositions \ref{prop-minsect-arg} and
  \ref{prop-comp-tauJf-tauJ}, and Lemma \ref{lem-tauhat}. Let
  $(\mathcal{R}_i)_{i = 1 \ldots n}$ be a $\eta$-covering for $z + \alpha
  \mathcal{W}^{\lambda}$ and $\tau^{\lambda}$. Propositions
  \ref{prop-upb-phco}, \ref{prop-minsect-arg} and \ref{prop-comp-tauJf-tauJ}
  and the properties of the $\eta$-covering imply that for $\varepsilon > 0$
  small enough (depending on $\eta$) and large enough $K$, the cost of phase
  coexistence outside of $\mathcal{R}$ is bounded above by
  \[ \limsup_N \frac{1}{N^{d - 1}} \log \mathbbm{E} \left[ \left( \prod_{i :
     \mathcal{R}_i \cap \mathcal{R}= \emptyset} \sup_{\bar{\sigma} \in
     \Sigma^+_{\widehat{N\mathcal{R}}_i}} \mu^{J,
     \bar{\sigma}}_{\widehat{N\mathcal{R}_i}} \left( \left\|
     \frac{\mathcal{M}_K}{m_{\beta}} - \chi_{z + \alpha \mathcal{W}^{\lambda}}
     \right\|_{L^1 (\mathcal{R})} \leqslant \varepsilon \right)
     \right)^{\lambda} \right] \]
  \begin{eqnarray*}
    & \hspace{4cm} \leqslant & - \int_{\partial (z + \alpha
    \mathcal{W}^{\lambda}) \setminus \mathcal{R}} \tau^{\lambda}
    (\tmmathbf{n}_.) d\mathcal{H}+ \underset{\eta \rightarrow 0}{o} (1) .
  \end{eqnarray*}
  Thus, choosing $h > 0$ small enough then $\varepsilon > 0$ small enough and
  $K$ large enough, the strict inequality holds in (\ref{eq-cond-tauJ}) and
  the claim follows.
\end{proof}

\bibliographystyle{abbrv}
\bibliography{biblio}

\end{document}